\documentclass[11pt,reqno]{article}

\usepackage{longtable}
\usepackage{amssymb}
\usepackage{amsthm}
\usepackage{amsmath}
\usepackage{wasysym}
\usepackage{dcpic, pictex}
\usepackage{bbm}
\usepackage{enumerate}
\usepackage{amsfonts}
\usepackage{amsthm}
\usepackage{amsmath}
\usepackage{pict2e}
\usepackage{verbatim}
\usepackage{cancel}
\usepackage{lipsum}
\usepackage{bm}

\usepackage{enumerate}
\allowdisplaybreaks
\makeindex
\newtheorem{theorem}{Theorem}[section]
\newtheorem{lemma}[theorem]{Lemma}
\newtheorem{proposition}[theorem]{Proposition}
\newtheorem{definition}[theorem]{Definition}

\newtheorem{remark}[theorem]{Remark}
\newtheorem{corollary}[theorem]{Corollary}
\newtheorem{defn}[theorem]{Definition}

\usepackage[top=2.54cm,bottom=2.54cm,left=2.54cm,right=2.54cm]{geometry}
\numberwithin{equation}{section}
\allowdisplaybreaks

\usepackage{xcolor}

\newcommand{\nc}{\normalcolor}
\newcommand{\dif}{\mathrm{d}}
\newcommand{\E}{\mathbb{E}}
\newcommand{\R}{\mathbb{R}}
\newcommand{\C}{\mathbb{C}}

\newcommand{\PP}{\mathbb{P}}
\newcommand{\N}{\mathbb{N}}
\newcommand{\D}{\mathbb{D}}

\newcommand{\ii}{\mathrm{i}}

\newcommand{\del}{\partial}

\def\XX{\mathcal{X}}

\numberwithin{equation}{section}

\def\beq{\begin{equation}}
\def\eeq{\end{equation}}
\def\eps{\varepsilon}
\def\bx{\mathbf{x}}
\def\by{\mathbf{y}}
\def\Tr{\mathrm{Tr}}
\def\tilsumij{\tilde{\sum}_{ij}}

\usepackage{hyperref}

\def\bet{\begin{theorem}}
\def\eet{\end{theorem}}
\def\bel{\begin{lemma}}
\def\eel{\end{lemma}}
\def\bas{\begin{ass}}
\def\eas{\end{ass}}
\def\bec{\begin{cor}}
\def\eec{\end{cor}}
\def\bed{\begin{defn}}
\def\eed{\end{defn}}
\def\bep{\begin{proposition}}
\def\eep{\end{proposition}}
\def\beq{\begin{equation}}
\def\eeq{\end{equation}}
\def\proof{\noindent {\bf Proof.}\ \ }
\def\bea{\begin{equation*}}
\def\eea{\end{equation*}}
\def\tr{\mathrm{Tr}}
\def\bex{\begin{ex}}
\def\eex{\end{ex}}

\def\rr{\R}
\def\cc{\C}
\def\nn{\mathbb{N}}

\def\d{\mathrm{d}}
\def\i{\mathrm{i}}
\def\O{\mathcal{O}}
\def\pp{\mathbb{P}}
\def\ee{\mathbb{E}}
\def\e{\mathrm{e}}
\def\mfa{\mathfrak{a}}
\def\mfb{\mathfrak{b}}
\def\mfc{\mathfrak{c}}
\def\mft{\mathfrak{t}}
\def\mfd{\mathfrak{d}}
\def\mfe{\mathfrak{e}}
\def\GG{\mathcal{G}}

\def\sumab{\sum_{ab}}

\let\Im\relax
\DeclareMathOperator{\Im}{Im} 
\let\Re\relax
\DeclareMathOperator{\Re}{Re} 

\DeclareMathOperator{\Var}{Var}

\def\P{\mathcal{P}}
\def\CC{\mathcal{C}}
\def\MM{\mathcal{M}}
\def\EE{\mathcal{E}}
\def\FF{\mathcal{F}}
\def\pp{\mathbb{P}}
\def\tilf{\tilde{f}}
\def\1{\boldsymbol{1}}
\def\ZZ{\mathcal{Z}}

\title{Maximum of the Characteristic Polynomial of I.I.D. Matrices}

\begin{document}

%\maketitle

%\begin{abstract}

%We compute the leading order asymptotic of the maximum of the characteristic polynomial for i.i.d. matrices with real or complex entries. In particular, this result is new even for real Ginibre matrices, which was left as an open problem in \cite{lambert2024law}; the complex Ginibre case was covered in \cite{lambert2020maximum}. These are the first universality results for the non--Hermitian analog of the first order term of the Fyodorov--Hiary--Keating conjecture. Our methods are based on constructing a coupling to the branching random walk via Dyson Brownian motion. In particular, we find a new connection between real i.i.d. matrices and inhomogeneous branching random walk.

%\end{abstract}

\begin{table}
\centering

\begin{tabular}{c}

\multicolumn{1}{c}{\parbox{12cm}{\begin{center}\Large{\bf Maximum of the Characteristic Polynomial of I.I.D. Matrices }\end{center}}}\\
\\
\end{tabular}
\begin{tabular}{ c c c  }
Giorgio Cipolloni
& \phantom{blah} & %\qquad
Benjamin Landon
 \\
 & & \\  
 \small{Princeton University} & & \small{ University of Toronto } \\
 \small{Mathematics and Physics Departments} & & \small{Department of Mathematics} \\
 \small{\texttt{gc4233@princeton.edu}} & & \small{\texttt{blandon@math.toronto.edu}} \\
  & & \\
\end{tabular}
\\
\begin{tabular}{c}
\multicolumn{1}{c}{\today}\\
\\
\end{tabular}

\begin{tabular}{p{15 cm}}
\small{{\bf Abstract:}  We compute the leading order asymptotic of the maximum of the characteristic polynomial for i.i.d. matrices with real or complex entries. In particular, this result is new even for real Ginibre matrices, which was left as an open problem in \cite{lambert2024law}; the complex Ginibre case was covered in \cite{lambert2020maximum}. These are the first universality results for the non--Hermitian analog of the first order term of the Fyodorov--Hiary--Keating conjecture. Our methods are based on constructing a coupling to the branching random walk via Dyson Brownian motion. In particular, we find a new connection between real i.i.d. matrices and inhomogeneous branching random walk.}
\end{tabular}
\end{table}

\tableofcontents

\section{Introduction}

Let $X$ be an $n \times n$ matrix of i.i.d. centered real or complex random variables, scaled  so that $\ee |X_{ij} |^2 = \frac{1}{n}$. It is natural to investigate the size of the fluctuations of the characteristic polynomial,
\beq
P_n (z) := \log | \det (X -z ) | ,
\eeq
where the logarithmic scaling turns out to best capture the interesting behavior of $P_n (z)$. 
A remarkable property of the field $\{ P_n (z) \}_z$ is that it is expected to asymptotically exhibit Gaussian fluctuations with covariance structure,\footnote{There is an important additional term in the case that $X$ is a matrix of real random variables which we will comment on later but neglect for now for the purposes of exposition.}
\beq \label{eqn:cov-struct}
\mathrm{Cov} (P_n (z ), P_n (w) ) \approx - \frac{1}{4} \log \big( |z-w|^2 + n^{-1} \big).
\eeq
The field $ \{ P_n (z) \}_z$ is ill-behaved as a function of $z$, as indicated by the fact that variance is diverging when $z=w$. In the case that $X$ is drawn from the complex Ginibre ensemble (in which the entries of $X$ are standard complex Gaussians), Rider and Virag showed that $\{ P_n (z) \}_z$ converges in distribution to a generalized function known as the Gaussian free field \cite{rider2007noise}. This was extended to more general classes of normal matrices in \cite{ameur2011random, ameur2015random} and i.i.d. matrices in \cite{cipolloni2021fluctuation,cipolloni2023central}, as well as to certain space--time correlations \cite{cipolloni2024}.

In fact, the covariance structure \eqref{eqn:cov-struct} indicates that $P_n(z)$ is an example of a \emph{logarithmically correlated field}, objects which are ubiquitous in probability theory and statistical mechanics. They emerge from any model where randomness contributes equally at all length scales. Two of the most prominent examples are the two-dimensional Gaussian free field and branching random walk.

The central quantity of interest in our work will be the \emph{maximum} of the field $\{P_n(z)\}_z$ in the argument $z$ (after a re-centering of $P_n (z)$ by its large $n$ limit). The study of the extreme values of logarithmically correlated fields has received significant attention in recent years, especially within the context of random matrix theory. Of course, the extrema of stochastic processes is a classical subject in probability theory dating to the early twentieth century. On the other hand, the classical theory does not apply to stochastic processes exhibiting strong correlations, and  logarithmically correlated fields are a natural candidate for the development of a new extreme value theory.

Investigation of the extreme values of random matrix characteristic polynomials is motivated in part by the seminal work of Fyodorov, Hiary and Keating who conjectured that the extremal statistics of the Riemann zeta function are well modeled by random matrix eigenvalues \cite{fyodorov2012freezing}. There has since been tremendous progress on both the number theoretic and random matrix sides of the problem. The breakthrough works of Arguin, Bourgade and Radziwi{\l}{\l} verified the conjectured asymptotics for the local max of the Riemann zeta function (at a random point high on its critical axis), by showing that the maximum is tight after accounting for leading and subleading deterministic contributions \cite{arguin2020fyodorov,arguin2023fyodorov}. Studies on the random matrix side have focused on the Circular Unitary Ensemble and its generalization, the Circular $\beta$-Ensemble (C$\beta$E), with the work of Paquette and Zeitouni \cite{paquette2022extremal} proving convergence in distribution of the centered maximum of the log-characteristic polynomial. Prior contributions in this direction include \cite{arguin2017maximum,paquette2018maximum,chhaibi2018maximum}. These developments and related literature will be discussed in greater detail in  Section~\ref{sec:rmlog} below. 

However, in the context of general i.i.d. matrices, progress has so far been sparse. In particular, the eigenvalues of the C$\beta$E are one-dimensional, being distributed on the unit circle, whereas the eigenvalues of i.i.d. matrices asymptotically fill the unit disc. The work \cite{lambert2020maximum} of Lambert finds the first-order asymptotics for the maximum of the log-characteristic polynomial of the complex Ginibre ensemble. The other paper on $2$-d random matrix models is that of Lambert, Lebl\'e and Zeitouni \cite{lambert2024law} which extends this result to the $2$-d Gaussian $\beta$-ensemble, an explicit measure on $\cc^n$ (see \eqref{eq:gibbsmeas} below) generalizing the complex Ginibre ensemble to other inverse temperatures (and extended to more general $2$-d $\beta$-ensembles in the recent work of Peilen \cite{peilen2024maximum}). Other than the complex Ginibre ensemble, these models are unrelated to the general non-invariant ensembles that we consider, whose eigenvalue distributions do not have explicit forms.  Significantly, the only work on the maximum of the characteristic polynomial of non-invariant random matrices is that of Bourgade-Lopatto-Zeitouni which studies general Hermitian random matrices (the Wigner ensembles) \cite{bourgade2023optimal}  (see Section~\ref{sec:rmlog} for further detailed discussion). Wigner matrices are ultimately a  simpler object than the ensembles we consider here, owing to the well developed universality theory available in the Hermitian case.

The methods developed in these works, while powerful, nonetheless do not apply to the general, non-invariant and non-Hermitian random matrix ensembles that we study, and so new approaches are required.  Our work provides the first treatment of the characteristic polynomial for general i.i.d. matrices by proving that for any real or complex i.i.d. matrix and any $0 < r < 1$, 
\beq \label{eqn:intro-result}
\max_{ |z| \leq r} \big[ P_n (z) - E_n (z) \big] = \frac{\log n}{ \sqrt{2}} (1 + o (1) )
\eeq
with probability tending to $1$ as $n \to \infty$. Here, $ n^{-1} E_n(z)$ is the a.s. limit of $n^{-1} P_n (z)$. 

Of particular interest is that we can also handle real i.i.d. matrices; indeed, before our work, the result \eqref{eqn:intro-result} was not available even for the real Ginibre ensemble, despite the fact that this ensemble also enjoys explicit formulas for its eigenvalue density. This is partly due to the fact that the real i.i.d. case exhibits a much richer structure than the complex case. In fact, for real i.i.d. matrices \eqref{eqn:cov-struct} is not exactly correct. There is a second term on the RHS of the form $\frac{1}{4}\log ( |z-\bar{w}|^2 + n^{-1})$ reflecting the symmetry of the eigenvalues about the real axis. Of course, away from the real axis this term is subleading and so one expects the same behavior as in the complex case; however, if $z$ approaches the real axis as $n \to \infty$, say $\Im[z] = n^{-\alpha}$, then this term matters. Moreover, $E_n(z)$ in the real case has an additional correction of order $\log n$, the same order as the maximum. Due to this, it is not a priori clear what occurs near the real axis - whether the maximum is the same or not. In order to explain what occurs, we briefly discuss our methods.

Many tools have been developed to study the extremal values of logarthimcally correlated fields, such as the second moment method, barriers, and convergence to Gaussian Multiplicative chaos; see, e.g., \cite{arguin2016extrema,biskup2020extrema} for reviews. Works  \cite{lambert2020maximum,lambert2024law} rely on computations of joint Laplace transforms of $P_n (z)$ in order to carry out union bounds, as well as prove convergence to the Gaussian multiplicative chaos. However, such computations are not available in the general models we consider here.

Instead, we use Dyson Brownian motion (DBM) to exhibit a coupling of the characteristic polynomial of our random matrix models to the branching random walk, one of the central objects of the universality class of logarithmically correlated fields. This relies on recent advances in the understanding of multi-resolvent local laws \cite{cipolloni2023optimal, cipolloni2023mesoscopic} in order to compute the joint distribution of the evolution of the characteristic polynomial at different $z$ under the DBM.  

The salient feature in the  real i.i.d. case close to the real axis is that the branching random walk (BRW) is \emph{inhomogeneous}. That is, the effective rate of branching as well as the step size changes abruptly at a macroscopic time part of the way through the walk. To our knowledge, this is the first appearance of an inhomogeneous branching random walk in random matrix theory. Apart from the change in the branching structure this is almost precisely the BRW considered in \cite{fang2012branching} (this aspect of our methods is discussed in further detail in  Section~\ref{sec:lbrep} below).

In the case of homogeneous branching random walk, the leading order of the, say, $n$ final values is the same as if the $n$ walkers were independent and there was no correlation structure. The logarithmic correlation only shows up at subleading order. However, in the case of inhomogeneous branching random walk, there is a distinction, as discovered in \cite{fang2012branching}. If the initial step size is smaller, then the leading order does coincide with the independent case. However, if the initial step size is larger, then the leading order is different and is strictly smaller.

Remarkably, our inhomogeneous branching random walk is an example of the latter; the initial step size is larger and so the leading order does not coincide with the ansatz of the random variables being independent. In fact, if one tries to do a naive union bound for $\Im[z] \approx n^{-\alpha}$, modelling it as the maximum of $n^{1-\alpha}$ Gaussians with variance $\frac{1+2 \alpha}{4} \log n$ (i.e., one Gaussian per radius $n^{-1/2}$ disc, the scale on which \eqref{eqn:cov-struct} decorrelates), then one will arrive at the incorrect answer for \eqref{eqn:intro-result}.\footnote{That is, $ \sqrt{  \frac{(1- \alpha)(1+2\alpha)}{2}} \log n \gg \frac{1}{ \sqrt{2}} \log n$ for $ \alpha \in (0, \frac{1}{2})$.}

By implementing our strategy of coupling the real i.i.d. case to an inhomogeneous branching random walk, we will in fact prove the stronger result,
\beq
\label{eq:diffrealcase}
\max_{ \Im[z] \approx n^{-\alpha} } \big[P_n (z) -E_n (z)\big] = \frac{ \log n}{ \sqrt{2}} ( 1 + o ( 1) ),
\eeq
with probability tending to $1$ as $n \to \infty$ and any $0 < \alpha < \frac{1}{2}$. It appears to be a coincidence that the parameters in our inhomogeneous BRW are set up so that the RHS above is independent of $\alpha$. On the contrary, the $\log n$ correction to $E_n(z)$ in the real case depends on $\alpha$ (see \eqref{eq:centerintrealcase} below).  Moreover, \eqref{eq:diffrealcase} distinguishes the real i.i.d. case from the complex one; in the latter, the above max would be smaller than the max over the entire unit disc (it would be $\sqrt{\frac{1-\alpha}{2}} \log n$), whereas in the real case they are the same.

We turn now to a more detailed discussion of the literature and our methods before stating our main results in Section \ref{sec:main}.

\subsection{Logarithmic correlated fields in random matrices}
\label{sec:rmlog}

We now review in greater detail the literature on logarithmically correlated fields. Aside from the two-dimensional GFF and the BRW, the celebrated work \cite{fyodorov2012freezing} of Fyodorov, Hiary and Keating (FHK) uncovered yet another instance in which logarithmically correlated fields naturally appear. They conjectured that the extremal statistics of characteristic polynomials of Hermitian random matrices and of the Riemann zeta function on the critical line are identical, and coincide with those of logarithmically correlated fields. More precisely, let $U_n$ be an $n\times n$ Haar--distributed unitary matrix, then the FHK conjecture states
\begin{equation}
\label{eq:FHK}
\max_{|z|=1}\log\big|\mathrm{det}(z-U_n)\big|=\log n-\frac{3}{4}\log \log n+X_n,
\end{equation}
with $X_n$ being an order one random variable that converges, as $n\to\infty$, to the sum of two independent Gumbel random variables.

The last decade has seen enormous progress towards the proof of \eqref{eq:FHK}, both for the Riemann zeta function and for unitary random matrices. Initial progress on the number theoretic side appeared in \cite{arguin2019maximum, harper2019partition, najnudel2018extreme}. The contribution of the works of Arguin, Bourgade, and Radziwill \cite{arguin2020fyodorov, arguin2023fyodorov} is to show that  \eqref{eq:FHK} holds for the Riemann zeta function up to tightness, i.e., that $X_n$ is a tight random variable when the LHS is replaced by the local max near a random point high up on the critical axis of the Riemann zeta function. Remarkably, they were also able to compute the (lower and upper) tail behavior of the random variable $X_n$, finding estimates in agreement with the predictions of \cite{fyodorov2012freezing}.  For further references we refer the interested reader to the survey \cite{harper2019riemann} (see also \cite{bailey2022maxima}). On the random matrix side, there have been a series of works proving \eqref{eq:FHK} term by term. The leading and second order terms were computed in \cite{arguin2017maximum} and \cite{paquette2018maximum}, respectively. Then, \eqref{eq:FHK} was proven up to tightness in \cite{chhaibi2018maximum}, even for the more general class of circular $\beta$--ensembles (C$\beta$E). In this case \eqref{eq:FHK} holds after rescaling its left--hand side by $\sqrt{\beta/2}$ (the Haar unitary case corresponds to $\beta=2$). Very recently, this progression of works culminated in the work of Paquette and Zeitouni \cite{paquette2022extremal}, where they proved the convergence in distribution of $X_n$. 

Progress for Hermitian random matrix ensembles has occurred only recently (see \cite{fyodorov2016distribution} for various predictions). In \cite{bourgade2023optimal} Bourgade, Lopatto, and Zeitouni study a similar question to \eqref{eq:FHK}, but for Wigner matrices and $\beta$--ensembles instead of C$\beta$E. More precisely, let $\lambda_i$ denote the eigenvalues of a Wigner matrix or the particles of a $\beta$--ensemble, then \cite{bourgade2023optimal} for any $\beta>0$ shows
\begin{equation}
\label{eq:wignermax}
\max_{E\in \mathrm{bulk}}\left(\log\left|\prod_{i=i}^n(\lambda_i-E)\right|-\E\left[\log\left|\prod_{i=i}^n(\lambda_i-E)\right|\right]\right)=\sqrt{\frac{2}{\beta}}\log n\big(1+o(1)\big).
\end{equation}
The case $\beta=2$ of \eqref{eq:wignermax} was already proven in \cite{claeys2021much, lambert2019law}. Additionally, \cite{bourgade2023optimal} proves that for some Wigner matrices there is universality of the left--hand side of \eqref{eq:wignermax} up to tightness. This means that if the analog of \eqref{eq:FHK} is proven for the GUE/GOE ensembles up to tightness, then \cite{bourgade2023optimal} implies the same result for some more general classes of Wigner matrices (i.e. with entries not necessarily Gaussian). We also mention that in \cite{bourgade2023optimal} the authors prove optimal rigidity estimates (with Gaussian tail) for the eigenvalues of such matrices.

Much less is known in the two dimensional case. The exact distribution of the maximum $X_n$ in \eqref{eq:FHK} was first conjectured in \cite{fyodorov2012freezing}, and then  identified as the sum of two independent Gumbel random variables in \cite{kundu2013}. 
%{\color{blue} In \eqref{eq:FHK} the exact distribution of the maximum $X_n$ for unitary matrices is conjectured based on \cite{chhaibi2019circle, remy2020fyodorov}.}
However, at the moment, there is no conjecture about the analog of $X_n$ for any $2$--d ensemble. The only known results are the leading order asymptotic for two dimensional Coulomb gases. These gases are comprised of $n$ interacting particles ${\bm x}_n=(x_1,\dots, x_n)\in (\R^2)^n$ distributed according to the Gibbs measure
\begin{equation}
\label{eq:gibbsmeas}
\dif \mathbb{P}_{n,\beta}=\frac{1}{Z_{n,\beta}}e^{-\beta H_n({\bm x}_n)}\, \dif {\bm x}_n.
\end{equation}
Here $Z_{n,\beta}$ is a normalization constant, and
\begin{equation}
H_n({\bm x}_n):=-\frac{1}{2}\sum_{1\le i\ne j\le n}\log|x_i-x_j|+n\sum_{i=1}^n V(x_i),
\end{equation}
for some potential $V$ growing sufficiently fast at infinity. For these models it is known that for any $\beta >0$ that % for any $\beta>0$ it holds
\begin{equation}
\label{eq:2Dresult}
\max_{z\in D_r}\left(\log\left|\prod_{i=i}^n(x_i-z)\right|-\E\left[ \log\left|\prod_{i=i}^n(x_i-z)\right|\right]\right)=\frac{1}{\sqrt{\beta}}\log n\big(1+o(1)\big),
\end{equation}
with $D_r$ being a disk of radius $r$ contained in the bulk of the limiting empirical measure %(as $n\to \infty$) 
of $\mathbb{P}_{n,\beta}$. The Gaussian case $V(x)=|x|^2/2$ was proven in \cite{lambert2024law}, and this result was recently extended to a more general class of potentials in \cite{peilen2024maximum}. Prior to these two results, only the case $\beta=2$ and $V(x)=|x|^2/2$ was known  \cite{lambert2020maximum}, as a consequence of the fact that the Gibbs measure \eqref{eq:gibbsmeas} coincides with the eigenvalue density of the complex Ginibre ensemble. Unlike in the one dimensional case, the case $\beta=1$ in \eqref{eq:gibbsmeas} does not correspond to the real Ginibre ensemble; in fact, the spectrum of real Ginibre matrices is symmetric with respect to the real axis, unlike \eqref{eq:gibbsmeas}. Hence, \eqref{eqn:intro-result} for real Ginibre matrices does not follow from the case $\beta=1$ of \eqref{eq:2Dresult}.

We now comment on the relation between \eqref{eq:2Dresult} and our result \eqref{eqn:intro-result}. Similarly to the one dimensional case \eqref{eq:wignermax}, for two dimensional Coulomb gases, \eqref{eq:2Dresult} depends on the values of $\beta$. In contrast, quite surprisingly, the leading order asymptotic of \eqref{eqn:intro-result} is exactly the same for both real and complex matrices $X$. At first, one may think that this is a consequence of the fact that the local statistics of the eigenvalues of real Ginibre away from the real axis are (asymptotically) the same as those of the complex Ginibre. However, as indicated above, the underlying reason for the same asymptotics is more subtle.  In fact, in \eqref{eq:diffrealcase} (and in Theorem~\ref{theo:realmaxlog} below) we clearly see the difference between the log--correlation structure in the complex and the real Ginibre ensemble: the maximum over any mesoscopic band $\Im z\asymp n^{-\alpha}$, with $\alpha\in (0,1/2)$, of the left--hand side of \eqref{eqn:intro-result} is given by $\log n$ in the real case, while it depends on $\alpha$ in the complex case. This shows that even though the local statistics of the eigenvalues of real and complex $X$ are the same for $|\Im z|\gg n^{-1/2}$, their contribution to the extremal values of this $\log$--correlated field is different, showing the cumulative effect of the randomness at each scale.

We conclude this section by mentioning that recently there has been great interest and progress in studying many other aspects of the connection between extremal value theory and spectral statistics of random matrices. See, e.g., \cite{cipolloni2023universality} for a recent example not falling in the log-correlated universality class.

\subsubsection{Emergence of log--correlated fields in other models}

Asymptotics of the form \eqref{eq:FHK} are very well understood for Gaussian logarithmically correlated fields. In fact, in this case it is known that the fluctuation of $X_N$ is always given by the sum of two independent random variables, one which is universally Gumbel distributed and one which depends on the long--range behavior of the covariance (i.e., it is model dependent). We refer the interested reader to \cite{biskup2020extrema, bolthausen2001entropic, ding2017convergence, zeitouni2016branching}, and references therein. Extending this theory beyond the Gaussian case has been a major challenge, which has recently attracted lots of activity. Beyond the models discussed in Section~\ref{sec:rmlog}, we will now briefly mention other models that fall in the universality class of logarithmically correlated fields. Some examples are: the sine--Gordon model \cite{bauerschmidt2022maximum} (which can be coupled directly with the GFF), the cover time for planar random walks \cite{belius2017subleading, belius2020tightness, dembo2004cover} (where tightness is known), the maximum of Ginzburg--Landau fields where very recently tightness was proven \cite{schweiger2024tightness} (see also \cite{belius2020maximum} for a previous result about the leading order asymptotics).  We also mention the maximum of permutation matrices \cite{cook2020maximum} (where only the leading order is known), and the model of two dimensional polymers \cite{caravenna2020two, cosco2023moments, cosco2023momentsII} (where not even the leading order asymptotic is known).

\subsection{Methods}

Girko's Hermitization formula is one of the most important backbones in the study of non-Hermitian spectral statistics. In particular, it enables one to express statistics of non--Hermitian eigenvalues in terms of joint statistics of a certain family of Hermitian matrices. More precisely, for $z\in\C$ we define the family of Hermitian matrices
\begin{equation}
H^z:=\left(\begin{matrix}
0 & X-z \\
(X-z)^* & 0
\end{matrix}\right).
\end{equation}
The $2n\times 2n$ matrix $H^z$ is called the \emph{Hermitization} of $X-z$. The spectrum of $H^z$ is symmetric with respect to zero, and its positive eigenvalues $\{\lambda_i^z\}_{i=1}^n$ coincide with the singular values of $X-z$. Then, Girko's formula states
\begin{equation}
\label{eq:logcharpol}
\log\big|\mathrm{det}(X-z)\big|=\frac{1}{2}\log\big|\mathrm{det} H^z\big|.
\end{equation}

After reducing the study of non--Hermitian characteristic polynomials to the Hermitized ones \eqref{eq:logcharpol}, the proof of \eqref{eqn:intro-result} consists of two main parts, an upper bound and a lower bound. In both cases we first show that the maximum over $|z|<1$ can be expressed as the maximum over a mesh $\P$ of $\asymp n^{-1}$ equidistant points on the unit disk. This follows by the Lipschitz continuity of the logarithm of the characteristic polynomial, which is a consequence of  recent \emph{multi--resolvent} local laws from \cite{cipolloni2023optimal}.  The key observation in \cite{cipolloni2023optimal} is that the fluctuations of the resolvent of $H^z$ are much smaller when tested against certain observable matrices, an effect  first observed in the context of Wigner matrices in \cite{cipolloni2021eigenstate}. 

A common difficulty in the analysis of the maximum of \eqref{eq:logcharpol} is that this quantity can be very singular if $H^z$ has eigenvalues close to zero. We thus regularize
\begin{equation}
\label{eq:reg}
\log\big|\mathrm{det} H^z\big|=\sum_{i=1}^n\log\lambda_i^z\approx \frac{1}{2}\sum_{i=1}^n\log\big[(\lambda_i^z)^2+\eta^2\big], \qquad \eta\asymp n^{-1}.
\end{equation}
While this regularization can easily be achieved in the proof of the upper bound (due to simple monotonicity), for the lower bound we need to ensure that for each $z\in\P$ the corresponding smallest singular value $\lambda_1^z$ is not too small; this is a well known difficult problem in the analysis of the spectrum of non--Hermitian matrices, even for a single fixed $z$. We achieve this using the smallest singular value estimate \cite{cipolloni2020optimal} (see also \cite{cipolloni2022condition, erdHos2023wegner, shcherbina2022least}) together with the asymptotic independence result from \cite[Section 7]{cipolloni2023central} for $\lambda_1^{z_1}, \lambda_1^{z_2}$, as long as $|z_1-z_2|\gg n^{-1/2}$. Essentially, the asymptotic independence allows us to find a sufficiently large (random) collection of points $\{ w_i \}_i$ where $\lambda_1^{w_i}$ is not too small; this allows us to achieve some amount of regularization and then in turn apply the Lipschitz continuity mentioned above to return to a \emph{deterministic} collection of points to which we will apply the second moment method.

We now explain the main differences and the common features of the proofs of the upper and lower bounds in \eqref{eqn:intro-result}. Our main tool, used in both parts of the proof, is  a new branching random walk representation for $\log|\mathrm{det}(X-z)|$ which we discuss in Section~\ref{sec:lbrep} below. One of the main advantages of this representation is that, to prove universality of \eqref{eqn:intro-result}, we do not need to compare $\log|\mathrm{det}(X-z)|$ with its Ginibre counterpart, but we can instead estimate it directly even for general i.i.d. matrices. This also allows us to prove \eqref{eqn:intro-result} in the real case; here,  a direct comparison to the real Ginibre ensemble would not work, as \eqref{eqn:intro-result} was not known for the real Ginibre ensemble prior to our work (partly as a consequence of the very complicated $k$--point correlation functions \cite{borodin2009ginibre,forrester2007eigenvalue}). Our methods are likely to be useful in further studies of the extremal statistics of the log--characteristic polynomial, such as determining lower order terms. Investigating subleading terms in the real case would be of particular interest, as they have a different character in the inhomogeneous and homogenous BRW; see \cite{fang2012branching}.

\subsubsection{Branching random walk structure and lower bound}
\label{sec:lbrep}

In this section we explain our coupling to the BRW and our proof of the lower bound. Traditionally, proving a lower bound for the leading order asymptotic of characteristic polynomials is closely related to proving the convergence of powers of the characteristic polynomials to the Gaussiam Multiplicative Chaos (GMC) measure (see e.g. \cite{bourgade2023optimal, claeys2021much, lambert2020maximum, lambert2024law}). We refer the reader to \cite{bourgade2206liouville, lambert2021mesoscopic, nikula2020multiplicative, webb2015characteristic, remy2020fyodorov} for  other works concerning the emergence of the GMC measure from spectral statistics of random matrix ensembles. Instead, we take a completely different route, and  extract the underlying branching structure using \emph{Dyson Brownian motion (DBM)}. The branching structure of $\log|\mathrm{det}(X-z)|$ is different in the complex and in the real case. In fact, one can think of $\log|\mathrm{det}(X-z)|$ as a Gaussian field on the disk with correlation kernel (here we omit a $1/n$ regularization) 
\begin{equation}
\label{eq:corrker}
K(z_1,z_2)=\begin{cases}
-\frac{1}{4}\log|z_1-z_2|^2 &\mathrm{if}\quad X\in \C^{n\times n}, \\
-\frac{1}{4}\log|z_1-z_2|^2-\frac{1}{4}\log|z_1-\overline{z_2}|^2 &\mathrm{if}\quad X\in \R^{n\times n}. \\
\end{cases}  
\end{equation}
This shows that heuristically in the real case $\log|\mathrm{det}(X-z)|$ can be thought as the cumulative effect of two different fields: one that has the same singularity as in the complex case, and one that is smoother on the disc. The fact that $\log|\mathrm{det}(X-z)|$ has a different behavior in the real and complex cases naturally emerges by the following branching random walk representation using DBM. %, which we now discuss.

We embed $X$ in a flow
\begin{equation}
\label{eq:flowintro}
\dif X_t=-\frac{1}{2}X_t\dif t+\frac{\dif B_t}{\sqrt{N}}, \qquad\quad X_0=X,
\end{equation}
which will asymptotically not affect the distribution of the maximum, due to moment matching arguments based on \cite{landon2020comparison}. Here $B_t$ is a  matrix valued standard i.i.d. Brownian motion.

Fix a final time $T = o (1)$. Associated with this flow are \emph{characteristics} $\eta_t \approx n^{-1} + (T-t) $.  Eigenvalue statistics such as $\log | \det (X_t - z)|$, evaluated along characteristics  obey particularly nice equations under the stochastic dynamics \eqref{eq:flowintro} and so we may approximate,
\begin{align}
\label{eq:brrw1}
    \log | \det (X_T - z) | \approx \frac{1}{2}  \log \det \left[ |X_T -z |^2 + \eta_T^2\right] \approx  \frac{1}{2} \int_0^T \d \left(  \log \det \left[ |X_t -z |^2 + \eta_t^2\right]  \right)
\end{align}
with the integral in the It\^{o} sense. The last approximation on the RHS is our BRW representation for the log-characteristic polynomial (we omit the other endpoint of the integral at $t=0$ for brevity as it plays a less important role). We want to think of the RHS of \eqref{eq:brrw1} as the different branches of a BRW for different $z$. Indeed, this can be seen as a BRW  due to the fact that the increments $ \d \left(  \log \det \left[ |X_T -z |^2 + \eta_T^2\right]  \right) $  are almost perfectly correlated for $|z_1 - z_2|^2 \ll \eta_t$ and decorrelated for $|z_1 - z_2 |^2 \gg \eta_t$.  Moreover, the fact that the BRW is inhomogeneous in the real case is immediate: in the real case the quadratic variation process of $\d \left[ \log \det \left( |X_t -z |^2 + \eta_t^2\right)  \right]$ is roughly
\beq \label{eqn:intro-qv}
 \frac{ \sigma_1}{ \eta_t} \1_{ \{ \eta_t > \Im[z]^2 \}} +   \frac{ \sigma_2}{ \eta_t} \1_{ \{ \eta_t < \Im[z]^2 \}},
\eeq
for some $\sigma_1 > \sigma_2$ (note that the logarithmic behavior comes from $\int_0^T \eta_t^{-1} \d t \approx \log n$ up to appropriate constants).  In particular, the step size of our random walk decreases abruptly at the time $t$ such that $\eta_t  = \Im[z]^2$.  After an exponential time change, this corresponds roughly to the model of an inhomogeneous BRW considered by Fang and Zeitouni in \cite{fang2012branching}, where the step size changes at a macroscopic time  part of the way along the walk. 

Beyond the work of \cite{fang2012branching} which partially inspires our analysis, there has been significant recent attention given to inhomogeneous BRWs. A sampling of these works includes \cite{arguin2024maxima,arguin2015extremes,berestycki2022simple,bovier2014extremal,fels2019extremes,mallein2019maximal,ouimet2017geometry}. In particular, many of these works study the $2$d discrete GFF with scale-dependent variance, and the work \cite{arguin2024maxima} finds an inhomogeneous structure in a randomized model of the Riemann zeta function. An attractive feature of our paper is then showing that real i.i.d. matrices are another setting in which inhomogeneous models arise naturally. %{\color{red} 
In particular, this is the first instance of an inhomogeneous BRW emerging in the context of random matrix theory.

The analysis of the covariation process of the Martingale increments on the RHS of \eqref{eq:brrw1} (which then leads to the important decorrelation/correlation dichotomy depending on the distance $|z_1-z_2|$ as well as the inhomogeneous structure \eqref{eqn:intro-qv}) is based on state-of-the-art multi--resolvent local laws for i.i.d. matrices \cite{cipolloni2023mesoscopic}. More precisely, we use that the size of the product of the resolvents of $H^{z_1}$ and $H^{z_2}$ gets smaller as $|z_1-z_2|$ becomes larger, an  effect  first observed in \cite{cipolloni2023central}. For matrices with complex entries this is proven in \cite[Theorem 3.3]{cipolloni2023mesoscopic} using the method of characteristics. Here, following the lines of \cite{cipolloni2023mesoscopic}, we extend this result to matrices with real entries (with a large Gaussian component) as well as to certain matrices of mixed symmetry (see Appendix~\ref{a:small-complex-comp} for more details). Furthermore, we point out that in Corollary~\ref{cor:nozreal} we also show that the $|z_1-z_2|$--gain persists below the typical fluctuation scale of the eigenvalues of $H^{z_1},H^{z_2}$.%}

The characteristic method has been used previously in \cite{bourgade2021extreme, huang2019rigidity}. A similar idea was used earlier in the context of edge universality for Hermitian matrices \cite{lee2015edge}. Since these results, the characteristic flow has been very widely used in the context of single resolvent observables \cite{adhikari2020dyson, adhikari2023local, landon2024single, landon2022almost, campbell2024spectral} as well as to prove multi--resolvent local laws \cite{bourgade2206liouville, cipolloni2023eigenstate, cipolloni2023mesoscopic, cipolloni2023universality, cipolloni2024out, riabov2024eigenstate, stone2023random} for various models.

Finally, given \eqref{eq:brrw1}, we use a (modified) \emph{second moment method} to obtain the desired lower bound. Here, we follow the presentation of \cite{arguin2016extrema} of Kistler's multiscale refinement of the second moment method \cite{kistler2014derrida}. As a consequence of \eqref{eq:corrker}, the second moment method needs to be performed  differently in the real and complex cases.

\subsubsection{Upper bound}

In the complex case, the correlation structure \eqref{eq:corrker} suggests that, to leading order, the maximum of $\log|\mathrm{det}(X-z)|$ can be modelled by the maximum of $n$ independent Gaussians with variance $\frac{1}{4} \log n$. This ansatz motivates the proof of the upper bound in the complex case. Indeed, the Lipschitz continuity mentioned above implies that we can consider the maximum over $n$ points, after which the upper bound essentially follows from the fact that $\log|\mathrm{det}(X-z)|$ is approximately Gaussian and a union bound. We point out that this argument actually applies to establish the upper bound of the leading order asymptotic of logarithm of characteristic polynomials in all models discussed in Section~\ref{sec:rmlog}. In our case, after the regularization \eqref{eq:reg}, the Gaussianity of $\log|\mathrm{det}(X-z)|$ is proven using again the representation \eqref{eq:brrw1}, obtained using DBM.

The situation in the real case is more complicated. If $\Im[z] \asymp n^{-\alpha}$ with $\alpha \in [0, 1/2]$, then the asymptotic variance of $\log | \det (X -z )|$ implied by \eqref{eq:corrker} is $\frac{1+2\alpha}{4} \log n$. In each mesoscopic slice $\Im[z] \asymp n^{-\alpha}$, there are fewer points if $\alpha$ is larger, but the fluctuation itself grows.

%In fact, the covariance \eqref{eq:corrker} varies over the unit disk, and it is constant only on mesoscopic slices $\Im z\asymp n^{-\alpha}$, for $\alpha\in [0,1/2]$. In particular, in each slice we effectively have less points but the fluctuation grows as we approach the real axis. 
At first, one may hope that in order to obtain \eqref{eqn:intro-result}, one could compute the maximum over each of these slices using a union bound (as in the complex case), and then maximize in $\alpha$. However, this does not give the correct answer (see \eqref{eq:diffrealcase} and the discussion below it). Instead, to obtain \eqref{eq:corrker}, we need to use again the representation \eqref{eq:brrw1} as an inhomogeneous BRW. This allows us to decompose $\log|\mathrm{det}(X-z)|$ as the sum of two independent Gaussian fields according to the covariance structure \eqref{eq:corrker}. For one of these two fields we can estimate its maximum using a union bound, and then, given this information as an input, we compute the maximum of the sum. This is partly inspired by \cite{fang2012branching}.

\subsection*{Notations and conventions}

For integers $k\in \nn$  we use the notation $[k]:=\{1, 2,\dots, k\}$. We write $\D\subset\C$ to denote the open unit disc, and for any $\sigma\in \C$ we use the notation $\dif^2\sigma := 2^{-1}\ii (\dif \sigma\wedge\dif \overline{\sigma})$ to denote the two dimensional volume form on $\C$. For positive quantities $f, g$ we write $f\lesssim g$ and $f\asymp g$ if $f\le C g$ or $cg\le f\le Cg$, respectively, for some $n$-independent constants $c, C > 0$ which depend only on the constants appearing in \eqref{eq:moments}. We denote vectors by bold-faced lower case Roman letters ${\bm x}, {\bm y}  \in\C^d$ , for some $d\in \nn$, and their scalar product by
\[
\langle {\bm x}, {\bm y}\rangle:=\sum_{i=1}^d \overline{x_i}y_i.
\]
For any $d \times d$ matrix $A$ we use the notation $\langle A\rangle:= d^{-1}\mathrm{Tr}[A]$ to denote the normalized trace of $A$, and $A^\mathfrak{t}$ denotes the transpose of $A$. We denote the $d$--dimensional identity matrix by $I=I_d$. Furthermore, we define the $2\times 2$ block matrices
\begin{equation}
\label{eq:defE1E2}
E_1:=\left(\begin{matrix}
1 & 0 \\
0 & 0
\end{matrix}\right),
\qquad\quad
E_2:=\left(\begin{matrix}
0 & 0 \\
0 & 1
\end{matrix}\right).
\end{equation}
We also use the notation
\[
\tilde{\sum}_{ij}:=\sum_{(i,j)\in \{(1,2),(2,1)\}},
\]
to denote sums over matrices $E_1,E_2$.

Throughout the paper we will use the notion of a \emph{set of well spaced points} $P$ within another set $\Omega$ to denote a mesh of $|P|$ equidistant points contained in the set $\Omega$.

We will use the concept of “with overwhelming probability” meaning that for any fixed $D > 0$ the probability of the event is bigger than $1=n^{-D}$ if $n\ge n_0(D)$, with $n_0(D)$ possibly depending on the constants appearing in  \eqref{eq:moments} of the definition of our model, Definition \ref{def:model} below. Moreover, we use the convention that $\xi>0$ denotes an arbitrary small constant which is independent of $n$.

Throughout the paper various estimates will hold for $n$ ``sufficiently large.'' Here, sufficiently large can depend on the constants in the definition of our model in Definition \ref{def:model} below, as well as on parameters introduced before the phrase ``sufficiently large'' in the various statements of lemmas, propositions, and theorems below. For clarity, we will usually state this dependence below; however, we will not explicitly state the dependence on the model parameters in Definition \ref{def:model} as all estimates involving our i.i.d. matrices are assumed to depend on these parameters. Note also that the ``$n$ sufficiently large'' in the statement of overwhelming probability will also depend on these model parameters, as well as parameters introduced in the statements of lemmas, propositions, and theorems.

For real-valued martingales $M_t, N_t$, we denote the covariation process by $\d [ M_t, N_t]$. For complex valued martingales $M_t = X_t+ \i Y_t, N_t = P_t + \i Q_t$ the covariation process is defined by, $\d [ M_t, N_t] := \d [X_t, P_t] - \d [Y_t, Q_t] + \i ( \d [Y_t, P_t] + \d [ X_t, Q_t])$. The total variation process of a real-valued martingale is denoted by $[M_t] := \d [ M_t, M_t]$. 

\vspace{5 pt}

\noindent{\bf Acknowledgements.} B.L. heartily thanks Jiaoyang Huang and Paul Bourgade for lengthy discussions about the logarithmic polynomial of random matrices and the characteristic method. The research of B.L. is partially supported by an NSERC Discovery Grant and a Connaught New Researcher award.

\section{Main result} \label{sec:main}

We consider the following model of i.i.d. matrices:

\bed \label{def:model} An i.i.d. matrix is an $n \times n$ matrix $X$ whose entries are all independent, identically distributed (i.i.d.) random variables, $X_{ab} \stackrel{\dif}{=} n^{-1/2} \chi$. We always assume that $\ee[\chi] = 0$ and $\ee[ |\chi|^2]=1$. We will consider two classes of i.i.d. matrices, real i.i.d. matrices and complex i.i.d. matrices. In the real case $\chi \in \rr$ and in the complex case $\chi \in \cc$ and we further assume that $\ee[ \chi^2]=0$. We will always assume that for all $p \in \mathbb{N}$ there exists a $C_p >0$ so that,
\begin{equation}
\label{eq:moments}
\E |\chi|^p \le C_p.
\end{equation}
Throughout, we will use the parameter $\beta$ to unify formulas that hold in the real and complex cases. Specifically, in the real case $\beta=1$ and in the complex case $\beta=2$. 

We will also say that $X$ has a Gaussian component of size $a >0$ if $\chi = (1- a)^{1/2} \chi' + a^{1/2} g$ where $g$ is a standard real or complex Gaussian (matching the symmetry of $X$) and $\chi'$ is independent of $g$ and also obeys \eqref{eq:moments}. Throughout the paper, we will also use the abbreviation \emph{GDE} to denote the Gaussian divisible ensemble, i.e. i.i.d. matrices having a nonzero Gaussian component.
\eed
Our main observable of interest is the logarithm of the characteristic polynomial of $X$,
\begin{equation}
P_n(z):=\log |\mathrm{det}(X-z)|.
\end{equation}
The leading order asymptotics of the characteristic polynomial are given by,
\begin{equation}
\varphi(z):=\int_\D \log|z-\sigma|\,\dif^2\sigma=(\log |z|)_+-\frac{(1-|z|^2)_+}{2}.
\end{equation}
Note that for $|z|<1$ we have $\varphi(z)=(|z|^2-1)/2$.

Our main result for complex i.i.d. matrices is the following. 
\bet \label{theo:main} Let $X$ be a complex i.i.d. matrix as in Definition \ref{def:model}. Then for any $\eps >0$ and $ 0 < r < 1$ we have that
\begin{equation}
\label{eq:desb}
\lim_{n\to \infty}\PP\left( \left(\frac{1}{\sqrt{2}}-\eps\right)\log n\le\max_{|z|\le r} \big[ P_n(z) - n \varphi (z) \big] \le  \left(\frac{1}{\sqrt{2}}+\eps\right)\log n\right)=1.
\end{equation}
\eet
We remark that by inspecting the proof one finds that the probability of the event in \eqref{eq:desb} is bounded below by $1 - n^{-c_\eps}$ for some $c_\eps >0$ depending on $\eps >0$.

In the real case, there is an additional subleading deterministic contribution to the log characteristic polynomial. That is, if $X$ is a real i.i.d. matrix, one expects that for $|z| <1$ the random variable,
\beq
\label{eq:centerintrealcase}
\frac{P_n(z) - E_n(z)}{\sqrt{ \frac{1}{4} \log n  +\frac{1}{4} \left|  \log \left( |z-\bar{z} |^2 + n^{-1} \right)\right| }}, \qquad\quad E_n(z):= n \varphi (z) - \frac{1}{4} \log \left( |z-\bar{z} |^2 + n^{-1}\right)
\eeq
converges to a standard normal random variable, even if $\Im[z] \to 0$ as $n \to \infty$. This would not be hard to show with our techniques, but given the length of the current paper we leave this to future work.

\bet
\label{theo:realmaxlog}
Let $X$ be a real i.i.d. matrix. Then for any $\eps >0$ and $0 < r < 1$ we have that,
\beq \label{eqn:realmaxlog-1}
\lim_{n \to \infty} \pp\left[ \left( \frac{1}{ \sqrt{2}} - \eps \right) \log n\le\max_{ |z| \leq r } \big[P_n (z) - E_n(z)\big] \le \left( \frac{1}{ \sqrt{2}} + \eps \right) \log n \right] = 1.
\eeq
In addition, for any $ 0 < \alpha < \frac{1}{2}$ and $0 < r < 1$ we have that 
\beq \label{eqn:realmaxlog-2}
\lim_{n \to \infty} \pp\left[ \left( \frac{1}{ \sqrt{2}} - \eps \right) \log n\le\max_{ \substack{ |z| \leq r \\ n^{-\alpha} \leq \Im[z] \leq 2 n^{-\alpha} }}  \big[P_n (z) -E_n(z)\big]
\le \left( \frac{1}{ \sqrt{2}} + \eps \right) \log n\right] = 1.
\eeq
\eet

\

\begin{remark}[Comparison with the complex case]
\emph{We point out that the first order asymptotic of the maximum over mesoscopic slices $n^{-\alpha} \leq \Im[z] \leq 2 n^{-\alpha}$ for the real case in \eqref{eqn:realmaxlog-2} substantially differs from the answer one would obtain in the complex case. In fact, by following the proof of Theorem~\ref{theo:main} one can see that in the complex case we would have}
\[
\max_{ \substack{ |z| \leq r \\ n^{-\alpha} \leq \Im[z] \leq 2 n^{-\alpha} }} \big[P_n (z) -n\varphi(z)\big]\approx \sqrt{\frac{1-\alpha}{2}}\log n,
\]
\emph{In particular, unlike in the real case, the first order asymptotic would depend on $\alpha$. The maximum over this mesoscopic band is strictly smaller in the complex case than it is in the real case for any $\alpha \in (0, \frac{1}{2})$.}
\end{remark}

\

\begin{remark}[Maximum over the real axis]
\emph{We point out that using techniques similar to the proof of Theorems~\ref{theo:main}--\ref{theo:realmaxlog} we can also prove}
\begin{equation}
\label{eq:maxrealline}
\pp\left(\max_{x\in [-1+r,1-r]}\left[P_n(x)-n\varphi(x)+\frac{\log n}{4}\right]=\frac{\log n}{\sqrt{2}}\big(1+\mathcal{O}(\eps)\big)\right)=1,
\end{equation}
\emph{for any small $r>0$. We do not present its proof here for brevity.}
\end{remark}

\

\begin{remark}
\emph{In Theorems~\ref{theo:main}--\ref{theo:realmaxlog} and \eqref{eq:maxrealline} we gave a leading order asymptotic for the maximum of the characteristic polynomial over the domain $|z|\le r$, for some $0<r<1$. We expect that the same proof works for the maximum over $|z|\le 1$ (giving the same answer), but this would require significant rewriting of technical inputs to our work and so  we omit this for brevity.}
\end{remark}

\

We now present some technical results that will be used throughout the paper.

\subsection{Preliminaries}

Given $z \in \cc$ and a matrix $X \in \cc^{n \times n}$, the \emph{Hermitization} of $X-z$ is
\begin{equation}
\label{eq:herm}
H^z(X) = H^z:=\left(\begin{matrix}
0 & X-z \\
(X-z)^* & 0
\end{matrix}\right).
\end{equation}
Notice that $H^z\in\C^{2n\times 2n}$ has a $2\times 2$ block structure, i.e. it consists of four $n\times n$ blocks. This structure (known as \emph{chiral symmetry}) induces a spectrum symmetric around zero, i.e., denoting the eigenvalues of $H^z$ by $\{\lambda_{\pm i}^z\}_{i\in [n]}$, we have $\lambda_{-i}^z=-\lambda_i^z$, for $i\in [n]$. Furthermore, we point out that that $\{\lambda_i^z\}_{i\in [n]}$ are exactly the singular values of $X-z$.

In this context, Girko's Hermitization formula is the identity,
\begin{equation}
\label{eq:relnohermherm}
\log |\mathrm{det}(X-z)|=\frac{1}{2}\log |\mathrm{det} H^z|.
\end{equation}
This relation reduces the analysis of the non--Hermitian eigenvalues to study the eigenvalues of  the Hermitian matrix $H^z$. In particular, we can write
\begin{equation} \label{eqn:girko}
P_n (z) - n \varphi (z) = \frac{1}{2} \left( \sum_{i=-n}^n \log | \lambda_i^z| - 2 n\int_\R \log |x| \rho^z(x)\, \dif x \right)
\end{equation}
We point out that with the notation $\sum_{i=-n}^n$ we denote a summation where the index $i$ runs from $-n$ to $-1$ and from $1$ to $n$, i.e. the term $i=0$ is omitted. This notation will be used throughout the paper. Here $\rho^z(x)$, denoting the limiting eigenvalues distribution of $H^z$, is defined by
\begin{equation}
\rho^z(x):=\lim_{\eta\to 0^+}\frac{1}{\pi}\Im m^z(\ii \eta),
\end{equation}
with $m^z(w)$, for $w\in\C\setminus\R$, being the unique solution of the cubic equation (see e.g. \cite[Eqs. (2.4a)--(2.4b)]{alt2018local})
\begin{equation}
\label{eq:mde}
-\frac{1}{m^z(w)}=w+m^z(w)-\frac{|z|^2}{w+m^z(w)}, \qquad\quad \Im [ w ] \Im [ m^z(w) ]>0.
\end{equation}We point out that \eqref{eq:mde} consists of only one equation unlike \cite[Eqs. (2.4a)--(2.4b)]{alt2018local} since in our case, using the notation therein, the variance matrix $S$ is such that $S_{ij}=n^{-1}$ for all $i,j\in [n]$. In particular, the identity $\varphi (z) = \int_\rr \log |x| \rho^z (x) \d x $ follows from either the fact that they must both be the a.s.--limits of $n^{-1} \log | \det X|$ or, when $|z|\leq 1$, a direct calculation using \eqref{eq:evscfl} below (with initial data being a delta function). We now summarize various properties of the density $\rho^z(x)$ that we will use in the remainder of the paper. The proof of this lemma is presented in Appendix~\ref{sec:miscres}.
\begin{lemma}
\label{lem:proprho} Fix $0 < r < 1$. Let $\rho^z(x)$ be the density defined in \eqref{eq:behrho}. Uniformly in $z$ satisfying $|z| \leq r$ we have, % satisfies the following properties:
\begin{itemize}

    \item[(i)] The density $\rho^z$ is symmetric, and its support is given by $[-\mathfrak{e}_z, \mathfrak{e}_z]$ for an explicit $\mathfrak{e}_z>0$. In particular,  it consists of a single interval.

    \item[(ii)] The edge $\mathfrak{e}_z$ satisfies the bound $C^{-1}\le\mathfrak{e}_z\le C$, for some $C>0$.

    \item[(iii)] The density $\rho^z(x)$ has square root behavior close to $\mathfrak{e}_z$:
    \[
    \rho^z(\mathfrak{e}_z\pm \lambda)=\begin{cases}
    \gamma\sqrt{\lambda}\big(1+\mathcal{O}(\sqrt{\lambda})\big) &\mathrm{if} \quad\lambda \le 0 \\
    0 &\mathrm{if} \quad \lambda>0,
    \end{cases}
    \]
    for an explicit $\gamma>0$, with $C^{-1}\le \gamma\le C$.

    \item[(iv)] Fix any small $\delta>0$, then for $|x|\le \mathfrak{e}_z-\delta$ we have $\rho^z(x)\asymp 1$.

    \item[(v)] Fix any small $\delta,c>0$, and let $m^z$ be the solution of \eqref{eq:mde}. Then, for $|x|\le \mathfrak{e}_z-\delta$ and $0<\eta\le c$ we have $\Im m^z(x+\ii \eta)\asymp 1$.
    
\end{itemize}
\end{lemma}

We define the $n$--quantiles $\gamma_i^z$ of $\rho^z$ implicitly by
\begin{equation}
\label{eq:defquantz}
\int_0^{\gamma_i^z}\rho^z(x)\, \dif x=\frac{i}{2n}, \qquad\quad \mathrm{for}\quad i\in [n],
\end{equation}
and $\gamma_{-i}^z=-\gamma_i^z$ for $i\in [n]$.  For $ w \in \cc \backslash \rr$, we denote the resolvent by $G^z(w):=(H^z-w)^{-1}$. The local law (see Theorem~\ref{thm:ll} below) states that in the large $n$--limit the resolvent $G^z$ becomes approximately deterministic, i.e. that $G^z\approx M^z$ with
\begin{equation}
\label{eq:defM}
M^z=M^z(w):=\left(\begin{matrix}
m^z(w) & -zu^z(w) \\
-\overline{z}u^z(w) & m^z(w)
\end{matrix}\right), \qquad\quad u^z(w):=\frac{m^z(w)}{w+m^z(w)}.
\end{equation}
Here $m^z(w)$ is the unique solution of \eqref{eq:mde}. Additionally, the equation \eqref{eq:mde} (see also \cite[Proposition 2.1]{ajanki2019stability}) implies that $M^z(w)$ is the unique solution of
\begin{equation}
\label{eq:bigMDE}
-\frac{1}{M^z(w)}=w+Z+\langle M^z(w)\rangle, \qquad\quad Z:=\left(\begin{matrix}
0 & z \\
\overline{z} & 0
\end{matrix}\right),
\end{equation}
satisfying $\Im [w ]\Im [ M^z(w) ]>0$.

The following local laws and rigidity estimates may be found in \cite[Theorem 3.1]{cipolloni2021fluctuation}. 
\bet \label{thm:ll}  Fix $ 0 < r < 1$, $C>0$, and any small $\xi>0$. Uniformly in $ \Im [w] \geq n^{-1}$, $|w| \leq C$, $|z| \le r$, matrices $A \in \cc^{2n \times 2n}$ and vectors $\bx, \by \in \cc^{2n}$, we have with overwhelming probability,
\beq \label{eqn:entrywise-ll}
\left| \langle \bx , (G^{z}(w) - M^z(w) ) \by \rangle \right| \leq \| \bx \| \| \by \| \frac{ n^\xi}{\sqrt{n \Im[w]}}
\eeq
and
\beq \label{eqn:ll-1}
\left| \langle A (G^z (w) - M^z (w) ) \rangle \right| \leq \frac{ n^\xi \| A \|}{n \Im[w] }.
\eeq
Additionally, we have with overwhelming probability that,
\beq \label{eqn:usual-rigidity}
| \lambda_i^z - \gamma_i^z| \leq \frac{ n^\xi}{n^{2/3} ( 1+n - |i| )^{1/3} }
\eeq
\eet

Due to the importance of the quantity on the RHS of \eqref{eqn:girko} we will denote,
\beq
\Psi_n (z) := \sum_{i=-n}^n \log | \lambda_i^z| - 2 n \int \log |x| \rho^z (x) \d x + \1_{\{ \beta=1\}} \frac{1}{2} \log\left( |z-\bar{z}|^2 + 2 n^{-1} \sqrt{1-|z|^2} \right).
\eeq
Note that, in the complex case, $\Psi_n(z)$ differs from $P_n(z)-n\varphi(z)$ only by a factor of $2$, while in the real case there is an additional subleading order correction. We will need to consider a more general quantity. Throughout our work the matrix $X$ will be allowed to depend on time $t$, and we will denote the eigenvalues of the Hermitization of $X_t -z$ (as in \eqref{eq:herm}) by $\lambda_i (t)^z$. Furthermore, for any $\eta >0$ we denote, 
\begin{align} \label{eqn:psi-def}
    \Psi_n (z, t, \eta) &:= \Re \left( \sum_{i=-n}^n \log ( \lambda_i^z (t) - \i \eta ) - 2 n \int_{\rr} \log ( x - \i \eta ) \rho^z_t (x) \d x \right)  \notag \\
    &\quad+ \1_{\{ \beta=1\}} \frac{1}{2} \log\left( |z-\bar{z}|^2 + ( n^{-1} \vee \eta)  \right) \notag\\
    &=\frac{1}{2} \left( \sum_{i=-n}^n \log ( \lambda_i^z (t)^2 + \eta^2 ) - 2 n \int_{\rr} \log ( x^2 + \eta^2) \rho^z_t (x) \d x \right) \notag\\
    &\quad+\1_{\{ \beta=1\}} \frac{1}{2} \log\left( |z-\bar{z}|^2 + ( n^{-1} \vee \eta)  \right).
\end{align}
In the case that $X$ does not depend on time we will denote the above observable by $\Psi_n(z,\eta)$.  
Above, $\rho^z_t$ will be a possibly time-dependent limiting spectral distribution of $X_t$; whenever we introduce a time-dependent models of $X_t$, we will also introduce $\rho_t^z$ at the same time. Due to taking the real part, the choice of branch cut of the logarithm is immaterial, but for definiteness we will take the branch cut along the positive imaginary axis. Note that in principle, the additional term present in the real case  should also have some time dependence, but since we will always have $t \leq n^{-c}$, for some possibly very small fixed $c>0$, this will turn out to be lower order. 

The following is a consequence of \cite[Theorems 4.4--4.5]{cipolloni2023optimal}, and we provide the proof in Appendix~\ref{sec:der-bd}. 
\bep \label{prop:der-bd}
Let $0 < r < 1$, and fix any small $\xi >0$. For $X$ a real or complex i.i.d. matrix we have
\beq \label{eqn:der-bd}
\left| \Psi_n (z_1, \eta) - \Psi_n (z_2, \eta) \right| \leq \frac{ n^\xi |z_1 -z_2|}{ \sqrt{\eta}}
\eeq
with overwhelming probability uniformly in $z_1,z_2$ satisfying $|z_i | < r$ and $1/n\le\eta\le 1$
\eep

\section{Fine rigidity estimates for the Hermitization of $X-z$}
\label{sec:finrigest}

In this section we will derive a very precise bound on the eigenvalues $\lambda_i^z$ for small $i$. That is, we will show that $| \lambda_i^z  - \gamma_i^z| \ll \log n/ n$ for small $i$. The first step towards this estimate is the following improvement on the averaged local laws of Theorem \ref{thm:ll}, which replaces the $n^\xi$ error term with a correction sub-logarithmic in $n$ (observe that the results hold only for small $\Im w \ll 1$).

\begin{definition}
For $|z| \le r<1$ and $\kappa >0$ we define the bulk interval $I_z (\kappa)$ by 
\beq
I_z (\kappa):=\{ x : |x|\le \mathfrak{e}_z-\kappa \},
%\{ E : \rho^z(E) > \kappa \}
\eeq
with $\mathfrak{e}_z$ denoting the edge of $\rho^z$ (see Lemma~\ref{lem:proprho}).
%Note that $I_z (\kappa)$ consists of a single interval for all $\kappa >0$ sufficiently small.
\end{definition}

\begin{proposition}
\label{lem:precllaw}
Let $X$ be a real or complex i.i.d. matrix. Then, for any sufficiently small $\delta >0$ and $ \kappa>0$ it holds
\begin{equation}
\label{eq:goodll}
|\langle G^z(w)-M^z(w)\rangle|\lesssim \frac{(\log n)^{1/2+\delta}}{n|\Im w|},
\end{equation}
with overwhelming probability uniformly in $n^{-\delta}\ge|\Im w|\ge (\log n)^{1/2+10\delta}/n$ and $\Re w \in I_z ( \kappa)$.
\end{proposition}

\begin{remark}
\emph{In Proposition~\ref{lem:precllaw} we prove an averaged local law for  sub--logarithmic scales. We expect that the same proof should give a similar bound for the isotropic law without any additional effort. This means that for any deterministic unit vectors ${\bm x}, {\bm y}$, with overwhelming probability, we have
\begin{equation}
\label{eq:goodlliso}
\big|\langle {\bm x}, (G^z(w)-M^z(w)){\bm y}\rangle\big|\lesssim \frac{(\log n)^{1/2+\delta}}{\sqrt{n|\Im w|}}.
\end{equation}
We do not present its proof here for brevity. We also expect to be possible to choose $\delta=0$ in \eqref{eq:goodll}--\eqref{eq:goodlliso}, giving an optimal bound in terms of $n$. This would also give an optimal delocalization bound on the eigenvectors of $H^z$.}
\end{remark}

The proof of Proposition~\ref{lem:precllaw} is deferred to Section \ref{sec:dynamical-local-law}. The estimate \eqref{eq:goodll} implies the following rigidity estimate via the Helffer-Sj{\"o}strand formula. The proof is standard and deferred to Section \ref{a:goodrig-proof}. 
\begin{corollary}
\label{lem:goodrig}
Let $X$ be a real or complex i.i.d. matrix and fix $|z|\le r<1$. Then for any large $C>0$ and small $\delta >0$ we have that
\begin{equation}
\label{eq:vergoodrig}
|\lambda_i^z-\gamma_i^z|\le \frac{\log n^{1/2+\delta}}{n},
\end{equation}
for $|i|\le (\log n)^C$ with overwhelming probability. Furthermore, for any $\delta, \eps >0$ we have that,
\begin{equation}
\label{eq:goodrig}
|\lambda_i^z-\gamma_i^z|\le \frac{(\log n)^{3/2+\delta}}{n}
\end{equation}
for $|i| < n^{1-\eps}$ 
with overwhelming probability.
\end{corollary}

\subsection{Proof of Proposition~\ref{lem:precllaw}} \label{sec:dynamical-local-law}

We start this section by defining the concept of matrices with a Gaussian component:
\bed
\label{def:GDE}
We will say that a matrix $X$, as in Definition~\ref{def:model}, has a Gaussian component of size $a >0$ if $\chi = (1- a)^{1/2} \chi' + a^{1/2} g$ where $g$ is a standard real or complex Gaussian (matching the symmetry of $X$) and $\chi'$ is independent of $g$ and also obeys \eqref{eq:moments}. Throughout the paper, we will also use the abbreviation \emph{GDE} to denote the Gaussian divisible ensemble, i.e. i.i.d.\ matrices having a nonzero Gaussian component.
\eed

We prove the local law in Proposition~\ref{lem:precllaw} dynamically. That is, we will first prove \eqref{eq:goodll} for matrices with a fairly large Gaussian component, using Dyson Brownian motion and then remove this Gaussian component by a standard Green's function comparison (GFT) argument.

First, note that the local law \eqref{eqn:ll-1} with $A = I$ implies that,
\begin{equation}
\label{eq:incond}
\big|\langle G^z(w)-M^z(w)\rangle\big|\le \frac{n^\xi}{n|\Im w|},
\end{equation}
with overwhelming probability for any small $\xi>0$, uniformly in $|\Im w|\ge n^{-1+\xi}$. We will show that along the flow \eqref{eq:OU} below we can improve this bound in two directions. First, that the $n^\xi$ in the RHS of \eqref{eq:incond} can be replaced by $(\log n)^{1/2+\delta_1}$, for some small fixed $\delta_1>0$; and then second, that this bound will hold uniformly in $|\Im w|\ge n^{-1}(\log n)^{1/2+10\delta_1}$.

Consider the Ornstein-Uhlenbeck flow
\begin{equation}
\label{eq:OU}
\dif X_t=-\frac{1}{2}X_t\dif t+\frac{\dif B_t}{\sqrt{n}}, \qquad\quad X_0=X,
\end{equation}
where we consider two cases. First, if $X$ is a complex i.i.d. matrix, then $B_t$ is a matrix of i.i.d. standard complex Brownian motions. If $X$ is a real i.i.d. matrix, then $B_t$ is a matrix of i.i.d. standard real Brownian motions. As indicated in Definition \ref{def:model}, we will use the parameter $\beta=1, 2$ to denote the real and complex cases, respectively.

Let $H_t^z = H^z (X_t)$ be the Hermitization of $X_t-z$ defined as in \eqref{eq:herm} with $X$ replaced by $X_t$, and define its resolvent by $G_t^z(w):=(H_t^z-w)^{-1}$, with $w\in\C\setminus\R$. In particular, along \eqref{eq:OU} the first two moments of $H_t^z$ are preserved and so $\rho^z (x)$ will continue to be a good approximation to its empirical eigenvalue distribution for any $t\ge 0$. 

Recall the definition of $E_1, E_2$ in \eqref{eq:defE1E2} as well as $\tilsumij$ directly below that equation. 
Then, using It\^{o}'s formula we obtain (recall that $A^\mft$ denotes the transpose of $A$)
\begin{equation}
\label{eq:flownochar}
\begin{split}
\dif \langle G^z_t(w)\rangle&=\dif N^z_t(w)+\frac{1}{2}\langle G^z_t(w)\rangle \dif t+\frac{1}{2}\langle (Z+w)G^z_t(w)^2\rangle \dif t+2\tilde{\sum}_{ij}\langle G^z_t(w)E_i\rangle\langle G^z_t(w)^2E_j\rangle \dif t \\
&\quad+ \frac{\1_{ \{ \beta=1\}}}{n} \tilde{\sum}_{ij} \langle G^z_t(w)^2 E_i G^z_t(w)^\mft E_j \rangle \d t
\end{split}
\end{equation}
where
\begin{equation}
\label{eq:defbigBM}
\dif N^z_t(w):=-  \frac{1}{n^{1/2}}\langle (G^z_t)^2\dif \mathfrak{B}_t \rangle, \qquad\quad \mathfrak{B}_t = \left( \begin{matrix} 0 & B_t \\ B_t^* & 0 \end{matrix}\right).
\end{equation}

Associated with \eqref{eq:flownochar} are the characteristics,
\begin{equation}
\label{eq:charnew}
\partial_t w_t=-m^{z_t}(w_t)-\frac{w_t}{2}, \qquad\quad \partial_t z_t=-\frac{z_t}{2},
\end{equation}
with initial conditions $\eta_0=n^{-\xi}$ and $|z_0|\le r<1$.  By implicitly differentiating \eqref{eq:bigMDE} with respect to $t$ one finds that,
\beq
\label{eq:derMt}
\frac{\d}{\d t} M^{z_t} (w_t) = \frac{1}{2} M^{z_t} (w_t).
\eeq
Computing the flow \eqref{eq:flownochar} along the characteristics \eqref{eq:charnew}, using \eqref{eq:derMt}, we find that,
\begin{align}
\label{eq:flowchar}
\dif \langle G^{z_t}_t(w_t)-M^{z_t}(w_t)\rangle&=\dif N^{z_t}_t(w_t)+\left(\frac{1}{2}+\langle G^{z_t}_t(w_t)^2\rangle\right)\langle G^{z_t}_t(w_t)-M^{z_t}(w_t)\rangle\dif t \notag\\
&\quad+ \frac{\1_{ \{ \beta=1\}}}{n} \tilde{\sum}_{ij} \langle G^{z_t}_t(w_t)^2 E_i G^{z_t}_t(w_t)^\mft E_j \rangle \d t
\end{align}
where we used that $2\langle G^{z_t}_t(w)E_i\rangle=\langle G^{z_t}_t(w)\rangle$ and $2\langle G^{z_t}_t(w)^2 E_i \rangle = 2 \del_w \langle G^{z_t}_t(w) E_i \rangle = \del_w \langle G^{z_t}_t (w) \rangle = \langle G^{z_t}_t(w)^2 \rangle$.  For notational simplicity we will drop the superscript and denote $G_t = G_t^{z_t}(w_t)$ and $N_t = N^{z_t}_t$. We remark also that if either $\Re[w_t] \in I_{z_t} (\kappa)$ or $\Re[w_0] \in I_{z_0} (\kappa)$ for some $\kappa >0$, then, if $t$ is sufficiently small, we have that $\Re[w_s] \in I_{z_s} (\frac{1}{2} \kappa)$ for all $0 \leq s \leq t$. 

In the remainder of the section we will denote $\eta_s := \Im[ w_s]$. In several places we will use the fact that if $\Re[w_s] \in I_{z_s} ( \kappa)$ then $ - \del_s \eta_s \asymp 1$, which follows from the last point of Lemma~\ref{lem:proprho}.

\bel \label{lem:ll-1} Let $\xi, \kappa >0$ and $\delta >0$. 
Let $(z_s, w_s)$ denote a characteristic with $\Re[w_0] \in I_{z_0} ( \kappa)$, $\Im[w_0] = n^{-\xi}$ and $n^{-5 \xi } \geq \Im[w_t] \geq n^{-1+10 \xi}/n$. Then with overwhelming probability,
\beq
\label{eq:firstimprov}
\left| \langle G_t^{z_t} (w_t) - M^{z_t} (w_t) \rangle \right| \leq \frac{ (\log n)^{1/2+\delta}}{ n \Im[w_t] }.
\eeq
\eel
\proof First note that \eqref{eqn:ll-1} with $A = I$ implies that with overwhelming probability we have
\beq \label{eqn:ll-initial}
\left| \langle G_s(w) - M^{z_s} (w) \rangle \right| \leq \frac{ n^{\xi}}{n \Im[w] } , \qquad \left| \langle \partial_w( G_s (w) - M^{z_s} (w)) \rangle \right| \leq \frac{ n^{\xi}}{n \Im[w]^2 } ,
\eeq
uniformly in $0 \leq s \leq t$ and $\Im[w] \geq n^{-1+\xi}$ (with the second inequality following from the Cauchy integral formula).

Integrating \eqref{eq:flowchar} in time, using \eqref{eqn:ll-initial} and $|\langle \partial_w M^z(w)\rangle|\lesssim 1$ for $\Re[w] \in I_z (\kappa)$, we conclude (recall  $\eta_0=n^{-\xi}$)
\begin{equation}
\begin{split}
\label{eq:almthere}
\langle G_t(w_t)-M^{z_t}(w_t)\rangle&= \int_0^t \dif N_s +\int_0^t\left(\frac{1}{2}+\langle \partial_{w_s} M^{z_s}(w_s)\rangle\right)\langle G_s(w_s)-M^{z_s}(w_s)\rangle\,\dif s\\
&\quad+\mathcal{O}\left(\frac{n^{2\xi}}{n}+\frac{n^{2\xi}}{(n\eta_t)^2} +\frac{ \1_{\{ \beta=1\}}}{ n \eta_t}\right) \\
&= \int_0^t \dif N_s + \mathcal{O}\left(\frac{n^{2\xi}}{n}+\frac{n^{2\xi}}{(n\eta_t)^2} +\frac{ \1_{\{ \beta=1\}}}{ n \eta_t} \right)
\end{split}
\end{equation}
In order to estimate the term on the second line of \eqref{eq:flowchar} we used,
\beq
\frac{1}{n} \left|  \langle G_t(w_t)^2 E_i G_t(w_t)^\mft E_j \rangle \right| \lesssim \frac{1}{n} \langle |G_t (w_t)|^4 \rangle^{1/2} \langle |G_t(w_t)|^2 \rangle^{1/2} \lesssim \frac{1}{n \eta_t^2} \langle \Im[G_t (w_t)] \rangle \lesssim \frac{1}{n \eta_t^2}
\eeq
with overwhelming probability. We are thus left only with the estimate of the martingale term. Let $\tau>0$ be the stopping time,
\beq
\tau :=\inf \{ 0 < s < t : \langle |G_s (w_s) - M^{z_s} (w_s)| \rangle > 1 \} 
\eeq
Note that $\tau = t$ with overwhelming probability. For $t < \tau$ we have for the quadratic variation of $N_s$ (by direct computation),
\begin{equation}
\label{eq:quadvarN}
\dif [N_s,\bar{N}_s] = \tilde{\sum}_{ij} \left( \frac{1}{n^2}\langle G_s(w_s)^2E_iG_s(\overline{w_s})^2E_j\rangle + \frac{\1_{ \{ \beta=1\}} }{n^2} \langle G_s(w_s)^2 E_i [ G_s (\bar{w}_s)^2]^\mft E_j \rangle \right) \dif s  \le\frac{C}{n^2\eta_s^3}\dif s.
\end{equation}
The inequality follows by Cauchy-Schwarz and the fact that 
\beq
\langle E_i G_s(w_s)^2 (G_s(w_s)^* )^2 \rangle \leq C \eta_s^{-3} \langle \Im [ G_t (w_t) ] \rangle \leq C \eta_s^{-3} . \eeq
By the martingale representation theorem, the real and imaginary parts of the stopped process $N^\tau_s$ are each equal in distribution to processes $X_s, Y_s$ that satisfy $X_s = \tilde{b}_{[X_s]}$, $Y_s = \tilde{b}_{[Y_s]}$, where $\tilde{b}$ is a standard Brownian motion. Since the total variation processes of the real and imaginary parts of $N^{\tau}_s$ are bounded above by $[N^{\tau}_s, \bar{N}^{\tau}_s]$ and by definition of $\tau$, $[N^{\tau}_s, \bar{N}^{\tau}_s] \lesssim ( n \eta_s)^{-2}$, we obtain
\begin{equation} \label{eqn:mart-rep-1}
\begin{split}
\PP\left(\sup_{0\le s\le t} |N^\tau_s|> \frac{u}{n\eta_t}\right)\lesssim \PP\left((n\eta_t)\sup_{0\le s\le C/(n\eta_t)^2} |b_s|> u\right)\lesssim e^{-cu^2},
\end{split}
\end{equation}
for some small constant $c>0$. This implies that
\begin{equation}
\label{eq:finmartb}
\PP\left(\exists s\in [0,t]: |N_s|\ge \frac{(\log n)^{1/2+\delta}}{n\eta_t} \right)\lesssim n^{-D},
\end{equation}
for any $D>0$, which together with \eqref{eq:almthere} completes the proof (here we use the fact that $n^{2 \xi -1} + n^{2 \xi} / (n \eta_t)^2 \leq n^{-\xi} / (n \eta_t)$ by our assumptions on $\eta_t$). \qed

We now propagate the above estimate to shorter scales. 
\bel \label{lem:ll-2}
Let $\xi, \kappa, \delta >0$ be sufficiently small. Let $(z_s, w_s)$ denote a characteristic with $\Re[w_0] \in I_{z_0} ( \kappa)$, $\Im[w_0] = n^{-\xi}$ and $\Im[w_t] \geq (\log n)^{1/2+10\delta} / n$. For all $n$ sufficiently large, depending on $\xi, \kappa, \delta$, the following holds. Assume that with overwhelming probability,
\beq
\label{eq:incondnec}
\left| \langle G_0^{z_0} (w_0) - M^{z_0} (w_0) \rangle \right| \leq \frac{ (\log n)^{1/2+\delta}}{n \Im[w_0] }, \qquad \left| \langle  \partial_w( G_0^{z_0} (w_0) - M^{z_0} (w_0) ) \rangle \right| \leq \frac{ (\log n)^{1/2+\delta}}{n \Im[w_0]^2 } .
\eeq
Then with overwhelming probability we have that
\beq
\left| \langle G_t^{z_t} (w_t) - M^{z_t} (w_t) \rangle \right| \leq \frac{ (\log n)^{1/2+2\delta}}{n \Im[w_t] } .
\eeq
\eel
\proof  Note that the term $\langle G_t(w_t)^2\rangle$ appears in the flow \eqref{eq:flowchar}. For this purpose we study the evolution of this term along the characteristics (see e.g. \cite[Eq. (5.7)]{cipolloni2023mesoscopic} for $A=B=I$):
\begin{equation}
\label{eq:evG2}
\begin{split}
\dif \langle G_s(w_s)^2-\partial_w M^{z_s}(w_s)\rangle&=\dif \widehat{N}_s+\big(1+ 2 \langle \partial_w M^{z_s}(w_s)\rangle\big)\langle G_s(w_s)^2-\partial_w M^{z_s}(w_t)\rangle \\
&\quad+ \langle G_s(w_s)^2-\partial_w M^{z_s}(w_s)\rangle^2 \\
&\quad+2\langle G_s(w_s)-M^{z_s}(w_s)\rangle\langle G_s(w_s)^3\rangle \\
&\quad +  \frac{\1_{\{ \beta=1\}}}{n} \tilde{\sum}_{ij} \left( \langle G_t(w_t)^3 E_i G_t(w_t)^\mft E_j \rangle + \langle G_t(w_t)^2 E_i [ G_t(w_t)^2]^\mft E_j \rangle \right) \d t
\end{split}
\end{equation}
with (recall the definition of $\mathfrak{B}_t$ from \eqref{eq:defbigBM})
\begin{equation}
\dif \widehat{N}_s :=- \frac{1}{n^{1/2}}\langle (G^{z_s}_s)^3 \dif \mathfrak{B}_s\rangle .
\end{equation}
We remark that in \eqref{eq:evG2} we used  \cite[Lemma 5.5]{cipolloni2023mesoscopic} in the form
\beq
\partial_s \langle \partial_w M^{z_s}(w_s)\rangle=\langle \partial_w M^{z_s}(w_s)\rangle+ \langle \partial_w M^{z_s}(w_s)\rangle^2,
\eeq
and that $ \langle G_s(w_s)^2(E_1-E_2)\rangle=\langle\partial_w M^{z_s}(w_t)(E_1-E_2)\rangle=0$ by spectral symmetry. Here by spectral symmetry we refer to the symmetry of the spectrum around zero as discussed below \eqref{eq:herm}.

Define
\beq
X_s:=\langle G_s(w_s)-M^{z_s}(w_s)\rangle, \qquad\quad Y_s:=\langle G_s(w_s)^2-\partial_{w} M^{z_s}(w_s)\rangle,
\eeq
and the stopping time,
\begin{equation}
\label{eq:stoptime}
\tau:=\inf\left\{s\ge 0:\,|X_s|= \frac{(\log n)^{1/2+2\delta}}{n \eta_s}, \quad |Y_s\big|=\frac{(\log n)^{1/2+3\delta}}{n \eta_s^2}\right\} \wedge t,
\end{equation}
where $t$ is as in the statement of the lemma and $\eta_t\ge (\log n)^{1/2+10\delta}/n$. Necessarily, $t \asymp \Im[w_0]$.  Note that by our assumptions on the initial conditions we have that $\tau>0$ with overwhelming probability. Then, by \eqref{eq:flowchar} and \eqref{eq:evG2}, we have with overwhelming probability, for any $0 < s < \tau$,
\begin{equation}
\label{eq:almtherenew}
X_s= \int_0^s \dif N_u +\int_0^s\left(\frac{1}{2}+\langle \partial_{w_u} M^{z_u}(w_u)\rangle\right)X_u\,\dif u+\mathcal{O}\left(\frac{(\log n)^{1/2+\delta}}{n\eta_0}+\frac{(\log n)^{1+5\delta}}{(n\eta_s)^2}\right),
\end{equation}
and
\begin{equation}
\label{eq:boundG2}
Y_s=\int_0^s\dif \widehat{N}_u+\int_0^t\big(1+2\langle\partial_{w} M^{z_u}(w_u) \rangle\big) Y_u\,\dif u+\mathcal{O}\left(\frac{(\log n)^{1/2+\delta}}{n\eta_0^2}+\frac{(\log n)^{1+6\delta}}{(n\eta_s)n\eta_s^2} + \frac{ (\log  n)^{1/2+2\delta}}{n \eta_s^2}\right).
\end{equation}
We point out that to estimate the error in \eqref{eq:boundG2} we used
\beq
\label{eq:w}
\big|\langle G_t(w_t)^3\rangle\big| \le\langle |G_t(w_t)|^2\rangle^{1/2}\langle |G_t(w_t)|^4\rangle^{1/2}=\frac{\sqrt{\langle \Im G_t(w_t)\rangle \langle \Im G_t(w_t)^2\rangle}}{\eta_t^{3/2}}\le \frac{\langle \Im G_t(w_t)\rangle}{\eta_t^2}\lesssim \frac{1}{\eta_t^2},
\eeq
where in the middle equality we used the Ward (resolvent) identity $G_t(w_t)G_t(w_t)^*=\Im G_t(w_t)/\eta_t$, in the penultimate inequality we used $\lVert \Im G_t(w_t)\rVert\le 1/\eta_t$, and in the last inequality we used the definition of $\tau$ in \eqref{eq:stoptime} and the fact that $(\log n)^{1/2+\delta} / (n \eta_s) \lesssim 1$. We point out that in the remainder of the proof we will often use similar bounds to \eqref{eq:w} even if we do not say it explicitly. For the terms when $\beta=1$ on the last line of \eqref{eq:evG2} we used,
\begin{align}
 & \left| \langle G_t(w_t)^3 E_i G_t(w_t)^\mft E_j \rangle + \langle G_t(w_t)^2 E_i [ G_t(w_t)^2]^\mft E_j \rangle \right| \notag \\
&\qquad\qquad\qquad\quad\lesssim  \langle |G_t(w_t)|^6 \rangle^{1/2} \langle |G_t(w_t)|^2 \rangle^{1/2} + \langle |G_t(w_t)|^4 \rangle \lesssim \frac{1}{\eta_t^3}
\end{align}
for $t < \tau$. 
For the martingale terms, for $s < \tau$, we have
\begin{align}
\d [ N_s, \bar{N}_s] &=    \frac{1}{n^2} \tilde{\sum}_{ij} \left( \langle G_t(w_t)^2 E_i G_t (\bar{w}_t)^2 E_j \rangle + \1_{\{ \beta =1 \}} \langle G_t (w_t)^2 E_i [G_t (\bar{w}_t)^2]^\mft E_j \rangle \right)  \notag \\
& \lesssim n^{-2} \langle |G_s(z_s)|^4 \rangle \lesssim n^{-2} \eta_s^{-3} \langle \Im[G_s (z_s)] \rangle   \lesssim \frac{1}{ n^2 \eta_s^3}, 
\end{align}
and
\begin{align}
\d [ \widehat{N}_s, \widehat{\bar{N}}_s] &= \frac{1}{n^2} \tilde{\sum}_{ij} \left( \langle G_t(w_t)^3 E_i G_t (\bar{w}_t)^3 E_j \rangle + \1_{\{ \beta =1 \}} \langle G_t (w_t)^3 E_i [G_t (\bar{w}_t)^3]^\mft E_j \rangle \right)  \notag\\
&\lesssim n^{-2} \langle |G_s(z_s)|^6 \rangle \lesssim \frac{1}{ n^2 \eta_s^5}.
\end{align}
Therefore, by the same argument that uses the martingale representation theorem in the proof of Lemma \ref{lem:ll-1}, for any $0 < s < t $, we have
\beq \label{eqn:aaa-3}
\pp\left[ \sup_{ 0 < u < s } | N_{u\wedge \tau }| >  \frac{ (\log n)^{1/2+\delta/2}}{ n \eta_s } \right] +\pp\left[ \sup_{ 0 < u < s } | \widehat{N}_{u\wedge \tau }| >  \frac{  (\log n)^{1/2+\delta/2}}{ n \eta_s^2 } \right] \leq n^{-D}
\eeq
for any $D>0$.  We now claim that in fact the stronger estimate,
\beq \label{eqn:aaa-2}
\pp\left[ \sup_{ 0 < s<t }  \eta_s n| N_{s\wedge \tau }| >  (\log n)^{1/2+\delta} \right] +\pp\left[ \sup_{ 0 < s<t } \eta_s^2 n | \widehat{N}_{s\wedge \tau }| >  (\log n)^{1/2+\delta} \right] \leq n^{-D}.
\eeq
holds.  To prove this, take a sequence of times $s_i$ such that $\eta_{s_i} = \frac{1}{2} \eta_{s_{i-1}}$, with $s_0=0$. There are at most $\O ( \log n)$ such times until $\eta_{s_i} < \eta_t$. For $s \in [s_{i-1}, s_i]$ we have that $\eta_s \asymp \eta_{s_i}$. Therefore, \eqref{eqn:aaa-3} implies
\beq
\pp\left[ \sup_{ u \in [s_{i-1}, s_i] } |(n \eta_u) N_{u\wedge \tau} | >  ( \log n) ^{1/2+\delta} \right] +\pp\left[ \sup_{ u \in [s_{i-1}, s_i] }  | n \eta_u^2 \widehat{N}_{u\wedge \tau }| >  ( \log n)^{1/2+\delta} \right] \leq n^{-D}
\eeq
From a union bound we therefore conclude \eqref{eqn:aaa-2}. 

Therefore, with overwhelming probability we have for all $0 < s < \tau$ that,
\beq
X_s = \int_0^s\left(\frac{1}{2}+\langle \partial_{w_u} M^{z_u}(w_u)\rangle\right)X_u\,\dif u + \O \left( \frac{ (\log n)^{1/2+\delta}}{n \eta_s } \right)
\eeq
and
\beq
Y_s = \int_0^s\left(1+2\langle \partial_{w_u} M^{z_u}(w_u)\rangle\right)Y_u\,\dif u + \O \left( \frac{ (\log n)^{1/2+2\delta}}{n \eta_s^2 } \right) .
\eeq
Note that in order to simplify the errors in \eqref{eq:almtherenew} and in \eqref{eq:boundG2} we used the fact that $n \eta_s \geq (\log n)^{1/2+10 \delta}$ by assumption. From the integral form of Gronwall inequality, using $|\langle \partial_{w_u} M(w_u)\rangle|\le C$, we then see that with overwhelming probability for any $0 < s < \tau$ we have that,
\beq
|X_s| \leq C \frac{ (\log n)^{1/2+\delta}}{n \eta_s}, \qquad |Y_s| \leq C \frac{(\log n)^{1/2+2\delta}}{n \eta_s^2} .
\eeq
Since $X_s$ and $Y_s$ are continuous, we cannot have that $\tau < t$, and so the claim follows. \qed

\

The above two lemmas easily imply the following. In particular, the assumption \eqref{eq:incondnec} of Lemma~\ref{lem:ll-2} is satisfied as a consequence of \eqref{eq:firstimprov} and Cauchy integral fromula.
\bep \label{prop:gde-ll}
Let $\xi, \kappa$ and $\delta >0$. Let $X$ be a real or complex i.i.d. matrix with Gaussian component of size at least $n^{-\xi/10}$. Then with overwhelming probability we have for all $w$ satisfying $(\log n)^{1/2+\delta} \leq \Im[w] \leq n^{-\xi}$ and $\Re[w] \in I_z ( \kappa)$ that
\beq
\left| \langle G^z (w) - M^z (w) \rangle \right| \leq \frac{(\log n)^{1/2+\delta}}{n \Im[w] }.
\eeq
\eep

\qed

Strictly speaking, our methods do not require us to prove the above estimates for matrices with no Gaussian component, as this local law is only used to analyze the dynamics. However, for notational simplicity, and because the results may be of use in other problems, in the next section we use a Green's function comparison argument to extend the local law to all matrices.

\subsubsection{Removal of Gaussian divisible component} \label{sec:gde-removal}

In this section we extend Proposition \ref{prop:gde-ll} to general i.i.d. matrices. We just present the proof in the complex i.i.d. case, the other case being analogous. 
Let $Z(X)$ be the function on the space of $n \times n$ matrices given by,
\beq
Z(X) := n \Im[w] \left| \langle G^z (w) - M^z (w) \rangle \right|.
\eeq
Let $X$ and $Y$ be two $n \times n$ matrices such that
\beq
\ee[ X_{ij}^a \bar{X}_{ij}^b] = \ee[ Y_{ij}^a \bar{Y}_{ij}^b] 
\eeq
for $0 \leq a + b \leq 3$ and 
\beq
\left| \ee[ X_{ij}^a \bar{X}_{ij}^b] - \ee[ Y_{ij}^a \bar{Y}_{ij}^b] \right| \leq T n^{-2}
\eeq
for $a+b = 4$. Here $T = n^{-\eps}$ for some $\eps >0$. Let $W^{(ab)}$ be the matrix obtained by replacing all the entries $(i, j)$ of $X$ with $i \leq a$ or $j \leq b$ with those of $Y$. Define now,
\beq
p(k) := \sup_{0 \leq a, b \leq n } \pp\left[ Z (W^{(ab)} ) > k (\log n)^{1/2+\delta} \right].
\eeq
Then we have the following which is proven in Appendix \ref{a:ll-gfct}. 
\bep \label{prop:ll-gfct} Assume that $\Re[w] \in I_z ( \kappa)$ and that $n^{-1} \leq \Im[w] \leq 1$. Then,
there is a constant $C>0$ so that for $k \geq 2$ we have,
\beq 
p(k) \leq C  \pp\left[ Z(Y) \geq (\log n)^{1/2+\delta} \right] + C T^{1/2} p(k-1) + n^{-D}.
\eeq
\eep

\noindent{\bf Proof of Proposition~\ref{lem:precllaw}}. For any fixed $\xi >0$, Proposition \ref{prop:gde-ll} implies that Proposition \ref{lem:precllaw} holds for matrices with Gaussian component of size at least $T := n^{-\xi/10}$. For any given ensemble $X$ we may find another ensemble $Y$ so that the first three moments of $Y$ match those of $X$ and the fourth moments differ by $\O ( Tn^{-2})$ and $Y$ has Gaussian component of size least $T$ (see, e.g., Lemma 3.4 of \cite{erdos2010universality}). Therefore, iterating the estimate of Proposition \ref{prop:ll-gfct} $k$ times, we get 
\beq
p(k) \leq C n^{-D} + C T^{k/2}.
\eeq
Taking $k$ sufficiently large, depending on $\xi >0$ yields the claim. \qed

\subsection{Local laws for matrices of mixed symmetry}

In this section we introduce matrices which have a mixed symmetry class. More precisely, they consist of the sum of two independent matrices, one being a real i.i.d. matrix and one being a (small) complex Ginibre matrix. This class of matrices will appear at a certain point in the proof of the lower bound for real matrices (see Lemma~\ref{lem:ind-prod} below) for purely technical reasons.

\bed \label{def:type-M} We say that $X$ is a matrix of type M if it can be written in the form $X = (1 - t)^{1/2} Y + \sqrt{t} G$ where $Y$ is a real i.i.d. matrix, $G$ is a complex Ginibre matrix, and $t \leq n^{-\eps}$ for some $\eps >0$.
\eed

We now claim that the local law and rigidity estimates from Proposition~\ref{lem:precllaw} and Corollary~\ref{lem:goodrig} still hold for this class of matrices. The proof of this lemma is postponed to Appendix~\ref{sec:m-laws-proof}.  

\begin{lemma} \label{prop:m-laws}
If $X$ is a matrix of type M, then the local law and rigidity \eqref{eqn:ll-1}--\eqref{eqn:usual-rigidity}, the estimate \eqref{eq:goodll}, and the results of Corollary \ref{lem:goodrig} hold. For other eigenvalues, for any $c>0$ and all $ 1 \leq i \leq (1-c) n$, we have
\beq \label{eqn:type-m-rig}
 | \lambda_i^z - \gamma_i^z| \leq \frac{ n^\xi}{n}
\eeq
and that $| \lambda_n^z  - \gamma_n^z| \leq n^\xi \sqrt{t}$ with overwhelming probability for any small $\xi>0$. 
\end{lemma}

\section{Maximum on almost-global scales}

In this section we present a bound for the regularized characteristic polynomial $\Psi_n (z, \eta)$ (recall the definition \eqref{eqn:psi-def}) when $\eta = n^{-\gamma}$ for small $\gamma>0$. This will be used to truncate various large scale contributions throughout our proofs. 
\bep \label{prop:global} Let $0 < r < 1$. There is a $C_1 >0$ so that the following holds. Let $\gamma >0$, $ C>0$, and define $\eta_*:=n^{-\gamma}$. Then,
 for any real or complex i.i.d. matrix we have,
\beq
\pp\left[ \max_{ |z|\le r, (\log n)^{-C} \eta_* \leq \eta \leq (\log n)^C \eta_* } | \Psi_n (z, \eta) | > C_1 \gamma \log n \right] \leq n^{-3 \gamma}
\eeq
for all sufficiently small $\gamma >0$, and all $n$ sufficiently large, depending on $\gamma, r, C$. 
\eep

The main ingredient to prove Proposition~\ref{prop:global} is the following estimate of the characteristic function of a linear statistic, whose proof is postponed to Appendix~\ref{app:addtechres}.

\begin{proposition}  \label{prop:CLT-1}

Fix any sufficiently small $\gamma>0$, and let $f:\R\to\R$ be in $C_0^\infty([-5,5])$ and such that $\lVert f\rVert_{C^k}\lesssim n^{k \gamma}$ for all sufficiently large $k$, depending on $\gamma$. Then for $\lambda \in \rr$ satisfying $| \lambda| \leq n^{1/100}$ we have
\begin{equation}
\label{eq:steinexpl}
\E\left[\exp\left(\ii\lambda \big(\mathrm{Tr} f(H^z)-\E \mathrm{Tr} f(H^z)\big)\right)\right]=\exp\left(-\frac{\lambda^2}{2}V(f)\right)+\mathcal{O}\left(\frac{n^{200\gamma}}{n^{1/4}}\right),
\end{equation}
for some explicit $V(f)\ge -n^{-1/5}$. Additionally, if $f(x) = \Re \log ( x - \i \eta)$, with $\eta = n^{-\gamma}$, then 
\beq
\label{eq:explvarlog}
\begin{split}
\E \mathrm{Tr}f(H^z)&=n\int\log(x^2+\eta^2)\rho^z (x) \d x +\frac{\1_{ \{ \beta =1 \}}}{2}\log\big[|z-\overline{z}|^2+\eta\big]+\O(1), \\
V(f) &= - \log \eta -\bm1_{\{\beta=1\}}\log[|z-\overline{z}|^2+\eta]+ \O ( (\log n)^{1/2}).
\end{split}
\eeq
\end{proposition}

The above is readily seen to imply the following via Fourier duality.

\bel \label{lem:global-tail}
There is a $C_1 >0$ so that the following holds. Let $f = \Re \log(x - \i \eta )$ with $ \eta = n^{-\gamma}$, for $ \gamma >0$ sufficiently small. Then, 
\beq
\pp\left[ \left| \Tr f (H^z) - 2n \int \Re \log(x- \i \eta ) \rho^z (x) \d x \right| > C_1 \gamma \log n \right] \leq n^{-5 \gamma} .
\eeq
\eel
\proof Let $0 \leq F(x) \leq 1$ be a smooth function with bounded derivatives such that $F(x) =1 $ for $|x| \leq C_1 \gamma \log n$ and $F(x) = 0$ for $|x| > C_1 \gamma \log n +1 $. Let $\hat{F} ( \lambda)$ denote its Fourier transform. Then,
\beq \label{eqn:F-decay}
| \hat{F} ( \lambda ) | \leq  \frac{ (\log n)^2 }{ 1 + |\lambda|^M }
\eeq
for any $M >0$ and $n$ large enough. Let $Y:= \Tr f (H^z) - \ee[ \Tr f (H^z)]$. We know that $\ee[ \Tr f(H^z) ] =  2 n \int \Re \log(x- \i \eta ) \rho^z (x) \d x + \O ( \gamma \log n )$ by \eqref{eq:explvarlog}, and so it suffices to prove the estimate for $Y$. For $Y$, we have
\begin{align} \label{eqn:fourier-1}
\ee[ F( Y) ] = \int_{\rr} \hat{F} ( \lambda) \ee \e^{ \i \lambda Y } \d \lambda =& \int_{ |\lambda| \leq n^{1/100} } \hat{F} ( \lambda) \ee \e^{ \i \lambda Y }  \d \lambda + \O ( n^{-2} ) \notag \\
= & \int_{ |\lambda| \leq n^{1/100} } \hat{F} ( \lambda) \e^{ - \frac{ \lambda^2}{2} V(f) } \d \lambda + \O ( n^{-1/5} ) \notag \\
= & \int_{ \rr } \hat{F} ( \lambda) \e^{ - \frac{ \lambda^2}{2} V(f) } \d \lambda + \O ( n^{-1/5} ).
\end{align}
The first and third lines use \eqref{eqn:F-decay} as this estimate implies,
\beq
\int_{ | \lambda| > n^{1/100}} | \hat{F} ( \lambda ) | \left( \left|  \ee[\e^{ \i \lambda Y} ] \right| + \e^{ - \frac{ \lambda^2}{2} V(f)} \right) \d \lambda \leq n^{-2} ,
\eeq
where we used the fact that $V (f) \geq 0$ by \eqref{eq:explvarlog}. The second line of \eqref{eqn:fourier-1} is a direct application of \eqref{eq:steinexpl} using that $\gamma$ is so small that $n^{200\gamma-1/4}\le n^{-1/5}$. The integral in the last line of \eqref{eqn:fourier-1} equals $\ee[ F(Z)]$ for a centered Gaussian random variable with variance $V(f) \asymp \gamma \log n$. In particular, 
\beq
\pp\left[ |Y| > C_1 \gamma \log n + 1 \right] \leq \ee[ (1- F) (Y)] = \ee[ (1- F) (Z)] + \O ( n^{-1/5} ).
\eeq
On the other hand,
\beq
 \left|  \ee[ (1- F) (Z) ]  \right| \leq \pp\left[ |Z| > C_1 \gamma \log n \right] \leq n^{ - 10 \gamma}
\eeq
if $C_1$ is taken sufficiently large. Above, the first inequality follows because $F(x) =1 $ for $|x| \leq C_1 \gamma \log n$. This yields the claim. \qed

\bel \label{lem:eta-inc} Let $\delta >0, \eps >0$ and $C_* >0$. 
Let $\eta_1 \leq \eta_2$ satisfy $\frac{  \log (n)^{1/2+\delta} }{n}\le \eta_1\le n^{-\eps}$ and $\eta_2 \leq ( \log n )^{C_*} \eta_1$. We have with overwhelming probability that,
\beq
\sup_{ \eta \in [\eta_1, \eta_2] } | \Psi_n (z, \eta) - \Psi_n (z, \eta_2 ) | \leq ( \log n )^{1/2+\delta} .
\eeq
\eel
\proof  In the complex i.i.d. case we have, 
\beq
\Psi_n(z, \eta) - \Psi_n (z, \eta_2) =2n  \int_{\eta}^{\eta_2} \Im \langle G^z ( \i u ) - M^z ( \i u ) \rangle \d u.
\eeq
By Lemma \ref{lem:precllaw}, the integral on the RHS is $\O ( (\log n)^{1/2+\delta/2 } )$ with overwhelming probability.

In the real $\beta=1$ case, there is an additional term in \eqref{eqn:psi-def} that is bounded by,
\beq \label{eqn:eta-dif-real-case}
\left| \log ( |z-\bar{z}|^2 + \eta  ) - \log( |z - \bar{z}|^2 +  \eta_2 )\right| \leq C \int_{\eta_1}^{\eta_2} \frac{1}{u} \d u \leq C \log \log n.
\eeq
The claim now follows. \qed

\vspace{3 pt}

\noindent{\bf Proof of Proposition \ref{prop:global}.} Recall $\eta_*=n^{-\gamma}$. By Lemma \ref{lem:eta-inc} it suffices to bound the max over $z$ with $\eta = \eta_*$ fixed. For an $\eps_1 >0$ we fix a set $P_1$ of $n^{\gamma+\eps_1}$-well spaced points of the disc $\{ z : |z| < r \}$. From Proposition \ref{prop:der-bd} we have that
\beq
\max_{ |z| < r } \Psi_n (z, \eta_* ) = \max_{ z \in P_1 } \Psi_n (z, \eta_*) + \O ( n^{-\eps_1/2} )
\eeq
with overwhelming probability. The claim now follows from a union bound, Lemma \ref{lem:global-tail} and taking $\eps_1 >0$ sufficiently small in terms of $\gamma$. \qed

\section{Upper bound of $\Psi_n (z)$ for complex i.i.d. matrices with Gaussian component} \label{sec:upper-1}

In this section we will prove the upper bound for complex i.i.d. matrices with a Gaussian component. The degree of precision in our upper bound will depend on the size of the Gaussian component.

\bep \label{prop:upper-gde-1} Let $0 < r < 1$. There are constants $c_1,C_1 >0$ so that the following holds. 
Let $\eps >0$,  and let $X$ be a complex i.i.d. matrix with Gaussian component of size $T = n^{-\eps}$. Then, for $n$ sufficiently large depending on $\eps$ and $r$, we have
\beq
\pp\left[ \max_{ |z| < r } \Psi_n (z, n^{-1}) \geq \left( \sqrt{2} + C_1 \eps \right) \log n \right] \leq n^{-c_1 \eps}.
\eeq
\eep
The proof of the above appears below in Section \ref{sec:upper-gde-proof}. We realize the matrix $X$ as the solution at time $T$ of the flow, 
\begin{equation}
\label{eq:matDBM}
\dif X_t=\frac{\dif B_t}{\sqrt{n}}, \qquad\quad X_0= (1- T)^{1/2} Y
\end{equation}
with $B_t$ being a matrix of i.i.d. standard complex Brownian motions, and $Y$ being a complex i.i.d matrix as in Definition \ref{def:model}. With this scaling the entries of $X_T$ have variance $1/n$.

Let $H_t^z = H^z (X_t)$ be the Hermitization of $X_t-z$ defined as in \eqref{eq:herm} with $X$ replaced with $X_t$, and define its resolvent by $G_t^z(w):=(H_t^z-w)^{-1}$, with $w\in\C\setminus \R$. By simple second order perturbation theory and the It{\^{o}} lemma (see e.g. \cite[Eq. (5.8)]{erdHos2012local}, \cite[Appendix B]{cipolloni2023central}), one can see that the eigenvalues of $H_t^z$, denoted by $\lambda_i^z=\lambda_i^z(t)$, are the solution of
\begin{equation}
\label{eq:DBMcompl}
\dif \lambda_i^z=\frac{\dif b_i^z}{\sqrt{2n}}+\frac{1}{2n}\sum_{j\ne i} \frac{1}{\lambda_i^z-\lambda_j^z}\dif t,
\end{equation}
with $b_{-i}^z=-b_i^z$ and $\lambda_{-i}^z=-\lambda_i^z$ as a consequence of the chiral symmetry of $H_t^z$.  Here, $b_i^z=b_i^z(t)$, with $i \in [n]$, is a family of independent standard Brownian motions.  Let $c_* (t) := \sqrt{1 + (t-T)}$. Since $X_t$ is a rescaling of an i.i.d. matrix, the limiting Stieltjes transform for $H_t^z$, denoted by $m_t^z$, is found by rescaling the function in \eqref{eq:mde} as,
\beq \label{eqn:mt-def}
m_t^z(w) :=\frac{1}{c_* (t)} m^{z/c_* (t)} (w/c_* (t)).
\eeq
We denote $\rho^z_t$ to be the measure associated to $m_t^z (w)$. With this definition we see that,
\beq \label{eq:evscfl}
\del_t m_t^z (w) = m_t^z(w) \del_w m_t^z (w).
\eeq
We now consider the evolution of $\sum_i \log (\lambda_i-w_t)$ along the characteristics of the above equation,
\begin{equation}
\label{eq:char}
\partial_t w_t=-m_t^{z_0}(w_t), \qquad  z_t = z_0,
\end{equation}
i.e., unlike in Section~\ref{sec:dynamical-local-law}, we now move only $w_t$ and not $z_t$. Note that along the characteristics \eqref{eq:char} we have
\beq
\label{eq:consttime}
\partial_t m_t^z(w_t)=(\partial_t m_t^z)(w_t)+(\partial_w m_t^z)(w_t)\partial_t w_t=0,
\eeq
which follows from \eqref{eq:evscfl}--\eqref{eq:char}. We point out that, by standard ODE theory (see the proof of \cite[Lemma 5.2]{cipolloni2023mesoscopic}), if we fix $w\in \C, T>0$, then there exists $w_0$ such that $|\Im w_0|\asymp T$ and the solution $w_t$ of \eqref{eq:char}, with initial condition $w_0$, is such that $w_T=w$. In this section we will only consider characteristics of the form $w_s = \i \eta_s$, and we use this notation extensively. Note that by the last point in Lemma~\ref{lem:proprho}, together with $T\ll 1$, we have $- \del_s \eta_s \asymp 1$. 

\begin{lemma}
\label{lem:charup}
Let $\lambda_i^z(t)$ be the eigenvalues of $H_t^z$. Let $\xi >0$ and let $T = n^{-\xi}$. Consider a characteristic $w_s = \i \eta_s$ such that $\eta_T = (\log n)^{C_*} /n$, for some  $C_* \geq 10$. We have with overwhelming probability that,
\begin{equation}
\begin{split} \label{eqn:aa-2}
&\sum_i \log (\lambda_i^z (T)-w_T)-2n\int_\R  \log (x-w_T)\rho_T^z(x) \\
&\qquad\quad=\sum_i \log (\lambda_i^z(0)-w_0)-2 n\int_\R  \log (x-w_0)\rho_0^z(x)\, \dif x+\xi_{n,T}+\mathcal{O}\left(\frac{(\log n)^{5}}{n|\eta_T|}\right),
\end{split}
\end{equation}
for a complex Gaussian random variable $\xi_{n,T}$. Furthermore, we have,
\begin{equation}
\label{eq:varxinT}
\mathrm{Var}(\Re \xi_{n,T})=\log\big|\eta_0/\eta_T\big|+\mathcal{O}\left(T+\frac{\log n}{n|\eta_T|}\right).
\end{equation}
\end{lemma}

The proof of the above lemma appears below in Section \ref{sec:upper-lemma}. 
Recall now our three-parameter version of $\Psi_n (z,  t, \eta)$ given by \eqref{eqn:psi-def}. The above quickly implies the following.

\bep \label{prop:upper-gde-2}
Let $\{ X_t \}_{ 0 \leq t \leq T}$ be as in \eqref{eq:matDBM}, let $\eta_1 = (\log n)^{C_*} / n$, and let $T = \eta_2 = n^{-\gamma}$ for some $\gamma  < 1/10$. There is a $c>0$ so that the following holds. Let $0 < r < 1$ and $\eps >0$. Then,
\begin{align}
\max_{ |z| \le r } \Psi_n (z, T, \eta_1) \leq \max_{ |z| \le r, (\log n)^{-1} \eta_2 \leq \eta \leq \eta_2 \log n } \Psi_n (z, 0 , \eta) + \left( \sqrt{2} + \eps \right) \log n.
\end{align}
with probability at least $ 1 - n^{-c \eps}$, for all $n$ sufficiently large depending on $r, \gamma$, and $\eps$. 
\eep
\proof Let $P_1$ be a grid of $n^{1+\eps}$ well-spaced points of $\{ z : |z| < r \}$. By Proposition \ref{prop:der-bd} we have that,
\beq
\max_{ |z| < r } \Psi_n (z, T, \eta_1)   = \max_{ z \in P_1} \Psi_n (z, T ,\eta_1) + \O (n^{-\eps/2} )
\eeq
with overwhelming probability. For any $z \in P_1$ we consider the characteristic $w_t = \i \eta_t$ with $\eta_T = \eta_1$. Then $\eta_0 \asymp T$. Let $Y_z = \Re[ \xi_{n, T} ]$ where $\xi_{n, T}$ is the Gaussian random variable from Lemma  \ref{lem:charup}. We therefore have,
\beq
\max_{ z \in P_1} \Psi (z, \eta_1, T) \leq \max_{ |z| < r , (\log n)^{-1} \eta_2 \leq \eta \leq \eta_2 \log n } \Psi_n (z, \eta, 0)  + \max_{ z \in P_1} Y_z + \O (1)
\eeq
with overwhelming probability. 
But since the variance of each $Y_z$ is bounded by $\log n+C$ we see that by a union bound,
\beq
\pp\left[ \max_{z \in P_1} Y_z > \left( \sqrt{2} + 10 \eps \right) \log n \right] \leq n^{-\eps}.
\eeq
for all $n$ sufficiently large. 
The claim now follows. \qed

\subsection{Proof of Lemma \ref{lem:charup}}

\label{sec:upper-lemma}

Let $w_t = \i \eta_t$ be as in the statement of  the lemma. We first compute the evolution of the $\log$--determinant for  fixed $w$ using It\^{o}'s formula,
\begin{equation}
\label{eq:1complrel}
\dif \sum_i \log (\lambda_i^z-w)=\frac{1}{\sqrt{2n}}\sum_i \frac{\dif b_i^z}{\lambda_i^z-w}+\frac{1}{2n}\sum_{j\ne i}\frac{1}{(\lambda_i^z-w)(\lambda_i^z-\lambda_j^z)}\d t-\frac{1}{4n}\sum_i \frac{1}{(\lambda_i^z-w)^2}\d t.
\end{equation}
Symmetrizing the $i,j$--summation, we get
\begin{equation}
\begin{split}
\label{eq:formnochar}
\dif \sum_i \log (\lambda_i^z-w)&=\frac{1}{\sqrt{2n}}\sum_i \frac{\dif b_i^z}{\lambda_i^z-w}-\frac{1}{4n}\left(\sum_i\frac{1}{\lambda_i^z-w}\right)^2\d t.
\end{split}
\end{equation}

Next, using \eqref{eq:formnochar} as an input, we consider the evolution of the $\log$--determinant along the characteristics $w_t$ from \eqref{eq:char}: 
%Note that along the characteristics \eqref{eq:char} we have $\partial_t m_t(w_t)=0$. We thus obtain
\begin{equation}
\label{eq:charappr}
\dif \sum_i \log (\lambda_i^z-w_t)=\frac{1}{\sqrt{2n}}\sum_i \frac{\dif b_i^z}{\lambda_i^z-w_t}-\sum_i\frac{1}{\lambda_i^z-w_t}\left(\frac{1}{4n}\sum_i\frac{1}{\lambda_i^z-w_t}-m_t^z(w_t)\right)\d t.
\end{equation}
Then, subtracting the deterministic approximation in the LHS of \eqref{eq:charappr}, we get
\begin{equation}
\label{eq:somestep}
\begin{split}
\dif \left[\sum_i \log (\lambda_i^z-w_t)-2n\int_\R  \log (x-w_t)\rho_t^z(x)\,\dif x\right] &=\frac{1}{\sqrt{2n}}\sum_i \frac{\dif b_i^z}{\lambda_i^z-w_t}-n\big[\langle G_t^z(w_t)\rangle-m_t^z(w_t)\big]^2\d t \\
&\quad-2n\int_\R\log(x-w_t)\partial_t\rho_t^z(x)\, \dif x-nm_t(w_t)^2.
\end{split}
\end{equation}
We now show that the last line of \eqref{eq:somestep} is equal to zero. By \eqref{eq:evscfl} we have that $\pi\partial_t \rho_t(x)=\partial_x\Im [m_t(x)^2]/2$. Then, using integration by parts, we have
\begin{equation}
\begin{split}
-2\int_\R\log(x-w_t)\partial_t\rho_t^z(x)\, \dif x&=\frac{1}{\pi}\int_\R \frac{\Im[m_t(x)^2]}{x-w_t}\,\dif x \\
&=\frac{1}{2\pi\ii}\int_\R\left[\frac{m_t(x)^2}{x-w_t}-\frac{\overline{m_t}(x)^2}{x-w_t}\right]\, \dif x=m_t(w_t)^2,
\end{split}
\end{equation}
where in the last equality we used residue theorem together with $\Im w_t>0$ (which makes the integral with $\overline{m_t}(x)$ equal to zero). This shows that the last line in \eqref{eq:somestep} is equal to zero and so,
\beq \label{eqn:somestep-2}
\dif \left[\sum_i \log (\lambda_i^z-w_t)-2n\int_\R  \log (x-w_t)\rho_t^z(x)\,\dif x\right] =\frac{1}{\sqrt{2n}}\sum_i \frac{\dif b_i^z}{\lambda_i^z-w_t}-n\big[\langle G_t^z(w_t)\rangle-m_t^z(w_t)\big]^2 \d t.
\eeq
We rewrite the martingale term in \eqref{eqn:somestep-2} as
\beq
\frac{1}{\sqrt{2n}}\sum_i \frac{\dif b_i^z}{\lambda_i^z(t)-w_t} = \frac{1}{\sqrt{2n}}\sum_i \frac{\dif b_i^z}{\gamma_i^z(t)-w_t} + \left( \frac{1}{\sqrt{2n}}\sum_i \frac{\dif b_i^z}{\lambda_i^z(t)-w_t} - \frac{1}{\sqrt{2n}}\sum_i \frac{\dif b_i^z}{\gamma_i^z(t)-w_t} \right)
\eeq
Here, $\gamma_i^z(t)$ are the $n$-quantiles of the measure $\rho_t^z$ (see e.g. the definition in \eqref{eq:defquantz}). 
We now bound the quadratic variation of the second term. With $ c= (100)^{-1}$, using the rigidity esimates in \eqref{eq:goodrig} for $|i| \leq n^{1-c}$ and the estimates \eqref{eqn:usual-rigidity} for the other terms, we have,
\begin{equation}
\begin{split}
\label{eq:estimatequantevalues}
\int_0^T \frac{1}{n} \sum_{i} \left| \frac{1}{ \lambda_i^z(t) - w_t} - \frac{1}{ \gamma_i^z(t) - w_t } \right|^2 \d t &\leq C \int_0^T \frac{1}{n} \sum_{ |i| < n^{1-c} } \frac{ | \lambda_i^z(t) - \gamma_i^z(t)|^2}{ | \gamma_i^z(t) - w_t |^4} \d t + n^{-3/2} \\
&\leq C  \int_0^T \frac{(\log n)^{4}}{n^3} \sum_{|i| < n^{1-c} } \frac{1}{ | \gamma_i^z(t) - w_t|^4} \d t + n^{-3/2} \\
&\leq C \int_0^T \frac{(\log n)^{4}}{ n^2 \eta_t^3} \d t + n^{-3/2} \leq C \frac{(\log n)^{4}}{ (n \eta_T)^2} 
\end{split}
\end{equation}
with overwhelming probability. Therefore, by the Burkholder-Davis-Gundy (BDG) inequality,
\begin{equation}
\label{eq:appr}
\sup_{0\le t\le T}\left|\int_0^t\left[\frac{1}{\sqrt{2n}}\sum_i \frac{\dif b_i(t)}{\lambda_i-w}-\frac{1}{\sqrt{2n}}\sum_i \frac{\dif b_i(t)}{\gamma_i-w}\right]\right|\lesssim  \frac{(\log n)^{5}}{n \eta_T} 
\end{equation}
with overwhelming probability.

Applying the local law \eqref{eq:goodll} to the second term in \eqref{eq:somestep} we have,
\beq \label{eqn:char-inter-1}
n \int_0^T \left[ \langle G^z (w_s) \rangle - m^z (w_s) \right]^2 \d s = \O \left( \frac{ (\log n)^4}{n \eta_T} \right)
\eeq
with overwhelming probability. Therefore, we conclude, 
\begin{equation}
\begin{split}
\label{eq:apprgaussbi}
\Psi_n (z, T, \eta_T) - \Psi_n (z, 0, \eta_0) = \frac{1}{\sqrt{2n}}\int_0^T\sum_i \frac{\dif b_i(t)}{\gamma_i^z(t)-w_t}\,+\mathcal{O}\left(\frac{  ( \log n)^{5}}{n \eta_T} \right),
\end{split}
\end{equation}
with overwhelming probability. This proves \eqref{eqn:aa-2}, after defining,
\begin{equation}
\label{eq:gauss}
\xi_n:=\frac{1}{\sqrt{2n}}\int_0^T\sum_i \frac{\dif b_i(t)}{\gamma_i-w_t}.
\end{equation}
We now compute the variance of the real part of $\xi_n$,
\begin{equation}
\label{eq:varcomp}
\begin{split}
\mathrm{Var}(\Re \xi_n)&=\frac{1}{n}\int_0^T\sum_i \frac{(\gamma_i^z(t))^2}{|\gamma_i^z(t)-\ii\eta_t|^4}\, \dif t =\frac{1}{2n}\int_0^T\sum_i \frac{1}{|\gamma_i^z(t)-\ii\eta_t|^2}\, \dif t + \Re \frac{1}{2n} \int_0^T \sum_i\frac{1}{ ( \gamma_i^z(t) - \i \eta_t )^2} \d t \\
&=\int_0^T\frac{\Im m^{z}_t(\ii\eta_t)}{\eta_t}\, \dif t+\mathcal{O}\left(T+\frac{\log n}{n\eta_T}\right) =-\int_0^T\frac{1}{\eta_t}\, \dif \eta_t+\mathcal{O}\left(T+\frac{\log n}{n\eta_T}\right) \\
&=\log\left(\frac{\eta_0}{\eta_T}\right)+\mathcal{O}\left(T+\frac{\log n}{n\eta_T}\right)
\end{split}
\end{equation}
In the third equality we replaced the sum over the quantiles with an integral against $\rho^z (x)$, at the price of a negligible error, and used the fact that $\partial_w m_t^z (w)$ is bounded near the imaginary axis. This completes the proof. \qed

\subsection{Proof of Proposition \ref{prop:upper-gde-1} } \label{sec:upper-gde-proof}

Before proving Proposition \ref{prop:upper-gde-1} we first prove the following preliminary lemma.

\begin{lemma}
\label{lem:reglog} For both real or complex i.i.d. matrices the following holds. 
For any  $C_* \geq 1$ and any $\delta >0$  it holds that with overwhelming probability,
\beq \label{eq:linmeslog}
\Psi_n (z, n^{-1} ) \leq \Psi_n (z, (\log n)^{C_*} n^{-1} ) + (\log n)^{1/2+\delta}.
\eeq
\end{lemma}
\proof We first consider the complex $\beta=2$ case. For any $\eta_1 < \eta_2$ we have that,
\begin{align} \label{eqn:a-1}
\left| \Psi_n (z, \eta_2 ) - \Psi_n (z, \eta_1) \right| \leq  \int_{\eta_1}^{\eta_2} n | \Im \langle G^z ( \i \eta ) - M^z ( \i \eta ) \rangle | \d \eta.
\end{align}
We first apply this with $\eta_1 = n^{-1}$ and $\eta_2 = (\log n)^{1/2+\delta} n^{-1}$. Since $y \to y \Im \langle G(\i y) \rangle$ is an increasing function, we see that, by applying \eqref{eq:goodll} at $\eta = \eta_2$, we have with overwhelming probability,
\beq
\int_{n^{-1}}^{(\log n)^{1/2+\delta} n^{-1}} n | \Im \langle G^z ( \i \eta ) - M^z ( \i \eta ) \rangle | \d \eta \leq (\log n)^{1/2+\delta} \int_{n^{-1}}^{(\log n)^{1/2+\delta} n^{-1}} C \eta^{-1} \d \eta \leq C (\log n)^{1/2+2\delta}.
\eeq
Then, applying \eqref{eqn:a-1} with $\eta_1 = (\log n)^{1/2+\delta} n^{-1}$ and $\eta_2 = (\log n)^{C_*} n^{-1}$ and using \eqref{eq:goodll} we have,
\beq
\int^{n^{-1} (\log n)^{C_*}}_{(\log n)^{1/2+\delta} n^{-1}} n | \Im \langle G^z ( \i \eta ) - M^z ( \i \eta ) \rangle | \d \eta \leq (\log n)^{1/2+\delta} \int^{n^{-1} (\log n)^{C_*}}_{(\log n)^{1/2+\delta} n^{-1}}  \eta^{-1} \d \eta \leq (\log n)^{1/2+2 \delta}
\eeq
and the claim follows in the complex case.  In the real case, there is an additional deterministic term which is bounded by $C \log \log n$ using the exact same argument as in \eqref{eqn:eta-dif-real-case}.  \qed

\vspace{5 pt}

\noindent{\bf Proof of Proposition \ref{prop:upper-gde-1}}. Lemma \ref{lem:reglog} reduces this to proving the upper bound at $\eta = (\log n)^{C_*} / n$. The upper bound for this quantity is then an immediate consequence of Proposition \ref{prop:upper-gde-2}, with $\gamma=\eps$, and Proposition \ref{prop:global}. \qed

\section{Upper bound in the GDE real case}

In this section we prove the upper bounds of Theorem~\ref{theo:realmaxlog} for real i.i.d. matrices having an almost order one Gaussian component. The following is the analog of Proposition \ref{prop:upper-gde-1}. 
\bep
\label{prop:upper-real-gde} Let $0 < r < 1$. There are constants $C_1, c_1 >0$ so that the following holds. Let $\eps >0$ and let $X_T$ be a real i.i.d. GDE matrix with Gaussian component of size at least $T = n^{-\eps}$. Then,
\beq
\pp\left[ \max_{ |z| \le r } \Psi_n (z, n^{-1} ) \geq \left( \sqrt{2} + C_1 \eps \right) \log n \right] \leq n^{-c_1 \eps}
\eeq
for $n$ sufficiently large, depending on $\eps$ and $r$.
\eep
We let $X_t$ solve \eqref{eq:matDBM}, but now $B_t$ is a matrix of i.i.d. standard real Brownian motions, and $X = (1- T)^{1/2} Y$ for a real i.i.d. matrix $Y$. We continue to denote $H_t^z$ the Hermitization of $X_t -z$, and use the notation $m_t^z(w)$, etc., as defined at the beginning of Section \ref{sec:upper-1}. We also recall our three parameter function $\Psi_n (z, t, \eta)$ as in \eqref{eqn:psi-def}, which now has the additional deterministic term compared to the complex case. 

The first step in Section~\ref{sec:upper-1} for the complex case was to write the log--determinant on small scales as the one on (almost) global scales plus a Gaussian term using the characteristics method (see Lemma~\ref{lem:charup}). We now replace this step with the following lemma. Notice that compared to \eqref{eqn:aa-2} we now write \eqref{eqn:real-char-main-est} for all intermediate mesoscopic scales; this will be useful in the analysis of Section~\ref{sec:mesoreal} below.

\bel \label{lem:real-char}
Let $\lambda_i^z(t)$ be the eigenvalues of $H_t^z$. Let $\xi, \eps_1, \eps_2>0$ be sufficiently small. Let $T = n^{-\eps_1}$. Consider a characteristic $w_s = \i \eta_s$ such that $\eta_T \geq  n^{\eps_2-1}$. Let $S_1,S_2$ satisfy $0\leq S_1< S_2 \leq T$. Then, uniformly in $S_1, S_2$ we have, with overwhelming probability,
\begin{align} \label{eqn:real-char-main-est}
& \sum_i \log ( \lambda_i^z (S_2) - w_{S_2} ) - 2 n \int_{\rr} \log (x - w_{S_2}) \rho^z_{S_2} (x) -  E_n (S_1,S_2) \notag \\
&\qquad\quad= \sum_i \log ( \lambda_i^z (S_1) - w_{S_1} ) - 2 n \int_{\rr} \log (x - w_{S_1}) \rho^z_{S_1} (x) + \xi_n(S_1,S_2) + \O \left(\frac{n^{\xi}}{n \eta_{S_2}} \right)
\end{align} 
where for fixed $S_1$, the process  $\{ \xi_n (S_1, S_2) \}_{S_2 \in [S_1, T]}$ is a complex martingale, and
\beq
\label{eq:compes1s2}
E_n(S_1,S_2)=\frac{1}{2}\log\left(\frac{|z-\overline{z}|^2+2\eta_{S_1}\sqrt{1-|z|^2}}{ |z-\overline{z}|^2+2\eta_{S_2}\sqrt{1-|z|^2}}\right)+\O\big(S_2\log n\big).
\eeq
Furthermore, there exists a coupling such that with overwhelming probability we have for all $s$ satisfying $S_1 \leq s \leq T$ that
\begin{equation}
\label{eq:martrepxis1s2}
\Re \left[ \xi_n(S_1,s) \right] =\int_{S_1}^{s} V_u^{1/2}\dif \tilde{b}_u+\O\left(\frac{n^\xi}{\sqrt{n\eta_{s}}}\right),
\end{equation}
for a standard Brownian motion $\tilde{b}_u$ and some explicit deterministic $V_u$ such that
\beq
\label{eq:Vtexplicit}
\int_{S_1}^{s} V_u\,\dif u=\log\left(\frac{\eta_{s}}{\eta_{S_1}}\right)+\log\left(\frac{|z-\overline{z}|^2+2\eta_{S_1}\sqrt{1-|z|^2}}{ |z-\overline{z}|^2+2\eta_{s}\sqrt{1-|z|^2}}\right)+\O\left(S_2\log n+\frac{\log n}{n\eta_{s}}\right).
\eeq
\eel

\proof  \, Some parts of this proof are similar to the proof of Lemma~\ref{lem:charup}, and for this reason we focus on the main differences.
 
By Proposition 7.6 and Appendix B of \cite{cipolloni2021fluctuation} the eigenvalues $\lambda_i^z(t)$ of $H_t^z$ are the unique strong solution of
\begin{equation}
\label{eq:realDBM}
\dif \lambda_i^z(t)=\frac{\dif b_i^z(t)}{\sqrt{2n}}+\frac{1}{2n}\sum_{j\ne i}\frac{1+\Lambda_{ij}^z(t)}{\lambda_j^z(t)-\lambda_i^z(t)}\,\dif t.
\end{equation}
The driving martingales $b_i^z(t)$ have the following covariation process
\begin{equation}
\dif [b_i^z(t),b_j^z(t)]=\big[\delta_{ij}-\delta_{i,-j}+\Lambda_{ij}^z(t)\big]\dif t, \qquad\quad  \Lambda_{ij}^z(t):=4\Re\big[\langle \bm{w}_i^z(t),E_1\bm{w}_j^{\overline{z}}(t)\rangle\langle \bm{w}_j^{\overline{z}}(t),E_2\bm{w}_i^z(t)\rangle\big].
\end{equation}
Here $\{{\bm w}_i^z(t)\}_i$ denote the eigenvectors of the Hermitization of $X_t-z$. Note that $\Lambda_{ij}^z(t)=\Lambda_{ji}^z(t)$. We also point out that if $z\in \R$, then $\Lambda_{i,j}^z(t)=0$, for $j\ne\pm i$, and $\Lambda_{i,\pm i}^z(t)=\pm 1$, i.e. for $z\in \R$ there is no repulsion from zero in \eqref{eq:realDBM}.

Proceeding similarly to \eqref{eq:1complrel}--\eqref{eq:apprgaussbi}, using \eqref{eq:realDBM} instead of \eqref{eq:DBMcompl}, for any $0\le S_1<S_2\le T$, we obtain
\begin{equation} \label{eqn:real-upper-intermediate}
\begin{split}
&\int_{S_1}^{S_2} \dif   \left[\sum_i \log (\lambda_i^z-w_t)-2n\int_\R  \log (x-w_t)\rho_t^z(x)\,\dif x\right] \\
&\qquad\quad =\frac{1}{\sqrt{2n}}\int_{S_1}^{S_2}\sum_i \frac{\dif b_i^z(t)}{\lambda_i^z(t)-w_t}+\tilde{\sum}_{ij}\int_{S_1}^{S_2}\langle G_t^z(w_t)E_i G_t^{\overline{z}}(w_t) E_j\rangle \dif t+\mathcal{O}\left(\frac{\log n}{n\eta_{S_2}}\right).
\end{split}
\end{equation}
Here $\tilde{\sum}_{ij}$ is defined below \eqref{eq:defE1E2}, and we recall that it denotes a summation over $(i,j)\in \{(1,2),(2,1)\}$. Next, we use that by rescaling Corollary~\ref{cor:nozreal} (i.e. the entries have variance $c_* (t)/n$ instead of $1/n$) we have
\begin{equation}
\label{eq:usenow}
\big|\langle \big(G_t^z(\ii\eta_t)E_i G_t^{\overline{z}}(\ii\eta_t)-M_t^{z,\overline{z}}(\ii\eta_t,E_i,\ii\eta_t)\big) E_j\rangle\big|\lesssim \frac{n^\xi}{n\eta_t^2}.
\end{equation}
Here for any $z_1,z_2\in \C$ and any matrix $A\in\C^{2n\times 2n}$, we defined
\begin{equation}
\label{eq:realdefm12}
M_t^{z_1,z_2}(\ii\eta_1,A,\ii\eta_2):=\big(1-c_*(t)^2M_t^{z_1}(\ii\eta_1)\mathcal{S}[\cdot]M_t^{z_2}(\ii\eta_2)\big)^{-1}\big[M_t^{z_1}(\ii\eta_1)A_1M_t^{z_2}(\ii\eta_2)\big]
\end{equation}
with the \emph{covariance operator} $\mathcal{S}:\C^{2n\times 2n}\times \C^{2n \times 2n}$ defined by
\[
\mathcal{S}[\cdot]:=2\langle \cdot E_1\rangle E_2+2\langle \cdot E_2\rangle E_1 .
\]
The constant $c_*(t)$ is defined below \eqref{eq:DBMcompl}, and
\begin{equation}
M_t^z(w):=\left(\begin{matrix}
m_t^z(w) & -zu_t^z(w) \\
-\overline{z}u_t^z(w) & m_t^z(w)
\end{matrix}\right), \qquad\quad u_t^z(w):=\frac{m_t^z(w)}{w+c_*(t)^2m_t(w)}.
\end{equation}
By \eqref{eq:usenow}, we thus get
\begin{equation}
\begin{split}
\label{eq:goodpoint}
&\int_{S_1}^{S_2}  \dif \left[\sum_i \log (\lambda_i^z-w_t)-2n\int_\R  \log (x-w_t)\rho_t^z(x)\,\dif x\right] \\
&\qquad\quad=\frac{1}{\sqrt{2n}}\int_{S_1}^{S_2}\sum_i \frac{\dif b_i^z(t)}{\lambda_i^z(t)-w_t}+\tilde{\sum}_{ij}\int_{S_1}^{S_2}\langle M_t^{z,\overline{z}}(\ii\eta_t, E_i, \ii\eta_t) E_j\rangle \dif t+\mathcal{O}\left(\frac{n^{\xi}}{n\eta_{S_2}}\right).
\end{split}
\end{equation}
To conclude the proof, we need to compute the leading order approximation of the deterministic term in the RHS of \eqref{eq:goodpoint} and write the martingale term (at leading order) as an integral of a rescaled Brownian motion using the martingale representation theorem.

We start with the computation of the deterministic term. By \eqref{eq:evscfl}--\eqref{eq:char} it follows that $m_t^z(w_t)$ and $u_t^z(w_t)$ are constant in time. Define the short--hand notations $c_t:=c_*(t)^2=1+(t-T)$, $m:=m_0(\ii \eta_0)=m_t(\ii\eta_t)$, and $u:=u_0^z(\ii\eta_0)=u_t^z(\ii\eta_t)$. By an explicit computation we obtain
\beq \label{eqn:real-det-1}
\tilde{\sum}_{ij} \int_{S_1}^{S_2}\langle M_t^{z,\overline{z}}(\ii\eta_t,E_i,\ii\eta_t)E_j\rangle \dif t = \int_{S_1}^{S_2}\frac{c_tu^2\Re[z^2]-c_t^2|z|^4u^4+c_t^2m^4}{1+c_t^2|z|^4u^4-c_t^2m^4-2c_tu^2\Re[z^2]} \dif t 
\eeq

Then, a straightforward but somewhat tedious computation shows that,
\beq \label{eqn:real-det-2}
1 + c_t^2 |z|^4 u^4 - c_t^2 m^4 - 2 c_t u^2 \Re[z^2] = ( 1 + \O ( \eta_t) ) |z-\bar{z}|^2 + 2 \eta_t (1 + \O ( \eta_t) ) \sqrt{1-|z|^2}
\eeq
using
\begin{equation}
\label{eq:exapnsmu}
u=1-\frac{\eta_T}{\sqrt{1-|z|^2}}+\mathcal{O}(\eta_T^2), \qquad\quad m=\ii\sqrt{1-|z|^2}+\ii\frac{\eta_T(2|z|^2-1)}{2(1-|z|^2)}+\mathcal{O}(\eta_T^2),
\end{equation}
and that $\eta_t=\eta_T+(T-t)\Im m$, for any $0\le t\le T$. With this, we then have that,
\begin{equation}
\begin{split}
\label{eq:impexpcomp}
&\tilde{\sum}_{ij} \int_{S_1}^{S_2}\langle M_t^{z,\overline{z}}(\ii\eta_t,E_i,\ii\eta_t)E_j\rangle \dif t \\
&\qquad\qquad\qquad\quad=\int_{S_1}^{S_2}\frac{c_tu^2\Re[z^2]-c_t^2|z|^4u^4+c_t^2m^4}{1+c_t^2|z|^4u^4-c_t^2m^4-2c_tu^2\Re[z^2]} \dif t \\
&\qquad\qquad\qquad\quad= -\frac{1}{2}\int_{S_1}^{S_2} \partial_t \log\big[1+c_t^2|z|^4u^4-c_t^2m^4-2c_tu^2\Re[z^2]\big]\,\dif t+\mathcal{O}(S_2 \log n)\\
&\qquad\qquad\qquad\quad=-\frac{1}{2}\log\left[\frac{1+c_{S_2}^2|z|^4u^4-c_{S_2}^2m^4-2c_{S_2}u^2\Re[z^2]}{1+c_{S_1}^2|z|^4u^4-c_{S_1}^2m^4-2c_{S_1}u^2\Re[z^2]}\right]+\mathcal{O}(S_2 \log n)\\
&\qquad\qquad\qquad\quad=\frac{1}{2}\log\left(\frac{[1+\mathcal{O}(\eta_{S_1})]\cdot|z-\overline{z}|^2+2\eta_{S_1}\cdot[1+\mathcal{O}(\eta_{S_1})]\sqrt{1-|z|^2}}{[1+\mathcal{O}(\eta_{S_2})]\cdot |z-\overline{z}|^2+2\eta_{S_2}\cdot [1+\mathcal{O}(\eta_{S_2})]\sqrt{1-|z|^2}}\right)+\mathcal{O}(S_2\log n) \\
&\qquad\qquad\qquad\quad=\frac{1}{2}\log\left(\frac{|z-\overline{z}|^2+2\eta_{S_1}\sqrt{1-|z|^2}}{ |z-\overline{z}|^2+2\eta_{S_2}\sqrt{1-|z|^2}}\right)+\mathcal{O}(S_2\log n),
\end{split}
\end{equation}
This concludes the computation of $E_n(S_1,S_2)$ in \eqref{eq:compes1s2}.

Next, we consider the martingale term in the RHS of \eqref{eq:goodpoint}. The martingale $\xi_n (S_1, s)$ is defined by,
\[
\xi_n(S_1,s):=\frac{1}{\sqrt{2n}} \int_{S_1}^{s} \sum_i\frac{\dif b_i^z(t)}{\lambda_i^z(t)-w_t}.
\]
We now compute the quadratic variation process of $\Re[ \xi_n (S_1, s) ]$:
\begin{equation}
\begin{split}
\label{eq:quadvarcomp}
& \dif \left[\Re\frac{1}{\sqrt{2n}}\sum_i\frac{\dif b_i^z(t)}{\lambda_i^z(t)-w_t}, \Re\frac{1}{\sqrt{2n}}\sum_i \frac{\dif b_i^z(t)}{\lambda_i^z(t)-w_t}\right]\\
=& \bigg[\frac{1}{n}\sum_i \frac{(\lambda_i^z)^2}{|\lambda_i^z-\ii\eta_t|^4}+2\tilde{\sum}_{ij}\langle \Re G_t^z(\ii\eta_t)E_i\Re G_t^{\overline{z}}(\ii\eta_t)E_j\rangle\bigg]\, \dif t \\
= &\bigg[\frac{1}{n}\sum_i \frac{(\lambda_i^z)^2}{|\lambda_i^z-\ii\eta_t|^4}+2\tilde{\sum}_{ij}\langle G_t^z(\ii\eta_t)E_i G_t^{\overline{z}}(\ii\eta_t)E_j\rangle\bigg]\, \dif t \\
 = &\bigg[\frac{1}{n}\sum_i \frac{(\gamma_i^z)^2}{|\gamma_i^z-\ii\eta_t|^4}+2\tilde{\sum}_{ij}\langle M_t^{z,\overline{z}}(\ii\eta_t,E_i,\ii\eta_t)E_j\rangle\bigg]\, \dif t +\O\left(\frac{n^\xi}{n\eta_t^2}\right)\,\dif t,
\end{split}
\end{equation}
where to go from the first to the second line we used that, by the symmetry of the spectrum of $H^z$, $\Im G^z$ is diagonal, $\Re G^z$ is off--diagonal, and $E_i\Im G^zE_j=0$ for $i\ne j$. Additionally, in the last equality we used similar estimates to \eqref{eq:estimatequantevalues} to compute the deterministic approximation of the first term and \eqref{eq:usenow} to compute the deterministic approximation of the second term.

Denote
\begin{equation}
\label{eq:defvt}
V_t:=\frac{1}{n}\sum_i \frac{(\gamma_i^z)^2}{|\gamma_i^z-\ii\eta_t|^4}+2\tilde{\sum}_{ij}\langle M^{z,\overline{z}}(\ii\eta_t,E_i,\ii\eta_t)E_j\rangle \asymp \frac{1}{ \eta_t} .
\end{equation}
Here, the estimate follows in a straightforward manner from the computations in \eqref{eq:varcomp}, \eqref{eqn:real-det-1} and \eqref{eqn:real-det-2}. Then, by the martingale representation theorem together with \eqref{eq:quadvarcomp}, we write (after passing to possibly a larger probability space),
\begin{equation}
\Re\frac{1}{\sqrt{2n}}\sum_i\frac{\dif b_i^z(t)}{\lambda_i^z(t)-w_t}=\left[V_t+\O\left(\frac{n^\xi}{n\eta_t^2}\right)\right]^{1/2}\dif \tilde{b}_t,
\end{equation}
with $\tilde{b}_t$ being a standard real Brownian motion. We now define
\begin{equation}
\dif \mathcal{X}_t:=V_t^{1/2}\dif \tilde{b}_t.
\end{equation}
By the BDG inquality we then have with overwhelming probability,
\begin{equation}
\Re\xi_n(S_1,s)=\Re\frac{1}{\sqrt{2n}} \int_{S_1}^{s} \sum_i\frac{\dif b_i^z(t)}{\lambda_i^z(t)-w_t}=\int_{S_1}^{s}\dif \mathcal{X}_t+\O\left(\frac{n^\xi}{\sqrt{n\eta_{s}}}\right),
\end{equation}
which shows \eqref{eq:martrepxis1s2} for $V_t$ being defined as in \eqref{eq:defvt}. Finally, we conclude the proof noticing that by \eqref{eq:varcomp} and \eqref{eq:impexpcomp} we compute the leading order approximation of the first and second term in the definition of $V_t$, respectively, and obtain \eqref{eq:Vtexplicit}.

\qed

Notice that, unlike in the complex case (cf. Lemma~\ref{lem:charup}), the variance of the martingale term in Lemma~\ref{lem:real-char} (see \eqref{eq:martrepxis1s2}--\eqref{eq:Vtexplicit}) depends on  the size of $|z-\overline{z}|=2|\Im z|$. For this reason, in the reminder of this section we divided the analysis of $\Psi_n(z)$ into three regimes according to the size of $|\Im z|$. We will first record an estimate that will be useful both here and later in the paper.

\subsection{Short time increment bound}

The following is a sub-optimal estimate controlling the process $\Psi_n (z, t, \eta_t)$ over short time intervals. It will be mostly used for passing from $(\log n)^C/n$ scales to $n^{\eps}/n$ scales.  In the following, we let $X_s$ solve $\d X_s = \d B_s/\sqrt{n}$ where $B_s$ is a matrix of either complex or real standard Brownian motions, with $X_0 = (1-T)^{1/2} Y$ for $Y$ an i.i.d. matrix, and $T \leq n^{-c}$ for some $c>0$. If $B_s$ is complex, we will allow $Y$ to be either a real or complex i.i.d. matrix. If $B_s$ is real then we will assume that $Y$ is real. We let $\Psi_n (z, s, \eta)$ be as in \eqref{eqn:psi-def} where $\beta=1, 2$ corresponds to the real or complex dynamics driving by $B_s$, respectively. 

\bep \label{prop:short-time-int} 
There is a $c_1>0$ so that the following holds.  Fix $T$ and $X_s$ as above, and let $0 \leq t \leq T$. Let $\eps_1 >0$ be sufficiently small. Let $ \eta_s >0$ be a characteristic, i.e. a solution of \eqref{eq:char} for $w_s=\ii\eta_s$, such that $(\log n)^{10} / n \leq \eta_t \leq n^{-\eps_1}$ and,
\beq
\log ( \eta_0 / \eta_t ) \leq \eps_1 \log n.
\eeq
Then, for all $n$ sufficiently large depending on $\eps_1$, we have
\beq
\pp\left[ | \Psi_n (z, t, \eta_t ) - \Psi_n (z, 0, \eta_0) | > \eps_1^{1/3} \log n \right] \leq \e^{ - c_1 \eps_1^{-1/3} \log n } + n^{-100}.
\eeq
\eep
\proof We first consider the case of complex dynamics. Integrating \eqref{eqn:somestep-2} in time, and applying \eqref{eq:goodll} or Proposition \ref{prop:m-laws} (for $Y$ complex or real, respectively) we see that,
\beq
\Psi_n (z, t, \eta_t ) - \Psi_n (z, 0, \eta_0 ) = \Re\left[ \frac{1}{ \sqrt{2n}} \int_0^t \frac{ \d b_i^z (s)}{ \lambda_i^z (s) - \i \eta_s } \right] + \O \left( ( \log n )^{-9} \right)
\eeq
with overwhelming probability. The quadratic variation process of the martingale term is bounded by,
\beq
\sum_{i=-n}^n \frac{1}{ | \lambda_i^z (s) - \i \eta_s |^2} \d s \leq \frac{ \Im \langle G_s ( \i \eta_s) \rangle}{ \eta_s} \d s \leq \frac{C}{ \eta_s} \d s
\eeq
again by \eqref{eq:goodll} or Lemma \ref{prop:m-laws} (in the case of complex or real initial data respectively) 
with overwhelming probability. Therefore, by the martingale representation theorem (as in \eqref{eqn:mart-rep-1}), we have that
\beq
\pp\left[ \left|  \frac{1}{ \sqrt{2n}} \int_0^t \frac{ \d b_i^z (s)}{ \lambda_i^z (s) - w_s} \right| > ( \eps_1)^{1/3} \log n \right] \leq C \e^{ -c \eps_1^{-1/3} \log n} + n^{-1000}
\eeq
with overwhelming probability; this completes the proof in the case of complex dynamics. 

In the real case, we first note that the proof of Lemma \ref{lem:real-char} up to and including the estimate \eqref{eqn:real-upper-intermediate} holds even if we only assume that $\eta_t \geq ( \log n)^{10} /n$. Then, integrating \eqref{eqn:real-upper-intermediate}, and applying an estimate similar to \eqref{eqn:eta-dif-real-case}  for the deterministic contribution to $\Psi_n (z, s, \eta_s)$, we find that,
\beq \label{eqn:real-short-time}
\Psi_n (z, t, \eta_t) - \Psi_n (z, 0, \eta_0) = \Re\left[  \frac{1}{ \sqrt{2n}} \int_0^t \frac{ \d b_i^z (s)}{ \lambda_i^z (s) - w_s} + \tilsumij \int_0^t \langle G_s^z (w_s) E_i G_s^{\bar{z}} (w_s) E_j \rangle \d s\right] + \O ( \eps_1 \log n ) ,
\eeq
with overwhelming probability. 
By Cauchy-Schwarz and \eqref{eq:goodll} we have with overwhelming probability,
\beq
\left| \int_0^t \langle G_s^z (w_s) E_i G_s^{\bar{z}} (w_s) E_j \rangle \d s \right| \leq C\int_0^t \eta_s^{-1} \langle \Im[ G_s^{z} (w_s)] \rangle \d s \leq C \eps_1 \log n
\eeq
because $\log (\eta_0 / \eta_t) \leq \eps_1 \log n$ by assumption. Similarly, using \eqref{eq:goodll} and the computations around \eqref{eq:quadvarcomp} we see that the quadratic variation process of the Martingale term in \eqref{eqn:real-short-time} is bounded by $C/\eta_s$ with overwhelming probability. The proof is then completed in the same fashion as in the case of complex dynamics above. 

\qed

\

We now divide the remainder of this section into three parts. In Section~\ref{sec:largeup} we prove an upper bound for the (regularized) maximum of the log--characteristic polynomial in the regime $\Im z> n^{-\eps}$, for some small fixed $\eps>0$. Then, in Section~\ref{sec:smallup} we consider the regime $\Im z\le n^{-1/2-\eps}$ and, finally, in Section~\ref{sec:mesoreal} we conside the regime $\Im z\asymp n^{-\alpha}$ for some intermediate $\alpha\in (0,1/2)$.

\subsection{Real case: upper bound for $\Im[z] \geq n^{-\eps}$.}
\label{sec:largeup}

We first prove the upper bound for points $z$ that have relatively large imaginary part. This proof is almost the same as in the complex i.i.d. case. 

\bep \label{prop:upper-real-1} There is a $C_1 >0$ so that for all sufficiently small $\eps >0$ the following holds. Let $X$ be a real i.i.d. matrix with Gaussian component of size at least $n^{-\eps}$. Let $0 < r <1$. Then
\beq
\pp\left[ \max_{ |z| \leq r, \Im[z] \geq n^{-\eps}} \Psi_n (z, n^{-1} ) \geq \left( \sqrt{2} + C_1 \eps \right) \log n \right] \leq n^{-c \eps}
\eeq 
\eep
\proof The proof of this statement is similar to Proposition~\ref{prop:upper-gde-1}. First, by using Lemma \ref{lem:reglog} and Proposition \ref{prop:der-bd}, it suffices to bound the maximum over a set of $n^{1+\eps}$ well-spaced points $P_1$ of 
\beq
\max_{ z \in P_1} \Psi_n(z, T, (\log n)^{C_*}/n )
\eeq
for any sufficiently large $C_*>0$, with all $z \in P_1$ satisfying $\Im[z] \geq n^{-\eps}$ and $|z| \leq r$.  For each $z$ we let $\eta(s,z)$ be a characteristic ending at $( \log n)^{C_*} / n$ at time $s = T$. Next, letting $S_2 = T - n^{\eps^3-1}$ we have by Proposition \ref{prop:short-time-int} that with probability at least $1 - n^{-10}$ that,
\beq
\max_{z \in P_1} \Psi_n(z, T, (\log n)^{C_*}/n ) \leq \max_{z \in P_1} \Psi_n(z, S_2, \eta(S_2,z)) + C \eps \log n,
\eeq
for some constant $C>0$. At this point we can directly apply Lemma \ref{lem:real-char}, with $T= n^{-\eps}$, in an analogous fashion to the proof of Proposition \ref{prop:upper-gde-2} (which uses Lemma \ref{lem:charup}) using \eqref{eq:martrepxis1s2}  in place of \eqref{eq:varxinT}. Observe also that the term $E_n (S_1, S_2)$ in \eqref{eqn:real-char-main-est} equals the deterministic correction in \eqref{eqn:psi-def} up to $\O (1)$. We have by estimates \eqref{eqn:real-char-main-est} and \eqref{eq:martrepxis1s2} of Lemma \ref{lem:real-char} that, 
\begin{align}
     & \pp\left[ | \Psi_n (z, S_2, \eta(S_2,z) ) - \Psi_n (z, 0, \eta(0,z) ) | > ( \sqrt{2} + C_1 \eps )\log n \right] \notag\\
    \leq & \pp\left[ \left| \int_{0}^{S_2} (V_u)^{1/2} \d \tilde{b}_u \right| > ( \sqrt{2} + (C_1 - 1) \eps ) \log n \right] + n^{-10} \leq n^{-1-10 \eps}
\end{align}
if $C_1 >0$ is sufficiently large. In the last inequality we used the fact that 
\beq
\int_0^{S_2} V_u \d u \leq \log n + C \eps \log n
\eeq
by our assumption that $\Im[z] \geq n^{-\eps}$ from \eqref{eq:Vtexplicit}. Therefore, by a union bound over $P_1$,
\beq \label{eqn:upper-real-intermediate}
\max_{z \in P_1}  \Psi_n(z, S_2, \eta(S_2,z)) \leq \max_{z \in P_1} \Psi_n (z, 0, \eta(0,z) ) + ( \sqrt{2} + C \eps ) \log n
\eeq
with probability at least $1 - n^{-\eps/2}$, for some sufficiently large $C>0$. Now, the first quantity on the RHS of the above estimate is bounded by Proposition \ref{prop:global}. This finishes the proof. \qed

\subsection{Real case: upper bound for $\Im[z] \leq n^{-1/2+\eps}$.}
\label{sec:smallup}

\bep \label{prop:upper-real-2} Let $0 < r < 1$. 
There is a $C_1 >0$ and a $c>0$ so that for all sufficiently small $\eps >0$, the following holds. Let $X$ be a real i.i.d. matrix with Gaussian component of size at least $n^{-\eps}$. Then, for all $n$ sufficiently large depending on $r$ and $\eps$,
\beq
\pp\left[ \max_{ |z| \leq r, \Im[z] \leq n^{\eps-1/2}} \Psi_n (z, n^{-1} ) \geq \left( \sqrt{2} + C_1 \eps \right) \log n \right] \leq n^{-c \eps}  
\eeq
\eep
\proof This proof follows identically to Proposition \ref{prop:upper-real-1} except for the fact that one chooses a grid $P_1$ of $n^{1/2+2 \eps}$ well-spaced points. Furthermore, the intermediate inequality \eqref{eqn:upper-real-intermediate} still holds, but this is due to the fact that Gaussian random variable on the RHS of \eqref{eq:martrepxis1s2} has variance of no more than $(2+ C \eps) \log n$ which is compensated by the fact that we are taking a union bound over only $n^{1/2+2 \eps}$ points. \qed

\subsection{Real case: upper bound for $\Im[z] \approx n^{-\alpha}$.}
\label{sec:mesoreal}

The main result of this section is the following:

\bep
\label{prop:upper-real-3} Let $0 < r < 1$. There are $C_1, c_1, c_* >0$ so that the following holds. Fix $0 < \alpha < \frac{1}{2}$ and $\eps >0$ satisfying $\eps < c_* \min \{ \alpha, 1/2 - \alpha \}$. \nc Let $X$ be a real GDE with Gaussian component of size at least $n^{-\eps}$. Then, for all $n$ sufficiently large depending on $\eps, r$,
\beq
\pp\left[ \max_{ |z| < r, n^{-\alpha - \eps} < \Im[z] < n^{-\alpha + \eps}} \Psi_n (z, n^{-1} ) > ( \sqrt{2} + C_1 \eps ) \log n \right] \leq n^{- c_1\eps}.
\eeq
\eep

The proof of this proposition will require a few intermediate results. Let $T = n^{-\eps}$, and consider $X_s$ as above. We consider the characteristic $\eta (s, z)$ that end at $\eta (T, z) = (\log n)^{10} /n$. Fix
\beq
t_1 := T - n^{-2\alpha}, \qquad t_2 := T - n^{-\eps^3-1} 
\eeq
and decompose
\begin{align}
\label{eq:decomp}
\Psi_n (z, \eta (T, z), T) &= \big[ \Psi_n (z, \eta (T, z), T) - \Psi_n (z, \eta (t_2, z), t_2 ) \big] + \big[\Psi_n (z, \eta (t_2, z), t_2 ) - \Psi_n (z, \eta (t_1, z) , t_1 ) \big] \notag \\
&\quad+\big[\Psi_n (z, \eta (t_1, z), t_1) - \Psi_n (z, \eta (0, z), 0)\big]+ \Psi_n (z, \eta (0, z), 0) \notag \\
&\quad =: F(z) + X_2 (z)+ X_1 (z) + Y (z)
\end{align}
First, $Y(z)$ is an initial step that is very small and so can be neglected by Proposition \ref{prop:global}, using the fact that $\eta(0,z)\asymp t_f$. Similarly, $F(z)$ is a short time increment and can be neglected by Proposition \ref{prop:short-time-int}. The component $X_1 (z)$ is then the part of the random walk bringing us down to the intermediate scale $\eta \approx n^{-2 \alpha}$, and then $X_2 (z)$ goes down to the microscopic scales. As will be seen below, the equation \eqref{eq:Vtexplicit} implies that the quadratic variation process of the main Gaussian contributions to $X_1(z)$ and $X_2(z)$ behave differently.

We now control the maximum of $X_1(z)$ by a union bound, and then use this a priori information as an input to give an upper bound on $X_1(z)+X_2(z)$. This part of the argument is inspired by \cite{fang2012branching}. 

\bel
\label{lem:halmax}
There exists $C_1, c_1 >0$ so that for all $\eps>0$ sufficiently small, and all $n$ sufficiently large depending on $\eps$,
\beq
\label{eq:unionbalpha}
\pp\left[ \max_{ |z| \leq r , n^{-\eps - \alpha} \leq  \Im[z] \leq n^{\eps-\alpha} } |X_1 (z) | > (2 \sqrt{2} \alpha + C_1 \eps ) \log n  \right] \leq n^{ -c_1 \eps}
\eeq
\eel
\proof The proof will be by a union bound. First we have by definition, $\eta (t_1, z) \asymp n^{-2 \alpha}$ and that $\eta(t_1, z) \leq \eta(0, z) \leq n^{-\eps}$. Therefore by Lemma \ref{lem:eta-inc} and Proposition \ref{prop:der-bd}, the max of $X_1(z)$ over the set $\mathcal{E} := \{ z : |z| < r , n^{-\alpha - \eps } < \Im[z] < n^{-\alpha + \eps} \}$ is approximated by the max over a subset of $\mathcal{E}$  points  of cardinality at most $n^{\alpha+3\eps}$, up to an error of size $\O ( ( \log n)^{3/4} )$, with overwhelming probability. We denote this subset of points by $P_1$. 

For each $z \in P_1$ the estimates of Lemma \ref{lem:real-char} imply that
\beq
\pp\left[ |X_1 (z) | > (2 \sqrt{2} \alpha +C_1 \eps ) \log n  \right] \leq \pp\left[ |Z_1 (z) | > ( 2 \sqrt{2} \alpha + (C_1-1) \eps ) \log n \right] + n^{-10}
\eeq
where $Z_1(z)$ is a centered Gaussian random variable with variance bounded by $(4 \alpha + C \eps ) \log n$. Then for $C_1$ sufficiently large we have,
\beq
\pp\left[ |Z_1 (z) | > ( 2 \sqrt{2} \alpha + (C_1-1) \eps ) \log n \right]  \leq n^{-\alpha - 10 \eps}
\eeq
for all $\eps >0$ sufficiently small. This yields the claim. \qed 

\

Using Lemma~\ref{lem:halmax}, we now obtain the desired bound on $X_1(z)+X_2(z)$:

\bel \label{lem:inter-2} There exists $C_2, c_2, c_* >0$ so that the following holds. 
Let $P$ be a set of $n^{1-\alpha + 2\eps}$ well-spaced points of the strip $\{ z : n^{-\eps - \alpha } < \Im[z] < n^{\eps-\alpha}, |z| \leq  r\}$. For all $\eps >0$ satisfying $\eps < c_* \min \{ \alpha, (1- 2 \alpha ) \}$ we have that for $n$ sufficiently large depending on $\eps, r$,
\beq
\label{eq:unionbalpha2}
\pp\left[ \exists z \in P : X_1(z) + X_2 (z) > ( \sqrt{2} + C_2 \eps ) \log n\right] \leq n^{ - c_2 \eps}
\eeq
\eel
\proof We have by \eqref{eq:unionbalpha} and a union bound that,
\begin{align}
     & \pp\left[ \exists z \in P : X_1(z) + X_2 (z) > ( \sqrt{2} + C_2 \eps ) \log n\right] \notag\\
    \leq &  \pp\left[ \left\{ \exists z \in P : X_1(z) + X_2 (z) > ( \sqrt{2} + C_2 \eps ) \log n \right\} \cap \{ \max_{ z \in P } |X_1 (z) | \leq (2 \sqrt{2} \alpha + C_1 \eps ) \log n \} \right] + n^{-c_1 \eps} \notag \\
    \leq & \sum_{ z \in P } \pp\left[ \{ X_1 (z) + X_2 (z) > ( \sqrt{2} +C_2 \eps ) \log n \} \cap \{ |X_1 (z) | \leq ( 2 \sqrt{2} \alpha + C_1 \eps ) \log n \} \right] + n^{-c_1 \eps }. \label{eqn:upper-inter-1} 
\end{align}
Note that we can take $C_1 >0$ larger if necessary and the above estimate still holds. 
By applying Lemma \ref{lem:real-char} twice, once to $X_1(z)$ and once to $X_1(z) + X_2 (z)$ we see that with overwhelming probability,
\beq
X_1 (z) = Z_1 (z) + \O ( 1), \qquad X_2 (z) = Z_2 (z) + \O (1)
\eeq
where $Z_1(z)$ and $Z_2(z)$ are centered independent Gaussian random variables, by passing to possibly a larger probability space. As long as $\eps >0$ is sufficiently small depending on $\alpha >0$ we see that,
\beq
 4 \alpha \log n \asymp \Var (Z_1 (z) ) \leq (4 \alpha + C_* \eps ) \log n, \qquad (1- 2 \alpha ) \log n \asymp \Var (Z_2 (z) ) \leq (1-2 \alpha + C_* \eps ) \log n
\eeq
for some $C_* >0$ that will be fixed until the end of the proof. Assuming that $C_2 \geq C_1$, we have that the probability in the sum on the last line of \eqref{eqn:upper-inter-1} is bounded by (in the second line we write $Z_i = Z_i (z)$)
\begin{align}
 & \pp\left[ \{ Z_1 (z) + Z_2 (z) > ( \sqrt{2} +C_2 \eps ) \log n \} \cap \{ |Z_1 (z) | \leq ( 2 \sqrt{2} \alpha + C_1 \eps ) \log n \} \right] \notag \\
 = & \int_0^{\infty} \pp\left[ Z_2 > x + ( \sqrt{2} (1-2 \alpha) + (C_2 -C_1) \eps \log n ) \right] \pp\left[ Z_1 = (2\sqrt{2}  \alpha +C_1 \eps) \log n - x \right] \d x \notag \\
 \leq & n^{\eps}  \int_0^{\infty} \exp \left( - \frac{ \big(x + ( \sqrt{2} (1- 2\alpha ) + (C_2 -C_1) \eps )\log n \big)^2}{( 2(1-2 \alpha ) +C_* \eps ) \log n} - \frac{ \big( (2\sqrt{2}  \alpha +C_1 \eps) \log n - x \big)^2}{( 8 \alpha + C_*\eps ) \log n }  \right) \d x \notag \\
 \leq & n^{\eps} \exp \left( - \log n \left( \frac{2(1- 2 \alpha)^2 + (1- 2 \alpha)(C_2 -C_1) \eps }{2(1-2 \alpha) +C_* \eps}  + \frac{ 8 \alpha^2 + C_1 \alpha \eps}{8 \alpha + C_* \eps} \right)   \right) \notag \\
 &\times \int_0^\infty \exp \left( -x \left( \frac{ 2^{3/2} ( 1 - 2 \alpha)}{2(1-2\alpha) + C_* \eps} - \frac{ 4 \sqrt{2} \alpha }{8 \alpha + C_* \eps } - \frac{2 C_1 \eps}{8 \alpha +C_* \eps}  \right) \right) \d x \label{eqn:upper-inter-2}
\end{align} 
As long as $\eps >0$ satisfies $\eps < (C_*)^{-1} \min\{ 1- 2 \alpha, \alpha \}$ we see that
\beq
 \frac{2(1- 2 \alpha)^2 + (1- 2 \alpha)(C_2 -C_1) \eps }{2(1-2 \alpha) +C_* \eps}  + \frac{ 8 \alpha^2 + C_1 \alpha \eps}{8 \alpha + C_* \eps}  \geq 1 - \alpha + \frac{ C_2 -C_1}{10} \eps  \geq 1 - \alpha + 100 \eps 
\eeq
as long as $C_1 \geq C_*$ and $C_2 \geq C_1 + C_* + 10^3$.  Fix $C_1, C_2 >0$ until the end of the proof. 

Assuming further that $\eps < 10^{-6} (C_*)^{-1} \min \{ \alpha, 1 - 2 \alpha \}$ and also $\eps < C_1 \alpha 10^{-6}$ we see that,
\beq
\frac{ 2^{3/2} ( 1 - 2 \alpha)}{2(1-2\alpha) + C_* \eps} - \frac{ 4 \sqrt{2} \alpha }{8 \alpha + C_* \eps } - \frac{2 C_1 \eps}{8 \alpha +C_* \eps} \geq \sqrt{2} - 2^{1/2} - \frac{1}{100} \geq \frac{1}{100}. 
\eeq
Therefore the integral on the last line of \eqref{eqn:upper-inter-2} converges, and is bounded by a constant. The claim now follows. \qed

\vspace{3 pt} 

\

We are now ready to conclude the estimate of the maximum of $\Psi_n(z,n^{-1})$ for $\Im z\asymp n^{-\alpha}$.

\noindent{\bf Proof of Proposition \ref{prop:upper-real-3}}. By Lemma \ref{lem:reglog} and Proposition \ref{prop:der-bd} it suffices to bound
\beq
\max_{z \in P_1} \Psi_n (z, (\log n )^{10}/n )
\eeq
for $P_1$ being a set of $n^{1-\alpha+2\eps}$ points in the set $\{ z : |z| < r , n^{-\eps - \alpha } \leq \Im [z ] \leq n^{\eps - \alpha} \}$. For each such $z$ we use the decomposition in \eqref{eq:decomp}. By Proposition \ref{prop:global} and Proposition \ref{prop:short-time-int} we have that
\beq
\pp\left[ \max_{ z \in P_1}  |Y(z)| > C \eps \log n \right] \leq n^{-\eps} , \qquad \pp\left[ \max_{ z \in P_1}  |F(z)| > C \eps \log n \right] \leq n^{-10}
\eeq
for some sufficiently large $C>0$ and all $\eps >0$ sufficiently small. The estimate for $X_1(z) + X_2 (z)$ follows from Lemma \ref{lem:inter-2}. \qed

\subsection{Proof of Proposition \ref{prop:upper-real-gde}}

The upper bound now follows in a straightforward manner from Propositions \ref{prop:upper-real-1}, \ref{prop:upper-real-2}, and \ref{prop:upper-real-3}. Fixing an $\eps_1 >0$, we apply Proposition \ref{prop:upper-real-1} and \ref{prop:upper-real-2} to control the maximum for $z$ s.t. $\Im[z] \leq n^{-1/2+\eps_1}$ or $\Im[z] \geq n^{-\eps_1}$. Proposition \ref{prop:upper-real-3} then applies for all $ \alpha \in (\eps_1/2, 1/2 - \eps_1/2)$, with the $\eps $ in the statement of Proposition \ref{prop:upper-real-3} equalling $\eps_2 := c_* \eps_1$ for the $c_* >0$ coming from that Proposition.  We then apply Proposition \ref{prop:upper-real-3} finitely many times (where finitely many depends on $\eps_1$ and $c_*$) to conclude that,
\beq
\pp\left[ \max_{ z : |z| \leq r } \Psi (z, n^{-1} ) \geq (\sqrt{2} + C \eps_1 ) \log n \right] \leq n^{ -c \eps_1}
\eeq
for all matrices with Gaussian component of size at least $n^{-c \eps_1}$, for some small $c, C>0$. This yields the claim, after defining $\eps >0$ in terms of $\eps_1 >0$ appropriately. \qed

\section{Upper bound for general ensembles; comparison}

\subsection{Comparison}

In this section we will state a general comparison result for a certain regularization of the maximum of the characteristic polynomial. We fix $0 < r < 1$ and a parameter $\delta >0$. Let $P$ be a set of points of the disc $\{ z : |z| \leq r \}$ such that $| P | \leq n^2$. With this data, consider the function on the space of iid matrices given by,
\beq
Z_\delta (X) := \frac{1}{ n^{\delta}} \log \left( \sum_{ z \in P_\delta } \e^{ n^{\delta} \frac{1}{2} \Psi_n (z, n^{-1} ) } \right)
\eeq
Let $X$ and $Y$ be two i.i.d. matrices that match moments to order $3$ and their fourth moments differ by $\O ( T n^{-2} )$. The proof of the following lemma is deferred to Appendix \ref{app:GFTmax}. 
\bel \label{lem:upper-compare-1}
In the above set up, we have
\beq
\left| \ee[ F ( Z_\delta (X) ) ] - \ee[ F (Z_\delta (Y ) ] \right| \leq \| F \|_{C^5}(  T n^{10\delta} + n^{-1/4} )
\eeq
\eel

\subsection{Proof of upper bounds of Theorems \ref{theo:main} and \ref{theo:realmaxlog}}

It suffices to prove the upper bounds of \eqref{eq:desb} and \eqref{eqn:realmaxlog-1}, as the upper bound of \eqref{eqn:realmaxlog-2} is a consequence of \eqref{eqn:realmaxlog-1}. 

In the set-up of the previous section, we choose $P = P_\delta$, a set of $n^{1+\delta}$ well-spaced points of the disc of radius $r$. 
It is easy to see that, almost surely,
\beq \label{eqn:upper-reg-1}
\left| \max_{z \in P_\delta} \frac{1}{2} \Psi_n (z, n^{-1} ) - Z_\delta (X) \right| \leq \frac{ 2 \log n}{n^{\delta}}.
\eeq
Moreover, from Proposition \ref{prop:der-bd} we have that
\beq \label{eqn:upper-reg-2}
\left| \max_{ |z| < r } \Psi_n (z, n^{-1} ) - \max_{z \in P_\delta} \Psi_n (z, n^{-1} )  \right| \leq n^{-\delta/2}
\eeq
with overwhelming probability.

Note that,
\beq
\sum_i \log ( ( \lambda_i^z)^2 ) \leq \sum_i \log ( ( \lambda_i^z)^2 + \eta^2)
\eeq
and that 
\beq
\left| \del_\eta \int \log ( x^2 + \eta^2) \rho_z (x) \d x \right|  = 2 \Im [ m^z ( \i \eta ) ] \asymp 1
\eeq
by the last point in Lemma~\ref{lem:proprho}, and so almost surely we have
\beq
\Psi_n (z) \leq \Psi_n (z, n^{-1} ) + C .
\eeq
It therefore suffices to prove the estimates for the regularized $\Psi_n (z, n^{-1})$. 
Let $X$ be an i.i.d. matrix and $\eps >0$. Let $Y$ match $X$ to three moments and the fourth to $\O ( T n^{-2} )$ with $T = n^{-\eps}$, with $Y$ a GDE with component of size $T$. We let $\delta = \eps / 20$. From the discussion above it suffices to bound $Z_\delta (X)$. We see by Lemma \ref{lem:upper-compare-1} that
\beq
\pp\left[ Z_{\delta} (X) \geq \left( \frac{1}{\sqrt{2} } +C_1 \eps \right) \log (n) \right] \leq \pp\left[ Z_{\delta} (Y) \geq \left( \frac{1}{\sqrt{2} } +C_1 \eps \right) \log (n) + 1 \right] + n^{- c\eps}.
\eeq
On the other hand, the probability on the RHS is easily bounded by relating $Z_\delta (Y)$ back to the max of $\Psi_n (z, n^{-1}, Y)$ using \eqref{eqn:upper-reg-1}  and then applying Proposition \ref{prop:upper-gde-1} and Proposition \ref{prop:upper-real-gde} in the complex and real cases, respectively. \qed

\section{Second moment method; lower bound on mesoscopic scales via DBM in the complex case} \label{sec:smm-complex}

In this section we will find a lower bound for the log-characteristic polynomial on mesoscopic scales. We will apply a dynamical version of the second moment method; our treatment of the second moment method follows roughly that outlined in the expository notes \cite{arguin2016extrema}. Our dynamical set-up will be similar to Section \ref{sec:upper-1}.  Fix two exponents $\mfa, \mfb >0$, and define
\beq
\mfc := \min \{ \mfa, \mfb \}.
\eeq
We assume $\mfc < 10^{-3}$. 
Let $t_\mfb = n^{-\mfb}$ and set $\d X_t = \d B_t/\sqrt{n}$ where $B_t$ is a matrix whose entries are i.i.d. standard real or complex Brownian motions. Here, $X_0 $ is a matrix of the form $X_0 = (1- t_\mfb)^{1/2} Y$ where $Y$ is an real or complex i.i.d. matrix as in Definition \ref{def:model}. The limiting Stieltjes transform of the Hermitization of $X_t -z$ is given by $m_t^z(w)$ as in \eqref{eqn:mt-def}, with $c_*(t) = \sqrt{1+(t-t_\mfb)}$, and it satisfies \eqref{eq:evscfl} with characteristics given by \eqref{eq:char}. 

In this section we will only consider the $\beta=1$ case in Proposition \ref{prop:covariation-estimate} below (for later use in Section \ref{sec:smm-real}); in all other statements in this section we will assume $\beta=2$. 

With this definition of $X_t$ we introduce the field,
\beq
\Phi (z) := \Psi_n (z, t_\mfb , \eta_\mfa (z, t_\mfb )) - \Psi_n ( z, 0, \eta_\mfa (z, 0) )
\eeq
where $\eta_\mfa (z, s)$ is a characteristic such that $\eta_\mfa (z, t_\mfb) = n^{\mfa-1}$.  The main result of this section, proven by the second moment method, is the following:
\bet \label{thm:meso-smm}
Let $P$ be a grid of $n^{1-\mfa}$ well-spaced points of the disc $\{ |z-\frac{1}{2} \i | < \frac{1}{4}\}$. There is a $c>0$ so that the following holds.  For all sufficiently small $\eps_1 >0$ we have that
\beq
\max_{ z \in P } \Phi (z) \geq \sqrt{2} ( \sqrt{1 - \mfa - \mfb } \sqrt{ 1 - \mfb}  - \eps_1 ) \log n
\eeq
with probability at least $ 1 - n^{ - c \eps_1}$. 
\eet

Note that due to the proof of Lemma \ref{lem:charup} we have the decomposition,
\beq
\Phi (z) = \Re\left[ \frac{1}{\sqrt{2n}} \int_0^{t_\mfb} \sum_i \frac{ \d b_i^z (s) }{ \lambda_i^z (s) - \i \eta_\mfa (z, s) } \right] + \O ( n^{-\mfa/2} )
\eeq
with overwhelming probability. Specifically, this follows from \eqref{eqn:somestep-2} and \eqref{eqn:char-inter-1}.  Fix now an integer $K \geq 1$ and define,
\beq
\delta_K := \frac{ 1 - \mfa - \mfb}{K}.
\eeq
Define now $t_K = t_{\mfb} = n^{-\mfb}$ and $t_0 = 0$. For $1 \leq i \leq K-1$ we define,
\beq
t_i := t_{\mfb} - \frac{ n^{\mfa + (K-i) \delta_K}}{n}
\eeq
Then $t_0 < t_1 < \dots < t_K$ and 
\beq
\log ( \eta_\mfa (z, t_i ) / \eta_\mfa (z, t_{i+1} ) ) = \delta_K \log n + \O (1)
\eeq
for $0 \leq i < K$, and
\beq
\eta_\mfa (z, t_K) = \frac{ n^{\mfa}}{n}, \qquad \eta_\mfa (z, t_i ) \asymp \frac{ n^{\mfa + (K-i) \delta_K }}{n} , \quad 0 \leq i < K.
\eeq
Define,
\beq
Y_i (z) := \Re \left[ \frac{1}{ \sqrt{2n}} \int_{t_{i-1}}^{t_i} \sum_i \frac{ \d b_i^z (s) }{ \lambda_i^z (s) - \i \eta_\mfa (z, s) } \right]
\eeq
We now compute the covariation process of the $Y_i (z)$. Fix $z_1,z_2\in\C$, for any $A\in \C^{2n \times 2n}$, we recall that the deterministic approximation of $G_t^{z_1}(\ii\eta)AG_t^{z_2}(\ii\eta_2)$ is given by (see \eqref{eq:realdefm12})
\begin{equation}
\label{eq:defm12a}
M_t^{z_1,z_2}(\ii\eta_1,A,\ii\eta_2)=\big(1-c_*(t)^2M_t^{z_1}(\ii\eta_1)\mathcal{S}[\cdot]M_t^{z_2}(\ii\eta_2)\big)^{-1}\big[M_t^{z_1}(\ii\eta_1)A_1M_t^{z_2}(\ii\eta_2)\big].
\end{equation}
Then, for the covariation process we have the following:
\bep \label{prop:covariation-estimate}
Denoting $\eta_{i, t} = \eta_\mfa (z_i, t)$, with overwhelming probability, for any small $\xi>0$ we have
\begin{align}
    & \left[ \Re \frac{1}{ \sqrt{2n}}  \sum_i \frac{ \d b_i^{z_1} (t) }{ \lambda_i^{z_1} (t) - \i \eta_\mfa (z_1, t) } , \Re \frac{1}{ \sqrt{2n}} \sum_i \frac{ \d b_i^{z_2} (t) }{ \lambda_i^{z_2} (t) - \i \eta_\mfa (z_2, t) } \right] \\
    = & \left[2\tilsumij \langle M_t^{z_1,z_2}(\ii\eta_{1,t},E_i,\ii\eta_{2,t})E_j\rangle+\O\left(\frac{n^\xi}{n\eta_*(t)^2}\right)\right] \d t \\
    + & \1_{ \{ \beta =1 \}} \left[ 2\tilsumij \langle M_t^{z_1,\bar{z}_2}(\ii\eta_{1,t},E_i,\ii\eta_{2,t})E_j\rangle \right] \d t
\end{align}
where $\eta_{*,t} = \min \{ \eta_{1,t} , \eta_{2,t} \}$ and $M_t^{z_1,z_2}$ is defined as in \eqref{eq:defm12a}.
\eep
\proof \, We first consider the complex case. Then the covariation process is, by direct calculation, 
\begin{equation} \label{eqn:prop-covar-1}
\begin{split}
&\frac{1}{2n} \sum_{i,j} \frac{\lambda_i^{z_1}\lambda_j^{z_2}\dif [b_i^{z_1},b_j^{z_2}]}{|\lambda_i^{z_1}-\ii\eta_{1,t}|^2|\lambda_j^{z_2}-\ii\eta_{2,t}|^2} \d t\\
&=\frac{1}{2n}\sum_{i,j} \frac{4\lambda_i^{z_1}\lambda_j^{z_2}\Re[\langle \bm{w}_i^{z_1},E_1\bm{w}_j^{z_2}\rangle\langle \bm{w}_j^{z_2},E_2\bm{w}_i^{z_1}\rangle]}{|\lambda_i^{z_1}-\ii\eta_{1,t}|^2|\lambda_j^{z_2}-\ii\eta_{2,t}|^2}\,\dif t \\
&=2 \tilsumij \langle \Re G_t^{z_1}(\ii\eta_{1,t})E_i\Re G_t^{z_2}(\ii\eta_{2,t})E_j\rangle \, \d t \\
&=2 \tilsumij \langle G_t^{z_1}(\ii\eta_{1,t})E_i G_t^{z_2}(\ii\eta_{1,t})E_j \rangle \,  \dif t \\
&=\left[2\tilsumij \langle M_t^{z_1,z_2}(\ii\eta_{1,t},E_i,\ii\eta_{2,t})E_j\rangle+\O\left(\frac{n^\xi}{n\eta_*(t)^2}\right)\right]\,\d t.
\end{split}
\end{equation}
We point out that to go from the third to the fourth line we used that $\Im G_t^{z_i} ( \i \eta) $ is diagonal and that $\Re G_t^{z_i} ( \i \eta)$ is off--diagonal as a consequence of the symmetry of the spectrum of $H_t^z$; in particular this implies that $E_1\Im G_t^{z_i} (\i \eta) E_2=E_2\Im G_t^{z_i} ( \i \eta) E_1=0$. Additionally, in the last line we used the following estimate, which is a consequence of rescaling the entries of the matrix considered in \cite[Theorem 3.3]{cipolloni2023mesoscopic} (i.e. we now have entries with variance $c_*(t)^2/n$ instead of $1/n$ as in \cite{cipolloni2023mesoscopic}),
\begin{equation} \label{eqn:multi-resolvent}
\big|\langle (G_t^{z_1}(\ii\eta_1)A_1G_t^{z_2}(\ii\eta_2)-M_t^{z_1, z_2} (\i \eta_1, A_1, \i \eta_2 ))A_2\rangle\big|\lesssim \| A_1 \| \|A_2 \|\frac{n^\xi}{n\eta_*^2},
\end{equation}
with $\eta_*:=|\eta_1|\wedge |\eta_2|$, with overwhelming probability, for any $\xi>0$ and for any deterministic $A_1, A_2 \in\C^{2n\times 2n}$.  In the real i.i.d. case, the third line of \eqref{eqn:prop-covar-1} has an additional term
\beq
2\tilsumij \langle \Re G_t^{z_1}(\ii\eta_{1,t})E_i\Re G_t^{\bar{z}_2}(\ii\eta_{2,t})E_j\rangle \, \d t,
\eeq
whose deterministic approximation can be computed as in \eqref{eqn:prop-covar-1}, again using \eqref{eqn:multi-resolvent} but with $z_2$ replaced with $\overline{z_2}$. This concludes the proof.

\qed

\bel \label{lem:cov-estimates}
For the deterministic process
\beq \label{eqn:M-def}
M (z_1, z_2, t):=2\tilsumij \langle M_t^{z_1,z_2}(\ii\eta_{1,t},E_i,\ii\eta_{2,t})E_j\rangle,
\eeq
denoting $\eta_{i, t} = \eta_\mfa (z_i, t)$, we have the following estimates.
First, 
\beq
\label{eq:firstbM}
M (z, z, t) = \frac{ \Im[ m_t^z ( \eta_\mfa (z, t))]}{\eta_\mfa (z, t)} +\O(1).
\eeq
Additionally, if there is a $\sigma >0$ so that $|z_1-z_2|^2 \geq n^{\sigma} \eta_\mfa (z_1, t)$ then,
\beq
\label{eq:secondbM}
|M (z_1, z_2, t)| \lesssim \frac{1}{n^{\sigma} \eta_\mfa (z_1, t)},
\eeq
where we note that $\eta_\mfa (z_1, t) \asymp \eta_\mfa (z_2, t)$. Finally, if $\eta_\mfa (z, t)$ is such that $\eta_\mfa (z, t) \geq n^{\sigma} \Im[z]^2$ then,
\beq
\label{eq:thirdbM}
M (z, \bar{z}, t) = \frac{ \Im[ m_t^z ( \eta_\mfa (z, t))]}{\eta_\mfa (z, t)} \big(1 + \mathcal{O}(n^{- \sigma/10})\big) + \O (1)
\eeq
\eel
\proof \,  By \cite[Lemma 6.1]{cipolloni2023central}, we have
\begin{equation}
\label{eq:deb12mhop}
\lVert M^{z_1,z_2}(\ii\eta_1,A,\ii\eta_2)\rVert\lesssim \frac{1}{|z_1-z_2|^2+|\eta_1|+|\eta_2|},
\end{equation}
for any $\eta_i\ne 0$ and $\lVert A\rVert\le 1$, which readily implies \eqref{eq:secondbM}.

We now prove \eqref{eq:firstbM} and \eqref{eq:thirdbM}. We recall that by an explicit computation (see \eqref{eqn:real-det-1} and the shorthand notation defined directly before it) we have
\begin{equation}
M (z, \bar{z}, t)=2\frac{c_tu^2\Re[z^2]-c_t^2|z|^4u^4+c_t^2m^4}{1+c_t^2|z|^4u^4-c_t^2m^4-2c_tu^2\Re[z^2]}.
\end{equation}
Then, using \eqref{eqn:real-det-2}, and that
\[
c_tu^2\Re[z^2]-c_t^2|z|^4u^4+c_t^2m^4=1-|z|^2+\mathcal{O}(|z-\overline{z}|^2+\eta_\mathfrak{a}(z,t))
\]
by \eqref{eq:exapnsmu}, we conclude \eqref{eq:thirdbM}. We point out that here we also used that
\[
c_t=1+(t-t_\mathfrak{b})=1-\frac{\eta_\mathfrak{a}(z,t)-\eta_\mathfrak{a}(z,t_\mathfrak{b})}{\Im m}=1+\mathcal{O}(\eta_\mathfrak{a}(z,t)).
\]
The proof of \eqref{eq:firstbM} is completely analogous (in fact simpler) and so omitted.

\qed

We now define
\beq
V_j (z_1, z_2 ) = \int_{t_{j-1}}^{t_j} M (z_1, z_2, s) \d s
\eeq
to be the leading order deterministic approximation to the covariance of the $Y_k$'s. We note that by \eqref{eq:firstbM} we have that
\beq \label{eqn:diag-var}
V_j (z, z) = \delta_K \log n + \O (1)
\eeq

\bep \label{prop:coupling-better}
Let $\sigma >0$ be sufficiently small. Fix $0 \leq k_0 \leq K$ an integer. Let $z_1, z_2$ be two points such that (recall that $\eta_{\mfa} (z_1, t_k )\asymp \eta_{\mfa} (z_2, t_k )$)
\beq
|z_1-z_2|^2 \geq n^{\sigma} \eta_{\mfa} (z_1, t_{k_0} ).
\eeq
Then there is a coupling between the random variables $\{ Y_j(z_1)\}_{j=1}^K, \{ Y_j (z_2) \}_{j=k_0+1}^K$ and a vector of independent Gaussian random variables $\{ Z_j (z_1) \}_{j=1}^K, \{ Z_j(z_2) \}_{j=k_0+1}^K$ such that with overwhelming probability we have that,
\beq \label{eqn:coupling-est-1}
Y_j (z_1) = Z_j (z_1) + \O ( n^{-\sigma/10} + n^{-\mfa/10} ) , \quad Y_l (z_2) = Z_l (z_2) + \O ( n^{-\sigma/10} + n^{-\mfa/10} )
\eeq
for $1 \leq j \leq K$ and $k_0+1 \leq l \leq K$. The variance of the $Z_j(u)$ are given by
\beq
\Var( Z_j (u) ) = V_j (u, u) = \int_{t_{j-1}}^{t_j} M (u, u, s) \d s = \delta_K \log n + \O (1) ,
\eeq
for $(u, j) \in \{ (z_1, i ) : 1 \leq i \leq K \} \cup \{ (z_2, i) : k_0+1 \leq i \leq K \}$.
Note that if $k_0 = K$ then we are only producing a coupling for the process $\{ Y_j (z_1) \}_{j=1}^K$, and the estimates \eqref{eqn:coupling-est-1} holds without the $n^{-\sigma/10}$ error.
\eep
\proof Let $\d[Y_j (z ) ,Y_k (w  )]  = \CC_{jk} (z, w, t) \d t$ be the covariation process of the $Y_k$'s. Note that this is non-zero for $j \neq k$ and for $j \leq k_0$ we are only considering the single process $\d Y_j (z_1)$. For each time $t$, the covariation process of $\{ Y_j(z_1)\}_{j=1}^K, \{ Y_j (z_2) \}_{j=k_0+1}^K$ is a $(K + (K-k_0) ) \times (K + (K-k_0) )$ dimensional matrix which we will denote by $\CC(t)$. Construct now a deterministic diagonal matrix $\MM(t)$ of the same dimension as follows. For every possible choice of $z_i$ and $j$ such that $Y_j (z_i)$ is one of the elements of our process, if the $(k_1, k_1)$--th entry of $\CC(t)$ corresponds to the variation process of this $Y_j (z_i)$ then set
\beq
\MM_{k_1, k_1} (t) := M (z_i, z_i , t) \1_{ \{ t \in (t_{j-1}, t_j ) \}} .
\eeq
Set all other entries of $\MM$ to be $0$. Then, with overwhelming probability by Proposition \ref{prop:covariation-estimate} and \eqref{eq:secondbM} we have that
\beq \label{eqn:eee-1}
\max_{i,j} | \CC_{ij} (t)  - \MM_{ij} (t) | \leq C \left( \frac{1}{ n^{\sigma} \eta_{\mfa} (z_1, t)} + \frac{n^\xi}{n \eta_\mfa (z_1, t)} \right)
\eeq
for any $\xi>0$. By the martingale representation theorem, there is a probability space and a sequence of $K + (K - k_0)$ standard Brownian motions $\tilde{b}_t$ such that
\beq
\d Y_t = \sqrt{ \CC(t) } \d \tilde{b}_t
\eeq
where by an abuse of notation we let $Y_t$ to denote the $K + (K - k_0)$-dimensional vector of the processes $\{ Y_j(z_1)\}_{j=1}^K, \{ Y_j (z_2) \}_{j=k_0+1}^K$. We can define now $Z$ by $\d Z_t = \sqrt{ \MM(t) } \d \tilde{b}_t$.  Clearly $Z$ has the desired distribution (recall that $\MM(t)$ is diagonal and deterministic). To estimate the difference we compute the quadratic variation of $Y_j (u)- Z_j(u)$, for $u$ and $j$ as in the proposition statement. This is bounded by,
\begin{align}
[ \d (Y_j - Z_j ) , \d (Y_j - Z_j )] (t) &= 1_{ \{ t \in (t_{j-1}, t_j ) \} } \sum_{ \alpha} ( \sqrt{ \CC} - \sqrt{ \MM} )^2_{j, \alpha } \d t \notag \\
&= 1_{ \{ t \in (t_{j-1}, t_j ) \} } ( (\sqrt{ \CC} -  \sqrt{ \MM} )^2 )_{jj} \notag \\
&\leq 1_{ \{ t \in (t_{j-1}, t_j ) \} } \mathrm{Tr} \left[ ( (\sqrt{ \CC} -  \sqrt{ \MM} )^2 \right] \notag \\
&\leq 1_{ \{ t \in (t_{j-1}, t_j ) \} }  \mathrm{Tr} | \CC - \MM |  \notag \\
&\lesssim  1_{ \{ t \in (t_{j-1}, t_j ) \} }   \left( \frac{1}{ n^{\sigma} \eta_{\mfa} (z_1, t)} + \frac{n^\xi}{n \eta_\mfa (z_1, t)} \right)
\end{align}
The last inequality uses \eqref{eqn:eee-1} and the fact that $\mathrm{Tr} |T| \leq \sum_{ij} |T_{ij} |$, and the second last inequality uses the Powers-St{\o}rmer inequality (see \cite[Section 3]{powers1970free}). By integrating the final inequality we then obtain
\beq
\int_{t_{j-1}}^{t_j} [ \d (Y_j - Z_j ), \d (Y_j - Z_j ) ] \leq C (n^{-\sigma/2} + n^{-\mfa/2} ) ,
\eeq
with overwhelming probability. 
The claim now follows from the BDG inequality.

\qed

\

We now state the following asymptotic for Gaussian random variables (see e.g. \cite[Theorem 1.2.3]{durrett2019probability} \nc), which will be useful in the following. If $Z$ is a centered Gaussian with variance $\sigma^2$ we have,
\beq \label{eqn:basic-gaussian}
\frac{1}{ \sqrt{2 \pi}} \left( \frac{ \sigma}{x} - \frac{ \sigma^3}{x^3} \right) \e^{ - \frac{x^2}{2 \sigma^2}} \leq \pp\left[ Z > x \right] \leq \frac{1}{ \sqrt{2  \pi}} \frac{ \sigma}{x} \e^{ - \frac{x^2}{2 \sigma^2}}.
\eeq
Let us now define,
\beq
\EE_m (z) := \left\{Y_m (z) >  \frac{ \sqrt{2}}{K} (\sqrt{ 1  - \mfa - \mfb}\sqrt{1-\mfa} - \eps_1 ) \log n \right\}
\eeq
as well as,
\beq
\FF_m (z) := \left\{  \ZZ_m (z) >  \frac{ \sqrt{2}}{K} ( \sqrt{ 1  - \mfa - \mfb } \sqrt{1- \mfa} - \eps_1 ) \log n \right\}
\eeq
where $\{ \ZZ_m (w) \}_{m, w}$ is a family of independent Gaussian random variables with variance $V_m (w, w)$. 
Define the functions
\beq
p_m (z, x) := \pp\left[ \ZZ_m (z) > x \right], \qquad \hat{p}_m (z) := \pp\left[ \ZZ_m (z) >  \frac{ \sqrt{2}}{K} ( \sqrt{ 1  - \mfa - \mfb }\sqrt{1-\mfa}  - \eps_1 ) \log n \right],
\eeq
and the point
\beq \label{eqn:xhat-def}
\hat{x} :=  \frac{ \sqrt{2}}{K} ( \sqrt{ 1  - \mfa - \mfb }\sqrt{1-\mfa} - \eps_1 ) \log n.
\eeq
\bel \label{lem:first-moment}
We have that, 
\beq \label{eqn:first-moment-1}
\pp\left[ \bigcap_{m=1}^K \EE_m (z) \right] = \left( \prod_{m=1}^K \hat{p}_m (z) \right) (1 + \O ( n^{-\mfa/100} )) \asymp (\log n)^{-K/2} n^{- (1-\mfa) + \eps_1 [ 2(1-\mfa)^{1/2}(1-\mfa - \mfb )^{-1/2} - \eps_1 / (1- \mfa - \mfb ) ]}
\eeq
and so,
\beq \label{eqn:first-moment-2}
\ee\left[ \sum_{z \in P}  \prod_{m=1}^K 1_{ \{ \EE_m (z) \} } \right] \asymp (\log n)^{-K/2} n^{  \eps_1 [ 2(1-\mfa)^{1/2}(1-\mfa - \mfb )^{-1/2} - \eps_1 / (1- \mfa - \mfb ) ]}
\eeq
\eel
\proof By \eqref{eqn:basic-gaussian} and \eqref{eqn:diag-var} we have that,
\begin{align}
\hat{p}_m (z) &= \frac{1}{ \sqrt{2 \pi} } \frac{ K \sqrt{ V_m (z, z)}}{ \sqrt{2} ( \sqrt{ 1 - \mfa - \mfb} - \eps_1 ) \log n} \exp\left( - \frac{ (\sqrt{1- \mfa - \mfb}\sqrt{1-\mfa} - \eps_1 )^2 (\log n)^2}{K^2 V_m (z, z) } \right)\left( 1 + \O ( (\log n)^{-1} ) \right) \notag \\
&\asymp \frac{ K^{1/2}}{ (\log n)^{1/2}} n^{-(1-\mfa)/K} n^{ \frac{\eps_1}{K} [2 (1-\mfa)^{1/2} (1 - \mfa - \mfb)^{-1/2} - \eps_1 / (1- \mfa -\mfb ) ] } \label{eqn:ee-2}
\end{align}

By Proposition \ref{prop:coupling-better} (applied to the case of just a single $z_1$ or $k_0=K$), with $\hat{x}$ as in \eqref{eqn:xhat-def}, we have that
\beq
\prod_{m=1}^K p_m (z, \hat{x} + n^{-\mfa/10}  ) - n^{-1000} \leq \pp\left[ \bigcap_{m=1}^K \EE_m (z) \right] \leq \prod_{m=1}^K p_m (z, \hat{x} - n^{-\mfa/10} ) + n^{-1000}.
\eeq
It is straightforward to check, using \eqref{eqn:basic-gaussian} and the explicit form of the Gaussian density that for $|s| \leq 1$ we have
\beq \label{eqn:basic-gaussian-2}
p_m (z, \hat{x} + s ) = p_m (z, \hat{x})(1 + \O (|s|  ) )
\eeq
This now completes the proof.  \qed

\bep \label{prop:second-mom-1}
Suppose that $|z-w|^2 > n^{-\mfb/10}$. Then,
\beq
\pp\left[ \bigcap_{m=1}^K \EE_m (z) \cap \EE_m (w) \right] = \left( \prod_{m=1}^K \hat{p}_m (z) \hat{p}_m (w) \right) (1 + \O (n^{-\mfc / 200} ) ).
\eeq
\eep
\proof We just prove the upper bound, the lower bound being very similar. We first have by Proposition \ref{prop:coupling-better} with $\sigma = \mfb /2$ (recall that $\eta_\mfa (z, 0) \asymp n^{-\mfb}$ by definition)  that 
\beq
\pp\left[ \bigcap_{m=1}^K \EE_m (z) \cap \EE_m (w) \right] \leq \pp\left[ \bigcap_{m=1}^K \GG_m (z, \hat{x} - n^{-\mfc /100} ) \cap \GG_m (w, \hat{x} - n^{-\mfc /100} ) \right] + n^{-100},
\eeq
where 
\beq
\GG_m (z, x) := \{  \ZZ_m (z) > x \}.
\eeq
Now, by the independence of the $\ZZ_m$'s we have,
\beq
\pp\left[ \bigcap_{m=1}^K \GG_m (z, \hat{x} - n^{-\mfc /100} ) \cap \GG_m (w, \hat{x} - n^{-\mfc /100} ) \right] = \prod_{m=1}^K \pp\left[ \GG_m (w, \hat{x} - n^{-\mfc / 100} ) \right] \pp \left[ \GG_m (z, \hat{x} - n^{-\mfc/100} ) \right]
\eeq
The claim now follows from \eqref{eqn:basic-gaussian-2}. \qed

\bep \label{prop:second-mom-2}
Suppose that for some $\sigma >0$ and $0 \leq k \leq K$ we have,
\beq
|z-w|^2 \geq n^{\sigma} \frac{ n^{\mfa + (K- k ) \delta_K}}{n}.
\eeq
Then,
\beq
\pp\left[ \bigcap_{m=1}^K \EE_m (z) \cap \EE_m (w) \right] \leq \left( \prod_{m=k+1}^K \hat{p}_m (z) \right) \left( \prod_{m=1}^K \hat{p}_m (w) \right) (1 + \O ( n^{-\mfa/20} + n^{-\sigma/10} ) )
\eeq

\eep

\proof The proof is analogous to the proof of Proposition \ref{prop:second-mom-1}, and so omitted. \qed

\bep \label{prop:meso-smm-technical}
We have for $K > 10 / \eps_1$ and $K > 10^6 / \mfc$ that,
\beq
\ee \left[ \left( \sum_{ z_i \in P } \prod_{m=1}^K 1_{ \{ \EE_m (z_i ) \} } \right)^2 \right] \leq \ee \left[ \left( \sum_{ z_i \in P } \prod_{m=1}^K 1_{ \{ \EE_m (z_i ) \} } \right) \right]^2 (1 + \O (n^{- \mfc/200 } + n^{-\eps_1/10} ) )  
\eeq
\eep
\proof Fixing $\frac{ \mfb}{100} > \sigma >0$ we have, (all of the sums below are over $(z, w) \in P \times P$)
\begin{equation}
\begin{split}
  \ee \left[ \left( \sum_{ z_i \in P } \prod_{m=1}^K 1_{ \{ \EE_m (z_i ) \} } \right)^2 \right] &= \sum_{ |z -w |^2 > n^{-\mfb/10} } \pp\left[ \bigcap_{m=1}^K \EE_m(z) \cap \EE_m (w) \right]  \\
 &\quad+  \sum_{ n^{\sigma} n^{-\mfb} < |z-w|^2 < n^{-\mfb/10} }  \pp\left[ \bigcap_{m=1}^K \EE_m(z) \cap \EE_m (w) \right] \\
  &\quad+  \sum_{k=1}^K \sum_{ n^{\sigma} n^{-\mfb - k \delta_K } < |z-w|^2 < n^{\sigma} n^{-\mfb - (k-1) \delta_K } }  \pp\left[ \bigcap_{m=1}^K \EE_m(z) \cap \EE_m (w) \right] \\
  &\quad+  \sum_{ |z-w|^2 < n^{\sigma} n^{\mfa-1} }  \pp\left[ \bigcap_{m=1}^K \EE_m(z) \cap \EE_m (w) \right]
  \end{split}
\end{equation}
By Proposition \ref{prop:second-mom-1} and Lemma \ref{lem:first-moment} we see that,
\beq
 \sum_{ |z -w |^2 > n^{-\mfb/10} } \pp\left[ \bigcap_{m=1}^K \EE_m(z) \cap \EE_m (w) \right] \leq \ee \left[ \left( \sum_{ z_i \in P } \prod_{i=1}^m 1_{ \{ \EE_m (z_i ) \} } \right) \right]^2 (1 + \O (n^{- \mfc/200 } ) )  
\eeq
By Proposition \ref{prop:second-mom-2} with $k=0$ we see that
\begin{equation}
\begin{split}
 & \sum_{ n^{\sigma} n^{-\mfb} < |z-w|^2 < n^{-\mfb/10} }  \pp\left[ \bigcap_{m=1}^K \EE_m(z) \cap \EE_m (w) \right] \lesssim  \sum_{ n^{\sigma} n^{-\mfb} < |z-w|^2 < n^{-\mfb/10} } \left( \prod_{m=1}^K \hat{p}_m (z) \right) \left( \prod_{m=1}^K \hat{p}_m (w) \right)  \\
 \lesssim &  n^{-\mfb/10} ( n^{1-\mfa})^2 (\log n)^{-K} \left( n^{- (1-\mfa) + \eps_1 [ 2(1-\mfa)^{1/2}(1-\mfa - \mfb )^{-1/2} - \eps_1 / (1- \mfa - \mfb ) ]} \right)^2 \\
 \lesssim &  n^{-\mfb/10} \ee \left[ \left( \sum_{ z_i \in P } \prod_{i=1}^m 1_{ \{ \EE_m (z_i ) \} } \right) \right]^2 
\end{split}
\end{equation}
where we used the fact that there are at most $\O ((n^{1-\mfa})^2 n^{-\mfb/10} )$ pairs of points such that $|z-w|^2 < n^{-\mfb/10}$.  The last inequality uses the second part of \eqref{eqn:first-moment-1}. For $1 \leq k \leq K$ we have by Proposition \ref{prop:second-mom-2} that
\begin{equation}
\begin{split}      
  & \sum_{ n^{\sigma} n^{-\mfb - k \delta_K } < |z-w|^2 < n^{\sigma} n^{-\mfb - (k-1) \delta_K } }  \pp\left[ \bigcap_{m=1}^K \EE_m(z) \cap \EE_m (w) \right] \\
 \leq & C n^{\sigma} n^{-\mfb - (k-1) \delta_K} (n^{1-\mfa} )^2 \left( (\log n)^{-K/2}  n^{- (1-\mfa) + \eps_1 [ 2(1-\mfa)^{1/2}(1-\mfa - \mfb )^{-1/2} - \eps_1 / (1- \mfa - \mfb ) ]} \right)^{2 - \frac{k}{K} } \\
 \leq & C n^{2 \sigma} n^{-\mfb - (k-1) \delta_K  +k (1-\mfa)/K - \hat{c} \eps_1 k/K}  \ee \left[ \left( \sum_{ z_i \in P } \prod_{i=1}^m 1_{ \{ \EE_m (z_i ) \} } \right) \right]^2 
\end{split}
\end{equation}
where we denoted $\hat{c} := 2(1-\mfa)^{1/2} (1- \mfa - \mfb)^{-1/2} - \eps_1 / (1- \mfa - \mfb ) \geq 1$ for simplicity, assuming $\eps_1 < 10^{-3}$. The exponent of $n$ in the last line equals,
\beq
2 \sigma - \mfb (1 + K^{-1} ) + \frac{1 - \mfa}{K} + \frac{k}{K} ( \mfb - \hat{c} \eps_1 ) \leq \max \{ 2 \sigma + \frac{1- \mfa - \mfb}{K} - \hat{c} \eps_1 , 2 \sigma - \mfb + \frac{1-\mfa}{K} - \frac{ \hat{c}}{K} \eps_1 \} 
\eeq
Since $\hat{c} \geq 1$, we can take $\sigma < \eps_1 /10$ and $K > 10/\eps_1+ 10/\mfb$ and $\sigma < \mfb /10$ to show that the RHS is less than $- \min\{ \mfb/2, \eps_1/2 \}$. Finally,
\beq
 \sum_{ |z-w|^2 < n^{\sigma} n^{\mfa-1} }  \pp\left[ \bigcap_{m=1}^K \EE_m(z) \cap \EE_m (w) \right] \leq n^{\sigma}  \ee \left[ \left( \sum_{ z_i \in P } \prod_{i=1}^m 1_{ \{ \EE_m (z_i ) \} } \right) \right] \leq n^{-\eps_1/2}  \ee \left[ \left( \sum_{ z_i \in P } \prod_{i=1}^m 1_{ \{ \EE_m (z_i ) \} } \right) \right]^2
\eeq
if $\sigma < \eps_1/10$, as the RHS of  \eqref{eqn:first-moment-2} is at least $n^{9 \eps_1/10}$ if $\hat{c} \geq 1$. The claim follows. \qed

\subsection{Proof of Theorem \ref{thm:meso-smm}}

Define the random variable,
\beq
\Xi := \sum_{ z \in P} \prod_{m=1}^K 1_{ \EE_m (z) } .
\eeq
If $\Xi >0$ then there is a $z \in P$ such that
\beq
\Phi(z) = \sum_{i=1}^K Y_i (z) + \O ( n^{-\mfa/2} )  \geq \sqrt{2} ( \sqrt{1-\mfa-\mfb} \sqrt{1-\mfa} - \eps_1) \log n - n^{-\mfa/2}.
\eeq
where the second equality holds with overwhelming probability. The Paley-Zygmund inequality states that,
\beq
\pp\left[ \Xi > \theta \ee[ \Xi ] \right] \geq ( 1- \theta)^2 \frac{ \ee[ \Xi]^2}{\ee[ \Xi^2] }.
\eeq
The claim now follows from Proposition \ref{prop:meso-smm-technical} and the choice of $\theta = n^{- c\eps_1}$ for some small $c>0$, and the fact that $\ee[ \Xi] \geq n^{\eps_1/2}$ by Lemma \ref{lem:first-moment}. \qed

\section{Second moment method for real i.i.d. matrices} \label{sec:smm-real}

In this section we carry out the second moment method in the real i.i.d. case. Some parts will be similar to the complex case considered in Section \ref{sec:smm-complex}.

We fix here three exponents $\mfa, \mfb, \mfc>0$, let $\mfe := \min \{ \mfa, \mfb\}$. Fix $\alpha >0$ and assume $\mfe +\mfc< \frac{\alpha}{100}$. We consider the set $\mathfrak{P}:= \{ z \in \cc : n^{- \alpha} \leq \Im[z] \leq 2 n^{-\alpha} , | \Re[z] | \leq \frac{1}{2} \}$. We let $P_1$ be a set of $n^{1-\alpha-\mfa}$ well-spaced points of $\mathfrak{P}$, and set $t_\mfb = n^{-\mfb}$. 

The matrix $X_t$ we will consider is the following satisfies $\d X_t = \d B_t/\sqrt{n}$ where $B_t$ is a matrix whose entries are i.i.d. standard real Brownian motions. The initial data is $X_0 = \sqrt{1- t_\mfb} Y $ where $Y$ is a real i.i.d. matrix.

For every $z \in \mathfrak{P}$ we let $\eta_\mfa (z, t)$ be a characteristic such that $\eta_\mfa (z, t_\mfb) = n^{\mfa-1}$ where $t_\mfb = n^{-\mfb}$. Fix a large integer $K>0$. We let,
\beq
\delta^{(1)}_K := \frac{2 \alpha - \mfb - \mfc}{K}, \qquad \delta^{(2)}_K := \frac{1 - \mfa - \mfc -2 \alpha}{K}
\eeq
We let $t^{(2)}_K = t_\mfb  $ and, for $0 \leq i \leq K-1$, let
\beq
t^{(2)}_i = t_\mfb - \frac{ n^{\mfa + (K-i ) \delta^{(2)}_K}}{n}
\eeq
Similarly, for $0 \leq i \leq K$ let,
\beq
t^{(1)}_i = t_\mfb - n^{-2 \alpha + \mfc +(K-i) \delta_K^{(1)}} .
\eeq
Then $0 =: t_0^{(1)} < t_1^{(1)} \dots < t_K^{(1)} < t_0^{(2)} < \dots < t_K^{(2)} := t_\mfb$. Moreover, for $0 \leq j \leq K$,
\beq
\eta_\mfa (z, t^{(2)}_j ) \asymp \frac{ n^{\mfa + (K-j) \delta_K^{(2)}}}{n} = n^{-\mfc - 2 \alpha - j \delta_K^{(2)}}, \quad \eta_\mfa (z, t^{(1)}_j ) \asymp n^{-2 \alpha + \mfc +(K-j) \delta_K^{(1)}} = n^{-\mfb -j \delta_K^{(1)}}.
\eeq
For each $z$ we then define $2K$ random variables as follows:
\beq
\label{eq:defYreal}
Y_j ( z) := \Re \left[ \frac{1}{ \sqrt{2n}} \int_{t^{(1)}_{j-1}}^{t^{(1)}_j} \sum_i \frac{ \d b_i^z (s)}{ \lambda_i^z (s) - \i \eta_\mfa (z, s)} \right], \quad Y_{K+j} ( z) := \Re \left[ \frac{1}{ \sqrt{2n}} \int_{t^{(2)}_{j-1}}^{t^{(2)}_j} \sum_i \frac{ \d b_i^z (s)}{ \lambda_i^z (s) - \i \eta_\mfa (z, s)} \right]
\eeq
for $1 \leq j \leq K$.  Note that $Y_K$ involves an integral over $[t_{K-1}^{(1)}, t_K^{(1)}]$ and $Y_{K+1}$ over $[t_0^{(2)}, t_1^{(2)}]$ and $t_K^{(1)} < t_0^{(2)}$. I.e., we are throwing away a small increment where $\eta_{\mfa} (z, t) \approx n^{-2 \alpha}$ in order to make calculations simpler. 
Let $t_{a K + j } = t^{(1+a)}_{j}$ for $a=0, 1$ and $0 \leq j \leq K-1$, and let $t_{2K} = t_\mfb$. 
We point out that, unlike in the complex case, we introduced two different families of random variables in \eqref{eq:defYreal} to reflect the fact that in the real case the characteristic polynomial consists of the sum of two different fields living on different scales (see e.g. \eqref{eq:corrker} and \eqref{eq:martrepxis1s2}--\eqref{eq:Vtexplicit}).

\bep
\label{pro:secmomreal}
Let $1 \leq k_0 \leq 2 K$ and assume that $|z-w|^2 \geq n^{\sigma} \eta_\mfa (z, t_{k_0} ) $. Then there is a coupling between the random variables $\{ Y_i (z) \}_{i=1}^{2K}$, $\{ Y_i (w) \}_{i=k_0+1}^{2K}$ and a vector of independent Gaussians $\{ Z_i (z) \}_{i=1}^{2K}$, $\{ Z_i (w) \}_{i=k_0+1}^{2K}$ such that
\beq
Y_j (z) = Z_j (z) + \O ( n^{-\sigma/10} + n^{- \mfe /10} ) , \qquad Y_l (w) = Z_l (w) + \O ( n^{-\sigma/10} + n^{-\mfe/10} )
\eeq
for $ 1 \leq j \leq 2K$ and $k_0+1 \leq l \leq 2K$ with overwhelming probability. The variance of the $Z_j(u)$ are given by, 
\beq \label{eqn:real-gaussian-variance}
\Var (Z_{aK+j} (u) ) = \int_{t^{(1+a)}_{j-1}}^{t^{(1+a)}_j} \big[M (u, u, s) + M (u, \bar{u}, s)\big]  \d s =  (2-a) \delta_K^{(1+a)} \log n + \O  (1)
\eeq
for $a = 0, 1$ and $1 \leq j \leq K$, and $u = z$ or $w$ as appropriate. Here $M(z, w, s)$ is defined in \eqref{eqn:M-def}.
\eep
\proof The proof is almost identical to Proposition \ref{prop:coupling-better}, as the main inputs, Proposition \ref{prop:covariation-estimate} and  \eqref{eq:secondbM}, apply also in the real case. The only difference is then the computation of the variance of the Gaussian random variables. For $1 \leq j \leq K$, one uses \eqref{eq:firstbM} and \eqref{eq:thirdbM}. For $j >K$ one uses \eqref{eq:secondbM} instead of \eqref{eq:thirdbM}.
\qed

\

Define,
\beq
\hat{x} := \frac{\sqrt{2}}{K} \left( 2 \alpha - \mfb - \mfc \right) \log n, \qquad \hat{y}:= \frac{\sqrt{2}}{K} \left( 1 - \mfa - \mfc -2 \alpha \right) \log n.
\eeq
Then for $ a = 0, 1$ and $1 \leq m \leq K$, we define the events
\beq
\EE_{aK+m} (z) := \{ Y_{aK+m} (z) > \1_{ \{ a = 0\} } \hat{x} + \1_{ \{ a=1 \}} \hat{y} \} 
\eeq
and
\beq
\FF_{aK+m} (z) := \{ \ZZ_{aK+m} (z) > \1_{ \{ a = 0\} } \hat{x} + \1_{ \{ a=1 \}} \hat{y} \}  ,
\eeq
where $\ZZ_{m}$ are a family of independent Gaussians having variance as in \eqref{eqn:real-gaussian-variance}. 
For $1 \leq m \leq K$, we use the notation
\beq
\hat{p}_m (z) = \pp\left[ \FF_{m} (z) \right], \qquad \hat{q}_m (z) = \pp\left[ \FF_{K+m} (z) \right] .
\eeq
 By \eqref{eqn:basic-gaussian}, we have
\begin{align}
\label{eq:asympdurreal}
\hat{p}_m (z) & \asymp ( \log n)^{-1/2} n^{ - \delta_K^{(1)}/2} =: p, \qquad\quad \hat{q}_m (z)  \asymp ( \log n )^{-1/2} n^{-\delta_K^{(2)}} =: q.
\end{align}
Finally, we define,
\beq
\Xi := \sum_{ z \in P_1} \prod_{m=1}^{2K} \1_{ \EE_m(z) }.
\eeq
\bep
\label{pro:techproreal}
Let $\sigma >0$. For any $z, w$ such that $|z-w|^2 \geq n^{\sigma - \mfb}$ we have
\begin{equation}
\label{eqn:real-smm-1}
\pp\left[ \bigcap_{m=1}^{2K} \EE_m (z) \cap \EE_m (w) \right] = \left( \prod_{m=1}^K \hat{p}_m(z) \hat{p}_m (w) \hat{q}_m(z) \hat{q}_m (w) \right) ( 1 + \O ( n^{-\mfe/10} + n^{-\sigma/10} )).
\end{equation}
Further we have that,
\beq
\pp\left[ \bigcap_{m=1}^{2K} \EE_m (z) \right] = \left(  \prod_{m=1}^K \hat{p}_m (z) \hat{q}_m (z) \right) ( 1 + \O ( n^{-\mfe/10}) ).
\eeq
and so $\ee[ \Xi] \asymp n^{1-\mfa - \alpha } (pq)^{K}$
\eep
\proof The second estimate is proven similarly to \eqref{eqn:first-moment-1}. The first is similar to Proposition \ref{prop:second-mom-1}. For concreteness, we prove the upper bound of \eqref{eqn:real-smm-1}. By Proposition \ref{pro:secmomreal} and the indepence of the Gaussians we have,
\begin{align}
    \pp\left[ \bigcap_{m=1}^{2K} \EE_m (z) \cap \EE_m (w) \right] & \leq \prod_{a=0, 1} \prod_{m=1}^K \bigg\{ \pp\left[  \ZZ_{aK+m} (z) \geq \1_{ \{ a = 0 \} } \hat{x} + \1_{ \{ a=1\}} \hat{y} - n^{-\sigma/10} - n^{-\mfe/10} \right] \notag\\
    & \times \pp\left[ \ZZ_{aK+m} (w) \geq \1_{ \{ a = 0 \} } \hat{x} + \1_{ \{ a=1\}} \hat{y} - n^{-\sigma/10} - n^{-\mfe/10} \right] \bigg\} + n^{-1000} .
\end{align}
We conclude the upper bound using an estimate similar to \eqref{eqn:basic-gaussian-2}. The lower bound is similar. \qed

\

Similarly to the proof of Proposition~\ref{pro:techproreal} we obtain the following bounds. We omit the proof for brevity. 
\bep
For $|z-w|^2 > n^{\sigma} n^{-\mfb - j \delta_K^{(1)}}$ and $1 \leq j \leq K$ we have, 
\begin{align} \label{eqn:real-smm-2}
    \pp\left[ \bigcap_{m=1}^{2K} \EE_m (z) \cap \EE_m (w) \right] \lesssim (pq)^{2K} p^{-j}.
\end{align}
For $|z-w|^2 > n^{\sigma} n^{-\mfc - 2 \alpha - j \delta_K^{(2)}}$ and $0 \leq j \leq K$ we have,
\beq \label{eqn:real-smm-3}
 \pp\left[ \bigcap_{m=1}^{2K} \EE_m (z) \cap \EE_m (w) \right] \lesssim (pq)^{2K} p^{-K} q^{-j}. 
\eeq
The bound for $j=K$ holds for any choice of $z, w$.
\eep

 We are now ready to state and prove the main result of this section: 

\bep \label{prop:real-smm}
There exists a small $c>0$ such that
\beq
\ee[ \Xi^2] \leq \ee[\Xi]^2 ( 1 + n^{-\mfa/100} + n^{-\mfb/100}  ).
\eeq
\eep
\proof Fix a small $\sigma >0$. Write $f(z, w) := \pp\left[ \prod_{m=1}^{2K} \EE_m(z) \cap \EE_m (w) \right]$ and write,
\begin{align} \label{eqn:real-smm-4} 
\ee[ \Xi^2] = \sum_{ |z-w|^2 > n^{\sigma} n^{-\mfb} } f(z, w) + \sum_{j=1}^{2K} \sum_{ \eta_j \leq n^{-\sigma}  |z-w|^2 < \eta_{j-1} } f(z, w) + \sum_{ |z-w|^2 < n^{\sigma} \eta_{2K}} f(z, w)
\end{align}
where $\eta_j = n^{-\mfb - j \delta_K^{(1)}}$ for $0 \leq j \leq K-1$ and $\eta_{i+K} = n^{-\mfc - 2 \alpha - i \delta_K^{(2)}}$ for $0 \leq i \leq K$. By \eqref{eqn:real-smm-1} we have,
\beq
 \sum_{ |z-w|^2 > n^{\sigma} n^{-\mfb} } f(z, w)  \leq \left( \ee[ \Xi] \right)^2 ( 1 + n^{-\sigma/10} + n^{-\mfe/10} ).
\eeq
For $1 \leq j \leq K$ the number of pairs of points $z, w$ such that $|z-w|^2 < n^{\sigma} \eta_{j-1}$ is of order,
\beq
n^{\sigma/2} (n^{1-\mfa})^2 n^{- 2 \alpha}  \sqrt{ \eta_{j-1}}  =n^{\sigma/2}  (n^{1-\mfa  - \alpha} )^2 n^{ - \mfb/2 - (j-1) \delta_K^{(1)}/2}.
\eeq
Therefore, for $1 \leq j \leq K$ we have, by \eqref{eqn:real-smm-2},
\beq
\sum_{ \eta_j \leq n^{-\sigma}  |z-w|^2 < \eta_{j-1} } f(z, w) \lesssim  n^{\sigma/2} ( n^{1 - \mfa - \alpha  } (pq)^K )^2 n^{-\mfb/2}  n^{\delta_K^{(1)}/2} (p^{-1}  n^{-\delta_K^{(1)}/2} )^j \lesssim n^{-\mfb/20} \ee[ \Xi]^2
\eeq
as long as $ \sigma < \mfb / 10$ and $K > 100 / \mfb$.  For $1 \leq j \leq K$ the number of pairs points $z, w$ such that $|z-w|^2 < n^{\sigma} \eta_{K+j-1}$ is bounded by,
\beq
n^{\sigma} (n^{1-\mfa} )^2 n^{-\alpha } \eta_{K+j-1} = n^{\sigma}   (n^{1 - \mfa   - \alpha} )^2 n^{ - \mfc - \alpha} n^{- (j-1) \delta_K^{(2)}}.
\eeq
Therefore,  using \eqref{eqn:real-smm-3}, we have the bound
\begin{align} \label{eqn:real-smm-a1}
  & \sum_{ \eta_{j+K} \leq n^{-\sigma}  |z-w|^2 < \eta_{j+K-1} } f(z, w) \lesssim n^{\sigma}  ( (pq)^K n^{1-\mfa-\alpha } )^2 p^{-K} q^{-j} n^{- \mfc - \alpha} n^{-(j-1) \delta_K^{(2)}}  \notag\\
  \leq & n^{\sigma} ( \ee[ \Xi ] )^2 ( \log n )^{K/2} n^{  - \mfb/2 - 3 \mfc/2} n^{ \delta_K^{(2)}} (q^{-1} n^{- \delta_K^{(2)}} )^j \leq n^{ - \mfb/20} ( \ee[ \Xi] )^2. 
\end{align}
In the second estimate we used that $p^K \gtrsim ( \log n )^{-K/2} n^{ - (2\alpha - \mfb- \mfc )/2 }$ and in the last estimate that $q \geq n^{- \delta_K^{(2)}} ( \log  n)^{-1/2}$. The last term of \eqref{eqn:real-smm-4} is also bounded above by the RHS of \eqref{eqn:real-smm-a1} when $j=K$. This completes the proof, after choosing $\sigma = \mfe /10$.  \qed 

\bet \label{thm:meso-real-smm} There are constants, $c, C>0$ so that the following holds. 
For a real i.i.d. matrix, $P_1$ as above and
\beq
\Phi (z) := \Psi_n (z, t_\mfb , \eta_{\mfa} (z, t_\mfb) ) - \Psi_n (z, 0, \eta_{\mfa} (z, 0) )
\eeq
we have that,
\beq
\pp\left[ \max_{ z \in P_1} \Phi (z) \geq \sqrt{2} \left( 1- \mfb - 2 \mfc - \mfa - C \mfc^{1/3} \right) \log n \right] \geq 1 - n^{ - c \min \{ \mfa, \mfb \} } ,
\eeq
as long as $\mfc, \mfa, \mfb >0$ are sufficiently small and $\mfd < \frac{\mfb}{20}$. 
\eet
\proof By definition, we have
\beq
\Phi (z) = \sum_{m=1}^{2K} Y_m (z) + \left( \Psi_n ( z, \eta_\mfa (t_0^{(2)} ), t_0^{(2)} ) - \Psi_n (z, \eta_\mfa (t_K^{(1)} ) , t_K^{(1)} ) \right),
\eeq
and
\beq
\left| \log \big[ \eta_\mfa (z, t_0^{(2)} ) / \eta_\mfa (z, t_K^{(1)} ) \big] \right| \leq C \mfc \log n.
\eeq
Proposition \ref{prop:short-time-int} thus implies 
\beq
\pp\left[ \left| \Psi_n ( z, \eta_\mfa (t_0^{(2)} ), t_0^{(2)} ) - \Psi_n (z, \eta_\mfa (t_K^{(1)} ) , t_K^{(1)} )  \right| > C \mfc^{1/3} \log n \right] \leq n^{-10}
\eeq
for $\mfc >0$ sufficiently small. The claim now follows from Proposition \ref{prop:real-smm} and the Paley-Zygmund inequality. \qed

\section{Technical lower bound for GDE}

In this section we work towards Theorems \ref{thm:tech-GDE} and \ref{thm:tech-GDE-real} below. They are stronger versions of the lower bounds of Theorems \ref{theo:main} and \ref{theo:realmaxlog} in that they bound below the maximum of the characteristic polynomial over points $z$ where $\lambda_1^z$ is not too small, for i.i.d. ensembles with an $n^{-\eps}$ Gaussian component. The reason for this is that it is hard to apply the four moment method directly to the maximum of the characteristic polynomial with no regularization, i.e., $\Psi_n (z, \eta=0)$. One could compare the maximum at $\eta \approx n^{-1}$, but then this is hard to relate back to the characteristic polynomial with no regularization. Keeping around the condition that $\lambda_1^z$ is not too small allows us to do so; see Lemma \ref{lem:aa-1} in the next section.

\subsection{Regularization}

We need the following notion of regular sets of points.

\bed
We say that a set of point $\P$ in $\cc$ is $\eps_1$-regular if $\P$ can be written as a disjoint union $\P = \bigsqcup \P_i$ where $  \log n \leq | \P_i | \leq 10 \log n $ and for all $i$ and all $w , z \in \P_i$ with $w \neq z$ we have
\beq
\frac{ n^{\eps_1}}{n^{1/2}} \leq |z-w| \leq \frac{ n^{2 \eps_1}}{n^{1/2}}
\eeq
Furthermore, $\P \subseteq \{ z: |z| < 0.99  \}$ and $| \P | \leq N^2$. 
\eed
In this section we will consider $X_t$ to satisfy $\d X_t = \d B_t/\sqrt{n}$ where $B_t$ is a matrix of i.i.d. complex Brownian motions. We will consider a final time $T_c$ and initial data $X_0 = (1- T_c)^{1/2} Y$, where $Y_0$ is either a complex or real i.i.d. matrix. Note that in both cases, the dynamics will be complex. Furthermore, in the real case we will assume that $Y_0$ has a Gaussian component; its size will be specified in the assumptions below.

In this section we will ignore the additional deterministic correction to \eqref{eqn:psi-def}. To avoid confusion, we introduce
\beq \label{eqn:psi-hat-def}
\hat{\Psi}_n (z, t, \eta) = \Re \left( \sum_{i=-n}^n \log ( \lambda_i^z (t) - \i \eta ) - 2 n \int_\rr \log (x  - \i \eta ) \rho_t^z (x) \d x \right) .
\eeq
Here, $\rho_t^z$ is as in \eqref{eqn:mt-def} and $\lambda_i^z$ are the eigenvalues of $H^z (X_t)$ in \eqref{eq:herm} as usual.

The first main technical step of this section is the following proposition connecting the maximum of $\hat{\Psi}_n$ over points where $\lambda_1^z$ is not too small to the maximum of $\hat{\Psi}_n$ regularized on a scale $\eta\gg 1/n$, which can then be estimated using the results from Sections~\ref{sec:smm-complex}--\ref{sec:smm-real}, in the real and complex case respectively. The proof of this proposition is presented in Section~\ref{sec:proffreg}.

\bep \label{prop:gen-reg} There is $C_1 >0$ so that the following holds, with $X_t$ as above. Let $\eps_1, \eps_3$ be sufficiently small and satisfy $\eps_1 < \eps_3/10$. Let $\P$ be an $\eps_1$-regular set of points. Assume that $Y_0$ is either a complex i.i.d. matrix, or a real i.i.d. matrix with Gaussian component of size at least $n^{-\eps_1/2000}$.  Then, for all $\eps_3 >0$ sufficiently small, and $n$ large enough depending on $\eps_1, \eps_3$,
\beq
\pp\left[ \max_{ z \in \P : |\lambda_1^z| \geq (\log n)^{-10} n^{-1} } \hat{\Psi}_n (z, T_c, (\log n)^{-100} n^{-1} ) \geq \max_{z \in \P } \hat{\Psi}_n (z, 0, \hat{\eta} ) - C_1 ( \eps_3 )^{1/3} \log n  \right] \geq 1 - n^{ -50}
\eeq
for $T_c = n^{\eps_3}/ n$ and $\hat{\eta}= n^{\eps_3} / n$. 
\eep

We first require the following. 

\begin{proposition}  \label{prop:eta-inc}
Fix any small $\delta>0$. Fix any small $c_*>0$, let $C_1, C_2 >0$ and $c_* \leq a \leq 1$, and let $(\log n)^{-C_1} n^{-a} \leq \eta_1 \leq \eta_2 \leq ( \log  n )^{C_2} n^{-a}$. Then we have,
\begin{equation}
| \hat{\Psi} (z, t, \eta_1) - \hat{\Psi} (z, t,\eta_2) | \leq (\log n)^{1/2+\delta}
\end{equation}
with overwhelming probability.
\end{proposition}
\proof 
By \eqref{eq:goodll} when $X_0$ is complex and Proposition \ref{prop:m-laws} when $X_0$ is real, for any $\delta >0$ we have
\begin{equation}
\label{eq:precllawlb}
|\langle G_t^z (\ii \eta) - M_t^z ( \ii\eta )\rangle | \leq \frac{ (\log n)^{1/2+\delta}}{n \eta}
\end{equation}
for all $n^{-c_*} \ge \eta \geq (\log n)^{1/2+\delta} / n$. Let $\eta_c := (\log n)^{1/2+\delta} / n \vee \eta_1$. We first consider the case $\eta_c > \eta_1$. It is easy to see that the deterministic part of $\Psi (z, t,\eta_1) - \Psi (z, t,\eta_c)$ contributes only $(\log n)^{1/2+\delta}$. For the random part we have,
\begin{equation}
\begin{split}
0  &\leq \sum_{i=1}^n \log ( (\lambda_i^z(t))^2 + \eta_c^2) - \log ( (\lambda^z_i(t))^2 + \eta_1^2) = n \int_{\eta_1}^{\eta_c} \Im \langle G_t^z (\ii \eta)\rangle \dif \eta  \\
& \leq n \eta_c \int_{\eta_1}^{\eta_c}  \eta^{-1} \Im \langle G_t^z (\ii \eta_c)\rangle \dif\eta \\
&= (n \eta_c) \int_{\eta_1}^{\eta_c}  \eta^{-1} \Im \langle G_t^z (\ii \eta_c)-M^z(\ii\eta_c)\rangle \dif \eta + \O ( (\log n)^{1/2+\delta} \log \log n ) \\
&\leq C \log \log n (\log n)^{1/2+\delta}
\end{split}
\end{equation}
where in the last step we used \eqref{eq:precllawlb} with $\eta = \eta_c$. Finally, we estimate
\begin{equation} \label{eqn:ccc-1}
| \hat{\Psi} (z, \eta_c) - \hat{\Psi} (z, \eta_2) | \le n \int_{\eta_c}^{\eta_2} | \langle G_t^z (\ii \eta)-M^z(\ii\eta)\rangle | \dif \eta \leq  C (\log n)^{1/2+\delta} \log \log n,
\end{equation}
where we used again \eqref{eq:precllawlb} to estimate the integral. This completes the proof in the case $\eta_c > \eta_1$. If $\eta_c = \eta_1$, then the estimate follows from \eqref{eqn:ccc-1}. \qed

\

We now show that the small singular values of $X-z$ are asymptotically independent for $z$'s sufficiently away from each other. This follows in a straightforward manner from \cite[Section 7]{cipolloni2023central} (see the proof in Appendix~\ref{a:approx-ind}).

\begin{lemma} \label{lem:ind-prod}
Fix $r<1$ and a small $\eps>0$, and let $J$ be a set of at most $ \O (\log n ) $ points, with $|z| \le r$, which are all at least $n^{-1/2+\eps}$ from each other. Let $H_t^z$ be the Hermitization of $X_t-z$, with $X_t$ as above, with $Y_0$ either a complex i.i.d. matrix or a real i.i.d. matrix with size $n^{-\eps/2000}$ Gaussian component.

Then for any $\frac{1}{2} > \eps_2 > 0$, let $\lambda_i ^z(t)$ be the positive eigenvalues of $H_t^z$, with $ t= n^{-1+\eps_2}$.  For any $C>0$ and for all $(\log n)^{-C} \leq s \leq 1$ it holds 
\begin{equation} \label{eqn:ind-prod}
\mathbb{P}\left[ \bigcap_{z \in J} \{ \lambda_1^z (t) \leq s n^{-1} \} \right] \lesssim \prod_{z \in J} \mathbb{P}\left[ \mu_1^z \leq 2 s n^{-1} \right] + n^{-100}. 
\end{equation}
where the $\mu_1^z$ are the singular values of the shifted complex Ginibre ensemble.
\end{lemma}

And immediate corollary of the above is the following.

\bep \label{prop:g-reg-1}
Let $0 < r < 1$ and $\eps_1,\eps_3 >0$ such that $\eps_1 < \eps_3 / 10$. Let $\P = \bigsqcup_{i=1}^K P_i$ be $\eps_1$- regular. Let $ T_c= n^{\eps_3-1}$, and $X_t$ as above. Then if $Y_0$ is either a complex i.i.d. matrix or a real i.i.d. matrix with Gaussian component of size at least $n^{-\eps_1/2000}$, then we have
\beq
\pp\left[ \bigcap_{i=1}^K \left\{ \exists z \in P_i : \lambda_1^z (T_c) \geq (\log n)^{-10} n^{-1} \right\} \right] \geq 1- n^{-90}.
\eeq 
\eep
\proof For $1 \geq s \geq (\log n)^{-C}$ we have the estimate, for $\mu_i^z$ being the singular values of the shifted complex Ginibre ensemble,
\beq
\label{eq:desc}
\mathbb{P}\left[ \mu_1^z \leq 2 s n^{-1} \right] \leq C s,
\eeq
by \cite[Eq. (4a)]{cipolloni2020optimal}. Therefore, by Lemma \ref{lem:ind-prod} we have that 
\beq
\pp\left[ \bigcap_{ z_i \in P_i } \left\{ \lambda_1^{z_i} (T_c) < (\log n)^{-10} n^{-1} \right\} \right] \leq n^{-100} + (C/\log n )^{ \log n},
\eeq
and so the claim follows. \qed

\subsubsection{Proof of Proposition \ref{prop:gen-reg}}
\label{sec:proffreg}

First, we see from Proposition \ref{prop:g-reg-1} that for each $P_i$ there exists a $w_i \in P_i $ so that $|\lambda_1^{w_i}| \geq (\log n)^{-10} n^{-1}$, with probability at least $1 - n^{-90}$. So we bound
\beq
\max_{ z \in \P : |\lambda_1^z| \geq (\log n)^{-10} n^{-1} } \hat{\Psi}_n (z, T_c, (\log n)^{-100} n^{-1} ) \geq \max_{ \{ w_i \}_i } \hat{\Psi}_n (w_i, T_c, (\log n)^{-100} n^{-1} )
\eeq
with probability at least $ 1- n^{-90}$. Letting now $\eta_2 = (\log n)^{C_2}  / n$ for some large $C_2 >0$ we see from Proposition \ref{prop:eta-inc} that,
\beq
\max_{ \{ w_i \}_i } |\hat{\Psi}_n (w_i, T_c, (\log n)^{-100} n^{-1} ) - \hat{\Psi}_n (w_i, T_c, \eta_2 ) | \leq (\log n)^{3/4}
\eeq
with overwhelming probability. Taking $C_2 = 100$, we now see, from Proposition \ref{prop:short-time-int} and choosing characteristics $\eta^{(i)}_s=\eta(s,w_i)$ ending at $\eta^{(i)}_t = \eta_2$ for each $w_i$, that we have
\beq
\max_{ \{ w_i \}_i } | \Psi_n (w_i, t, \eta_2 ) - \Psi_n (w_i, 0, \eta^{(i)}_0 ) | \leq (\eps_3)^{1/3} \log n,
\eeq
with probability at least $ 1- n^{-100}$, if $\eps_3 >0$ is sufficiently small. Letting now $\hat{\eta} = n^{\eps_3} / n$, using $\eta^{(i)}_0 \asymp t$, by Proposition \ref{prop:eta-inc}, we have  that
\beq
\max_{ \{ w_i \}_i } | \Psi_n (w_i, 0, \hat{\eta} ) - \Psi_n (w_i, 0, \eta^{(i)}_0 ) |  \leq (\log n)^{3/4},
\eeq
with probability at least $ 1- n^{-100}$. Now for any fixed $P_i$ we have for all $z, w \in P_i$ by Proposition \ref{prop:der-bd} that with overwhelming probability,
\beq
| \Psi_n (w, 0, \eta_3 ) - \Psi_n (z, 0, \eta_3 ) | \leq \frac{ n^{4 \eps_1}}{n^{\eps_3/2}} \leq n^{-\eps_3/10}
\eeq
and so the desired estimate follows. \qed

\subsection{Technical lower bound for GDE}

In this section we will develop technical lower bounds for the log characteristic polynomial of Gaussian divisible ensembles. We first deal with the complex i.i.d. case. Fix now $\eps_1 >0$ and $\eps_3 >0$ with $\eps_1 = \frac{\eps_3}{10}$. We construct a specific $\eps_1$-regular set $\hat{\P}$ as follows. First, let $\P_1$ be a well-spaced set of $n^{1-\eps_3}$ points of the disc $\{ | z- \frac{1}{2} \i | < \frac{1}{4} \}$. In particular for all distinct $z, w \in \P_1$ we have $|z-w|^2 \geq c n^{\eps_3-1}$. Then, around each $z \in \P_1$ we add a set of $P_i$ points of size $|P_i | = \log n$, such that for all distinct $z_1, z_2 \in \P_1 \cup \{ z \}$ we have $n^{\eps_1-1/2} \leq |z_1 -z_2 | \leq n^{2 \eps_1 -1/2}$. We let $\hat{\P}$ be the union of all of the $P_i$ and $\P_1$. 

\bet \label{thm:tech-GDE}
There are $c, C>0$ so that the following holds. Let $X$ be  a matrix of the form $X= (1-T)^{-1/2}Y + T^{1/2} G$ where $T = n^{-\mfb} + n^{\eps_3-1}$, $G$ is a complex Ginibre matrix and $Y$ is a complex i.i.d. matrix. Then, we have
\beq \label{eqn:tech-1}
\pp\left[ \max_{ z \in \hat{\P} : | \lambda_1^z | \geq (\log n)^{-10} n^{-1} } \Psi_n (z,  (\log n)^{-100} n^{-1} ) \leq \sqrt{2} (1 - C ( \mfb + (\eps_3)^{1/3} )  ) \log n \right] \leq n^{-c \mfb} + n^{- c \eps_3 } . 
\eeq
\eet
\proof Let $T_3 = n^{\eps_3} /n$. We can consider $X$ as the solution at time $T_3$ of $\d X_s = \d B_s/\sqrt{n}$ with $B_s$ a matrix of i.i.d. complex Brownian motions and $X_0 = (1- T_3)^{1/2} Y_0$ with $Y_0$ an i.i.d. matrix with Gaussian component of size $n^{-\mfb}$.  Let $\hat{\Psi}_n (z, s, \eta)$ denote the characteristic polynomial of $X_s$, as in \eqref{eqn:psi-def}. Then the observable $\Psi_n$ in the probability in \eqref{eqn:tech-1} is given by $\hat{\Psi}_n (z, T_3, ( \log n)^{-100} n^{-1})$.  By Proposition \ref{prop:gen-reg} we have that,
\beq
\max_{ z \in \hat{\P} : | \lambda_1^z | \geq (\log n)^{-10} n^{-1} } \hat{\Psi}_n (z, T_3, (\log n)^{-100} n^{-1} )  \geq \max_{z \in \hat{\P} } \hat{\Psi}_n (z, 0, n^{\eps_3} / n )- C_1 \eps_3^{1/3} \log n,
\eeq
with probability at least $ 1- n^{-50}$. We then lower bound,
\beq
\label{eq:almostendlb}
\max_{z \in \hat{\P} } \hat{\Psi}_n (z, 0, n^{\eps_3} / n ) \geq \max_{z \in \P_1  } \hat{\Psi}_n (z, 0, n^{\eps_3} / n )
\eeq 
We want to apply Theorem \ref{thm:meso-smm} to the quantity on the RHS of \eqref{eq:almostendlb}. However, it is the log characteristic polynomial of an i.i.d. matrix with variance $(1-T_3) n^{-1}$ which is not exactly $n^{-1}$. Nonetheless, for any $a>0$ and matrix $M$, we have that,
\beq
\log \left| \det \left( \begin{matrix} - \i \eta  & a M - z \\ a M^* - z & -\i \eta \end{matrix} \right) \right| = 2n a  + \log \left| \det \left( \begin{matrix} - \i \eta a^{-1} &  M - z a^{-1} \\  M^* - z a^{-1} & -\i \eta a^{-1} \end{matrix} \right) \right|
\eeq 
and by the definition of $\rho_t^z(x)$ in \eqref{eqn:mt-def}, with $c_* = \sqrt{1 -T_3}$, 
\beq
n \int \log ( x^2 + \eta^2) \rho_0^z (x) \d x = 2n \log c_* + \int \log (x^2 + (\eta /c_*)^2 ) \rho^{z / c_*} (x) \d x 
\eeq
after a rescaling. Therefore, if $\tilde{\Psi}_n (z, \eta)$ is the log-characteristic polynomial of the matrix $Y_0$ we have that $\hat{ \Psi}_n (z, 0, \eta) = \tilde{\Psi}_n (z / c_*, \eta / c_*)$. Note that after the rescaling by $c_* \asymp 1$, the set $\P_1$ remains a well-spaced subset of the unit disc. We denote this set by $\tilde{\P}_1$. We let now $T_\mfb = n^{-\mfb}$ and assume that $Y_0$ is equal to $\tilde{X}_{T_\mfb}$ where $\d \tilde{X}_s= \frac{ \d \tilde{B}_s}{\sqrt{n}}$ where $\tilde{B}_s$ is a matrix of complex i.i.d. Brownian motions, and $\tilde{X}_0 = (1- T_\mfb )^{1/2} \tilde{Y}_0$, where $\tilde{Y}_0$ is an i.i.d. complex matrix. Denote the characteristic polynomial of $\tilde{X}_s$ by $\tilde{\Psi}_n (z, s, \eta)$. We write now,
\beq
 \tilde{\Psi}_n (z, T_\mfb, n^{\eps_3} / n ) = \left( \tilde{\Psi}_n (z, T_\mfb, n^{\eps_3} / n ) - \tilde{\Psi}_n (z, 0, \eta_z ) \right) + \tilde{\Psi}_n (z, 0, \eta_z )
\eeq
where $\eta_z$ is a characteristic that ends at $n^{\eps_3}/n$ at time $T_\mfb$. We have that $\eta_z \asymp n^{-\mfb}$. From Proposition \ref{prop:global} we see that
\beq
\max_{z \in \tilde{\P}_1} | \tilde{\Psi}_n (z, 0, \eta_z ) | \leq C \mfb \log n,
\eeq
with probability at least $1  - n^{-c\mfb}$ if $\mfb >0$ is sufficiently small. Therefore on this event,
\beq
\max_{z \in \tilde{\P}_1  } \tilde{\Psi}_n (z, t_1, n^{\eps_3} / n ) \geq \max_{z \in \P_1  } \left( \tilde{\Psi}_n (z, t_1, n^{\eps_3} / n ) - \tilde{\Psi}_n (z, 0, \eta_z ) \right)  - C \mfb  \log n.
\eeq
On the other hand, Theorem \ref{thm:meso-smm} then applies to the max on the right--hand--side. We see that for any sufficiently small $\eps_2 >0$ we have
\beq
\max_{z \in \tilde{\P}_1  }\left( \tilde{\Psi}_n (z, t_1, n^{\eps_3} / n ) - \tilde{\Psi}_n (z, 0, \eta_z ) \right)  \geq \sqrt{2} \big( 1 - C ( \eps_2 + \eps_3 + \mfb ) \big) \log n,
\eeq
with probability at least $ 1 - n^{-c\eps_2}$. The claim follows. \qed

We now deal with the real i.i.d. case.  We will develop a lower bound for points in the rectangle consisting of $z$ such that $\Im[z] \asymp n^{-\alpha}$.  Again, fix $\eps_1, \eps_3 >0$ with $\eps_1 = \frac{ \eps_3}{10}$. Assume also that $\eps_3 < \frac{ \alpha}{1000}$. Let $\P_1$ be a set of $n^{1-\alpha - \eps_3}$ well-spaced points of the rectangle $\{ z : |\Re (z) | \leq \frac{1}{2}, n^{-\alpha} \leq \Im[z] \leq 2 n^{-\alpha} \}$. Note for distinct $z, w \in \P_1$ we have that $|z-w|^2 \geq c n^{\eps_3-1}$. For each $z \in \P_1$ we add a set of $P_i$ of size $|P_i| =  \log n$ such that for distinct $z_1, z_2 \in P_1 \cup \{ z \}$ we have $n^{\eps_1-1/2} \leq |z_1 -z_2 | \leq n^{2 \eps_2 - 1/2}$.  We define $\hat{P}$ to be the union of all of the $P_i$ as well as $\P_1$. 

\bet \label{thm:tech-GDE-real}
There are $c, C>0$ so that the following holds. Let $T_\mfb = n^{-\mfb}$ and $T_3 = n^{\eps_3-1}$ with $\mfb >0$ sufficiently small, satisfying $\mfb < 10^{-6} \eps_3$. Let $X = (1-T_\mfb - T_3)^{1/2} Y + \sqrt{T_\mfb} G_r + \sqrt{T_3} G_c$ where $Y$ is a real i.i.d. matrix, and $G_r$ and $G_c$ are from the real and complex Ginibre ensembles, respectively. Then,
\beq \label{eqn:tech-2}
\pp\left[ \max_{ z \in \hat{\P} : | \lambda_1^z | \geq (\log n)^{-10} n^{-1} } \Psi_n (z,  (\log n)^{-100} n^{-1} ) \leq \sqrt{2} (1 - C \eps_3^{1/3}  ) \log n \right] \leq n^{-c \mfb}.
\eeq
Above, $\Psi_n (z, \eta)$ is as in \eqref{eqn:psi-def} with $\beta=1$. 
\eet
\proof The proof is similar to Theorem \ref{thm:tech-GDE} and so we focus on the differences. We first consider $X$ as the solution at time $T_3$ of $\d X_s =\d B_s/\sqrt{n}$ with $B_s$ a matrix of i.i.d. complex Brownian motions and $X_0 = (1- T_3)^{1/2} Y_0$ with $Y_0$ a real i.i.d. matrix with a real Gaussian component of size of order $T_\mfb$. 

Denote by $\hat{\Psi}_n (z, t, \eta)$ the log characteristic polynomial of $X_t$ as in \eqref{eqn:psi-hat-def} without the additional $\beta=1$ deterministic correction that appears in \eqref{eqn:psi-def}. Arguing as in the proof of Theorem \ref{thm:tech-GDE} we have by Proposition \ref{prop:gen-reg} that,
\beq
\max_{ z \in \hat{\P} : | \lambda_1^z | \geq (\log n)^{-10} n^{-1} } \Psi_n (z,  (\log n)^{-100} n^{-1} ) \geq \max_{ z \in \P_1} \left( \hat{ \Psi}_n (z, 0 , n^{\eps_3}/n) + \alpha \log n \right) - C_1  (\eps_3 )^{1/3} \log n .
\eeq
Let now $Y_0$ be $\tilde{X}_{T_\mfb}$ where $\d \tilde{X}_s = \frac{ \d \tilde{B}_s}{ \sqrt{n}}$ where $\tilde{B}$ is a matrix of i.i.d. real Brownian motions, with $\tilde{X}_0 = (1- T_\mfb)^{1/2} \tilde{Y}$ with $\tilde{Y}$ a real i.i.d. matrix. Denote by $\tilde{\Psi}_n (z, s, \eta)$ its log characteristic polynomial as in \eqref{eqn:psi-def}, including the $\beta=1$ correction term. By a rescaling (similar to in the proof of Theorem \ref{thm:tech-GDE}) we have,
\beq
\max_{ z \in \P_1} \left( \hat{ \Psi}_n (z, 0 , n^{\eps_3}/n) + \alpha \log n \right) = \max_{z \in \tilde{P}_1} \left( \tilde{ \Psi}_n (z, T_\mfb , c_*^{-1} n^{\eps_3} / n \right) + \O (1) 
\eeq
where $c_* \asymp 1$ is a constant and $\tilde{P}_1$ is a set of $n^{1- \alpha-\eps_3}$ well spaced points lying in the rectangle $\{ z : | \Re (z) | \leq \frac{3}{4}, n^{-\alpha} \leq  2 \Im[z] \leq 5 n^{-\alpha} \}$. We conclude the proof similarly to Theorem \ref{thm:tech-GDE}, using now Theorem \ref{thm:meso-real-smm} (taking the $\mfc>0$ in that theorem to be $\mfc = \mfb^3$). \qed

\section{Lower bound for $\Psi_n (z)$}

We now remove the Gaussian component in the lower bound in Theorem \ref{thm:tech-GDE}. For this purpose we use a comparison argument (see, e.g., Proposition~\ref{pro:GFTmax} below). The following deterministic lemma is a straightforward consequence of writing $\langle G^z ( \i \eta) \rangle$ in terms of eigenvalues and so the proof is omitted.
\begin{lemma} \label{lem:reg-compare} The following holds deterministically for any matrix of the form $\left( \begin{matrix} 0 & M -z\\ M^*-\bar{z} & 0 \end{matrix} \right)$  with resolvent $G^z ( \i \eta)$ and eigenvalues $\lambda_i^z$. First, for any $\eta>0$, we have
\beq
 \label{eqn:bb-1}
n \eta \Im [ \langle G^z (\ii \eta)\rangle ] < \frac{1}{10} \implies \lambda_1(z) > \eta
\eeq
Second, assume that the bound,
\beq
N \tilde{\eta} \Im[ \langle G^z ( \i \tilde{\eta} ) \rangle ] \leq (\log n)^4,
\eeq
holds for $\tilde{\eta} = ( \log n)^2/n$. Then if $\lambda_1(z) > n^{-1} (\log n)^{-10}$ we have that,
\beq
n \eta_1 \Im[ \langle G^z ( \i \eta_1 ) \rangle ] \leq (\log n)^{-1},
\eeq
if $\eta_1 = (\log n)^{-20} / n$.
\end{lemma}
For this section we will denote by $Q$ be a smooth function which is equal to $1$ for $|x| < 1/20$ and equal to $0$ for $|x| >1 /10$. Let now,
\beq
\eta_1 := \frac{1}{n (\log n)^{20}} .
\eeq
For any i.i.d. ensemble we know that $N \eta_2 \Im[ \langle G^z ( \i \eta_2 ) \rangle ] \leq (\log n)^4$ with overwhelming probability with $\eta_2 = (\log n)^2 / n$ by \eqref{eq:goodll}. Therefore, by Lemma \ref{lem:reg-compare}, with overwhelming probability on the event $\lambda^z_1 \geq (\log n)^{-10} n^{-1}$, we have
\beq
Q ( n \eta_1 \Im  \langle G^z (\ii \eta_1)\rangle ) = 1.
\eeq
Let now $\hat{\eta} := (\log n)^{-100} n^{-1}$. We have the following. 
\bel \label{lem:reg-compare-2}
Let $X$ be an i.i.d. matrix and $\hat{P}$ a set of points. Let $\mfb >0$ and $\eps_3 >0$ be sufficiently small. There are constants $c, C>0$ so that,
\beq 
\label{eqn:bb-3}
\pp\left[ \max_{z \in \hat{P}} Q ( n\eta_1 \Im  \langle G^z (\ii \eta_1)\rangle) \left( \Psi_n (z, X,\hat{\eta}) - \Psi_n (z, X, n^{-\mfb} ) \right)  \leq  \sqrt{2} \left( 1 - C  \eps_3^{1/3}  \right) \log n\right] \leq n^{- c \mfb} 
\eeq
in the case that:
\begin{enumerate}
    \item $X$ is a complex i.i.d. matrix with Gaussian component of size at least $n^{-\mfb}$ and $\hat{P}$ is a certain set of at most $n$ points of the unit disc; in this case $\eps_3 = \mfb$. 
    \item $X = (1- n^{\eps_3-1} - n^{-\mfb} )^{1/2} Y + n^{\eps_3/2-1/2} G_c + n^{-\mfb/2} G_r$ where $G_r, G_c$ are real and complex Ginibire matrices and $Y$ is a real i.i.d. matrix, and $\hat{P}$ is a certain set of at most $n$ points of the strip $\{ z : n^{-\alpha} \leq  2\Im[z] \leq 5 n^{-\alpha}, |\Re[z] | \leq \frac{3}{4} \}$, with $\alpha\in (0,1/2)$, and $\mfb < 10^{-6} \eps_3$. In this case, $\Psi_n (z, X, \eta_3)$ is as in \eqref{eqn:psi-def} with the $\beta=1$ deterministic correction.
\end{enumerate}
\eel
\noindent{\bf Remark.} \emph{The second case holds also for  a set of points in the region $\Im[z] \geq c$, $|z| \leq 1- c$ for and small $c>0$, as can be seen from the proof. However, we will not need this as we will only prove our lower bounds near the real axis.}

\

\proof[Proof of Lemma~\ref{lem:reg-compare-2}]  By the discussion immediately preceding this lemma and Theorems\ref{thm:tech-GDE}--\ref{thm:tech-GDE-real} we see that the lemma holds if the term $\Psi_n (z, X, n^{-\mfb} )$ is not present in \eqref{eqn:bb-3}. So we need only to show that this term contributes $\O ( \eps_3^{1/3} \log n )$ to the max. For the complex case, this follows from Proposition \ref{prop:global} (recall that in that case $\eps_3 = \mfb$). For the real case, we first use Proposition \ref{prop:short-time-int} to remove the complex Ginibire contribution at the cost of $\O (\eps_3^{1/3} \log n)$, leaving us to control the max of the log characteristic polynomial of a real i.i.d. matrix at scale $\eta \asymp n^{-\mfb}$. This follows also from Proposition \ref{prop:global}. \qed

\

We are now ready to state our comparison argument for the quantity on the LHS of \eqref{eqn:bb-3}; the proof of this result is presented in Appendix~\ref{app:GFTmax}. We denote by $\Xi (X)$ the function in the probability on the LHS of \eqref{eqn:bb-3}, considered as a function on the space of matrices to $\rr$,
\beq
\Xi (X) := \max_{z \in \hat{P}} Q ( n\eta_1 \Im  \langle G^z (\ii \eta_1)\rangle) \left( \Psi_n (z, X,\eta_3) - \Psi_n (z, X, n^{-\mfb} ) \right) 
\eeq
We have the following moment matching result for $\Xi$. The proof is presented in Appendix~\ref{app:GFTmax}.
\begin{proposition}
\label{pro:GFTmax}
Let $F$ be a Schwarz function. Let $X_1$ and $X_2$ be two either real or complex i.i.d. matrices matching the moments up to third order and matching the fourth moment up to $n^{-2} t$. That is,
\beq
\left| \ee\left[ (X_1)_{ij}^a (\bar{X}_1)_{ij}^b \right] - \ee\left[ (X_2)_{ij}^a (\bar{X}_2)_{ij}^b \right] \right| \leq \1_{ a+b=4} t n^{-2}
\eeq
for all $0 \leq a + b \leq 4$ and $i,j$. Then for all $\eps >0$ we have
\beq
\left| \E[ F ( \Xi (X_1) ] - \E[ F ( \Xi (X_2) ) ] \right| \leq \| F \|_{C^5}( n^{-\eps} + n^{10\eps} t ).
\eeq
\end{proposition}

In the real i.i.d. case we first need to remove the small complex Gaussian component. For this we use the following set up. We let $X_0$ be a real i.i.d. matrix with a size $n^{-\eps}$ Gaussian component, and let $\d X_s = \frac{ \d B_s}{ \sqrt{n}} - \frac{ X_s}{2} \d t$ where $B_s$ is a matrix of complex Brownian motions. The proof is presented in Appendix~\ref{a:small-complex-comp}. We point out that a proof similar in spirit to this one was performed in \cite{xu2024bulk} in a different setting.

\bep \label{prop:GFTmax-2}
Let $X_t$ and $\Xi$ be as above, and let $F$ be a Schwarz function. Assume that the points $\hat{P}$ satisfy $\Im[z] \geq n^{c_1 - 1/2}$ for some $c_1 >0$ and all $z \in \hat{P}$. As long as $\eps >0$ satisfies $\eps < \frac{ c_1}{2000}$ we have,
\beq
\left| \ee[ F( \Xi (X_t))] - \ee[ F(\Xi (X_0))] \right| \leq C \| F \|_{C^4} \left( (1 + nt) n^{-c_1/100} + n^{3/4} t \right).
\eeq
\eep

Finally, we rely on the following lemma to remove the last bit of regularization in $\Xi$.

\bel \label{lem:aa-1} Let $X$ be a real or complex i.i.d. matrix, and let $\Psi_n (z, \eta)$ denote its log characteristic polynomial in \eqref{eqn:psi-def}. For any $D>0$ and any $C_2 > 50$ we have
\beq
\pp\left[ \lambda_1^z \geq (\log n)^{-20} n^{-1} , | \Psi_n (z, n^{-1} (\log n)^{-C_2} ) - \Psi_n (z, 0 ) | > (\log n)^{-10} \right] \leq n^{-D}
\eeq
for $n$ large enough and all $|z| <r$.
\eel

\proof Let $\eta_2 := (\log n)^{-C_2} n^{-1}$ for notational simplicity. 
We have that,    
\beq
\label{eq:estdetpart}
\del_\eta \int \log (x^2 + \eta^2) \rho^w (x) \dif x = \Im[ m^w ( \ii \eta ) ] \le C,
\eeq
so the deterministic contribution to $\Psi_n (z, 0) - \Psi (z, \eta_2)$ is less than $C N \eta_2 \leq C (\log n)^{-50}$. Using the notation $\lambda_i = \lambda_i^z$, we have 
\begin{align}
\sum_i \left| \log ( \lambda_i^2 + \eta_2^2) - \log ( \lambda_i^2) \right| \leq  C \eta_2^2 \sum_i  \frac{1}{ \lambda_i^2}.
\end{align}
When $\lambda_1 \geq (\log n)^{-20} n^{-1}$ we can bound
\beq
\sum_{|i| < (\log n)^{10} } \frac{ \eta_2^2}{ \lambda_i^2} \leq (\log n)^{-50}.
\eeq
For $i < n^{1-1/100}$ we have that $| \lambda_i - \gamma_i | \leq \frac{ (\log n)^2} {n}$ with overwhelming probability, and since $\gamma_i \asymp i/n$ we have,
\beq
\sum_{ (\log n)^{10} < i < n/2} \frac{ \eta_2^2}{ \lambda_i^2} \leq \frac{C}{(\log n)^{100}} \sum_{i > 1 } \frac{1}{i} \leq C (\log n)^{-98}.
\eeq
Finally, if $i > n^{1-1/100}$ then $\lambda_i \geq c(i/n)$ so $\sum_{i > n/2} \frac{ \eta_2^2}{ \lambda_i^2} \leq n^{-1/2}$. The claim now follows. \qed

\subsection{Proof of lower bounds of Theorem \ref{theo:main} and \ref{theo:realmaxlog}}

We begin with the lower bound of Theorem \ref{theo:realmaxlog}. The lower bound of \eqref{eqn:realmaxlog-1} follows from that of \eqref{eqn:realmaxlog-2}, so we only prove the latter. Let $0  < \alpha < \frac{1}{2}$.  We will first prove that the lower bound holds for real i.i.d. matrices with sufficiently large Gaussian component. I.e., we will remove the complex Gaussian component from Lemma \ref{lem:reg-compare-2} by applying Proposition \ref{prop:GFTmax-2}.

Fix $\eps_3 >0$ and $\mfb >0$ with $\mfb = 10^{-7} \eps_3$. Let $X_0$ be a real i.i.d. matrix with Gaussian component of size at least $n^{-\mfb}$. Let $X_t = (1- \e^{-t} )^{1/2} X_0 + \e^{-t/2} G_c$ where $G_c$ is complex Ginibre matrix, and $ t= n^{\eps_3-1}$.  By Lemma \ref{lem:reg-compare-2} it holds that,
\beq \label{eqn:real-compare-1}
\max_{z \in \hat{P}} Q( n \eta_1 \Im \langle G( \i \eta_1 ) \rangle ) \left( \Psi_n (z, X_t, \hat{\eta} ) - \Psi_n (z, X_t, n^{-\mfb} ) \right) \geq \sqrt{2} (1 - C \eps_3^{1/3} ) \log n 
\eeq
with probability at least $ 1- n^{-c\eps_3}$. As long as $\eps_3 < 10^{-3} ( \frac{1}{2} - \alpha)$ (recalling $\mfb = 10^{-7} \eps_3 $), we see that the above holds for $X_0$ as well, by applying Proposition \ref{prop:GFTmax-2} (with $F$ being some appropriate smoothed indicator function). Hence, \eqref{eqn:real-compare-1} holds for any real i.i.d. ensemble with a sufficiently large Gaussian component, of size $n^{-c_* \eps_3}$. Now, given any real i.i.d. ensemble $Y$, we may find a Gaussian divisible $X$ whose first three moments match with $Y$ and the fourth moment matches up to order $n^{-2- c_*\eps_3}$. Hence, by Proposition \ref{pro:GFTmax}, we have that \eqref{eqn:real-compare-1} holds also for the matrix $Y$, with probability at least $1 - n^{- c_* \eps_3/2}$.  

Now, by applying Proposition \ref{prop:global}, we see that for the matrix $Y$, 
\beq
\label{eqn:bb-4}
\max_{z \in \hat{P}} Q ( n\eta_1 \Im  \langle G^z (\ii \eta_1)\rangle ) \Psi_n (z, Y, \hat{\eta}) \geq  \sqrt{2} ( 1 - C \eps_3^{1/3} ) \log n
\eeq
with probability at least $ 1 - n^{ -c \eps_3}$ after possibly adjusting the constants $C, c>0$. We may assume that $C (\eps_3)^{1/3} < \frac{1}{2}$, so the RHS is positive.

To conclude the proof we are thus left only with removing the $\hat{\eta}$--regularization above, which we now turn to. Since the RHS of \eqref{eqn:bb-4} is positive, some points on the LHS must be non-zero. Any point on the LHS which is non-zero must have that $\Psi_n (z, \hat{\eta} ) >0$ and that $n\eta_1 \Im  \langle G^z (\ii \eta_1)\rangle < \frac{1}{10}$. Then by \eqref{eqn:bb-1} we have that $\lambda_1^z > \eta_1$ for such $z$'s. Then by  Lemma \ref{lem:aa-1}, with overwhelming probability, for all such $z$'s such that $Q \Psi$ is non-zero we have
\beq
| \Psi_n (z , \hat{\eta} ) - \Psi_n (z, 0) | \leq \frac{1}{\log n}.
\eeq
Hence, if $z \in \hat{P}$ is a maximizing point in \eqref{eqn:bb-4} we have
\beq
Q ( n\eta_1 \Im  \langle G^z (\ii \eta_1)\rangle ) \Psi_n (z,  \hat{\eta})  \leq Q ( n\eta_1 \Im  \langle G^z (\ii \eta_1)\rangle ) \Psi_n (z,0) + ( \log n)^{-1} \leq  \Psi_n (z, 0) + ( \log n )^{-1},
\eeq
using that $0 \leq Q \leq 1$. We thus conclude the lower bound in \eqref{eqn:realmaxlog-2}.

The proof of the lower bound of Theorem \ref{theo:main} is similar but easier. We start from the first part of Lemma \ref{lem:reg-compare-2} and proceed as in the above proof, except that we do not need the intermediate Proposition \ref{prop:GFTmax-2}. Instead, we conclude directly that \eqref{eqn:bb-4} holds for any complex i.i.d. matrix using directly Proposition \ref{pro:GFTmax}. The rest is the same. \qed

\appendix

\section{Proof of Lemma~\ref{lem:ind-prod}} \label{a:approx-ind}

The proof of this lemma follows closely the proof of \cite[Section 7]{cipolloni2023central}, with a few very minor modifications. For this reason we only present a sketch of the proof and highlight the differences. The arguments in \cite[Section 7]{cipolloni2023central} considered the flow \eqref{eq:flowindepcompreal} with initial condition being the eigenvalues of a complex i.i.d. matrix. In the current case we will also consider the case when the initial data are given by eigenvalues of a matrix of type $M$ (see Definition~\ref{def:type-M}), where the real part has a large (almost order one) Gaussian component.

The proof in this case is analogous to \cite[Section 7]{cipolloni2023central} once the a priori estimate \eqref{eq:apriorievect} is proven. While in the complex case \eqref{eq:apriorievect} follows by directly \cite[Lemma 7.9]{cipolloni2023central} and \cite[Theorem 3.1]{cipolloni2023mesoscopic}, we need to prove this bound, in Corollary~\ref{cor:ovb} below, when the initial data is real or of type M. The second difference is that \cite[Section 7]{cipolloni2023central} considered \eqref{eq:flowindepcompreal}, and performed a coupling, only for finitely many different $z$'s; instead for the proof of Lemma~\ref{lem:ind-prod} we need to couple the flows \eqref{eq:flowindepcompreal} and \eqref{eq:flowtobecoupl} for a slightly diverging number of different $z$'s, in fact $\log n$ of them will suffice. With these changes in mind we now present the main steps of the proof.

In the argument below, we will couple the eigenvalue flows to some auxiliary processes and then derive the estimate \eqref{eqn:ind-prod} after a short time $t = n^{\eps'} / n$, for $\eps' >0$ being sufficiently small, where ``sufficiently small'' will depend on the $\eps >0$ in the hypotheses of Lemma \ref{lem:ind-prod}. However, the statement of the lemma is for $t = n^{\eps_2-1}$ for possibly larger $\eps_2 >0$. This case can be reduced to the one we derived below, simply by starting the coupling argument later in the flow of the eigenvalues $\lambda_i^z(s)$, (i.e., from $t_0 = n^{\eps_2-1} - n^{\eps'-1}$) starting from a different i.i.d. ensemble of type $M$. For notational simplicity we ignore this distinction in the proof below.

Recall that the eigenvalues of the Hermitization $H_t^z$ of $X_t-z$ are the solution of
\begin{equation}
\label{eq:flowindepcompreal}
\dif \lambda_i^z(t)=\frac{\dif b_i^z(t)}{\sqrt{2n}}+\frac{1}{2n}\sum_{j\ne i}\frac{1}{\lambda_j^z(t)-\lambda_i^z(t)}\dif t.
\end{equation}
Here $\{b_i^z(t)\}_{i\in [n]}$ is a family of standard i.i.d. real Brownian motions, and $b_{-i}^z(t)=-b_i^z(t)$; in particular, this ensures that $\lambda_{-i}^z(t)=-\lambda_i^z(t)$ for $i\in [n]$ and any $t\ge 0$. The rigidity of these eigenvalues close to zero follows by \eqref{eqn:usual-rigidity} in the complex case and by Lemma~\ref{prop:m-laws} in the real case.

We now consider the joint evolution of $\lambda_i^{z_l}(t)$ for all $z_l\in J$. In order to prove their asymptotical independence, we couple their flows with the following fully independent flows (see e.g. \cite[Section 7]{cipolloni2023central}): 
\begin{equation}
\label{eq:flowtobecoupl}
\dif \mu_i^{(l)}(t)=\frac{\dif \beta_i^{(l)}(t)}{\sqrt{2n}}+\frac{1}{2n}\sum_{j\ne i}\frac{1}{\mu_j^{(l)}(t)-\mu_i^{(l)}(t)}\dif t,
\end{equation}
with initial data $\{\mu_i^{(l)}(0)\}_{i\in[n]}$ being the singular values of $\log n$ independent complex Ginibre matrices $X^{(l)}$, and $\mu_{-i}^{(l)}(0)=-\mu_i^{(l)}(0)$. Here $\{\beta_i^{(l)}\}_{i\in [n], l\in [\log n]}$ is a family of standard real i.i.d. Brownian motions, which are then defined also for negative indices by symmetry. To make sure that the coupling argument in \cite{cipolloni2023central} can be applied we need the overlap bound (see \cite[Lemma 7.9]{cipolloni2023central}, which is needed to ensure \cite[Assumption 7.11 (c)]{cipolloni2023central}):
\begin{equation}
\label{eq:apriorievect}
\big|\langle {\bm u}_i^{z_1}(t), {\bm u}_j^{z_2}(t)\rangle\big|+\big|\langle {\bm v}_i^{z_1}(t), {\bm v}_j^{z_2}(t)\rangle\big|\le n^{-\omega_E},\qquad\quad 1\le i,j\le n^{\omega_B},
\end{equation}
for some constants $\omega_E,\omega_B>0$, uniformly in $t\le n^{100\eps}/n$. Here ${\bm w}_i^z(t)=({\bm w}_i^z(t), \pm{\bm v}_i^z(t))$ are the eigenvectors of $H_t^z$. This follows by \cite[Lemma 7.9]{cipolloni2023central} if the initial condition of $X_t$ is a complex i.i.d. matrix and it is proven in Corollary~\ref{cor:ovb} if the initial condition of $X_t$ is a type $M$ matrix.

Then, by the coupling argument in \cite[Section 7.2.1]{cipolloni2023central}, it follows that there exists a small $\omega>0$ such that
\begin{equation}
\big|\lambda_1^{z_l}(t)-\mu_1^{(l)}(t)\big|\le \frac{1}{n^{1+\omega}},
\end{equation}
with overwhelming probability in the joint probability space of all the $\{\lambda_1^{z_l}(t)\}_{z_l\in J}$ (recall that $J$ consists of $\log n$ points). We point out that the coupling in \cite[Section 7.2.1]{cipolloni2023central} was performed only for finitely many $z$'s, however, inspecting the proof, it is clear that it is actually possible to couple the flow for up to $n^c$ different $z$'s for some sufficiently small fixed $c>0$.

Then, we compute
\begin{equation}
\begin{split}
\mathbb{P}\left[ \bigcap_{l=1}^{\log n} \{ \lambda_1^z (t) \leq s n^{-1} \} \right] &\le \mathbb{P}\left[\bigcap_{l=1}^{\log n} \left\{ \mu_1^{(l)}(t)\leq s n^{-1}+|\mu_1^{(l)}(t)-\lambda_1^{z_l} (t)| \right \} \right] \\
&\le \mathbb{P}\left[ \bigcap_{l=1}^{\log n} \left\{ \mu_1^{(l)}(t)\leq s n^{-1}+n^{-1-\omega} \right \} \right]+n^{-100} \\
&=\bigcap_{l=1}^{\log n} \mathbb{P}\left[ \left\{ \mu_1^{(l)}(t)\leq s n^{-1}+n^{-1-\omega} \right \} \right]+n^{-100} \\
&\le \bigcap_{l=1}^{\log n} \mathbb{P}\left[ \left\{ \lambda_1^{z_l} (t)\leq s n^{-1}+n^{-1-\omega}+|\mu_1^{(l)}(t)-\lambda_1^{z_l} (t)| \right \} \right] +n^{-100} \\
&\le \bigcap_{l=1}^{\log n} \mathbb{P}\left[ \left\{ \lambda_1^{z_l} (t)\leq s n^{-1}+2n^{-1-\omega} \right \} \right] +n^{-100}.
\end{split}
\end{equation}
Noticing that $s\ge (\log n)^{-C}$, and so that $s n^{-1}+2n^{-1-\omega}\le 2sn^{-1}$, this concludes the proof.

\qed

\section{Removal of a small complex Gaussian divisible ensemble from a real GDE} \label{a:small-complex-comp}

The purpose of this section is to prove Proposition \ref{prop:GFTmax-2}. Recall that our set-up is that $X_t$ satisfies $\d X_t = \frac{\d B_t}{\sqrt{n}} - \frac{X_t}{2} \d t$ where $B_t$ is an i.i.d. matrix of complex Brownian motions and $X_0$ is a real i.i.d. matrix. Moreover, we are considering the observable, 
\beq
\Xi (X) := \max_{z \in \hat{P}} Q ( n \eta_1 \Im \langle G^z ( \i \eta_1 ) \rangle ) \left( \Psi_n (z, X, \eta_3 ) - \Psi_n (z, X, n^{-\mfb} ) \right).
\eeq
For $z_i \in \hat{P}$ let us denote 
\begin{align}
    \XX_i &=  Q ( n \eta_1 \Im \langle G^{z_i} ( \i \eta_1 ) \rangle ) Y_i \\
    Y_i &=\Psi_n (z, X, \eta_3 ) - \Psi_n (z, X, n^{-\mfb} )  =  \int_{\eta_3}^{n^{-\mfb}} n  \Im [ \langle G^{z_i} ( \i \eta ) - M^{z_i} ( \i \eta) \rangle ] \d \eta - c_{i}  
\end{align} 
for some $n$--dependent constant $c_i$ that is $\O ( \log n )$. Fixing a small $\mfa >0$ we see that, 
\beq \label{eqn:cc-1}
\left| \Xi(X) - Z_\mfa (X)\right| := \left| \Xi (X) - \frac{1}{n^{\mfa}} \log \left( \sum_{i \in \hat{P}} \e^{ n^{\mfa} \XX_i } \right) \right| \leq \frac{ \log n}{ n^{\mfa}} .
\eeq
(with the $Z_\mfa(X)$ defined implicitly in the obvious way). It suffices to prove the estimate for $Z_\mfa (X)$, after taking $\mfa >0$ sufficiently small. Define now,
\beq
F_1 (X) = F ( Z_\mfa (X) ) .
\eeq
In \eqref{eq:GFTcomplreal} below, the derivatives $\del_{ab}$ are with respect to the $(a, b)$ entries of the Hermitization of $X$, as $F_1$ depends on $X$ only through its Hermitization. Moreover, we recall the definition of $\sumab$ in \eqref{eqn:sumab-def}. 

\bel
Let $F_1$ be as above. Then, 
\beq
\label{eq:GFTcomplreal}
\frac{\d}{\d t} \ee[ F_1 (X_t) ] = - \frac{\e^{-t}}{2n}  \ee\left[ \sumab \del_{ab}^2 F_1 (X_t) \right] + \O ( n^{1/2+\eps+3 \mfa} ).
\eeq
for any $\eps >0$. 
\eel

\proof Let,
\beq
\frac{a_t}{n} := \ee[ \Re[X_{ab}(t)]^2] = \frac{1}{n} - \ee[ ( \Im[X_{ab}] (t)^2] = \frac{ 1 + \e^{-t}}{2n}
\eeq
Then, by  It\^{o}'s lemma, we have
\beq
\frac{\d}{ \d t} \ee[ F_1(X_t)] =  \sum_{a, b=1}^N \ee\left[ \frac{1}{4n} ( \del_{\Re[X_{ab}]}^2 + \del_{ \Im[X_{ab}]}^2 )F_1 (X_t) - \frac{\Re[X_{ab}]}{2} \del_{\Re[X]_{ab}} F_1  - \frac{ \Im[ X_{ab}]}{2} \del_{\Im[X]_{ab}} F \right].
\eeq
Due to a cumulant expansion, similar to e.g., \eqref{eq:writesum}, we have
\beq
\ee\left[\frac{\Re[X_{ab}]}{2} \del_{\Re[X]_{ab}} F_1  + \frac{ \Im[ X_{ab}]}{2} \del_{\Im[X]_{ab}} F_1 \right] = \ee\left[ \frac{a_t}{2n} \del_{\Re[X_{ab}]}^2 F_1 + \frac{1-a_t}{2n} \del_{\Im[X_{ab}]}^2 F_1 \right] + \O ( n^{\eps-3/2} )
\eeq
Here, the error in the cumulant expansion can be controlled by Lemma~\ref{lem:general-der-bd}, the derivatives of $F_1$ with respect to matrix elements are bounded above by some small powers of the derivatives of the $\XX_i$. However, by direct calculation,
\beq
\left| \del_{ij}^k \XX_i \right| \leq C_k \sup_{ \eta_1 \wedge \eta_3 \leq \eta \leq 1} |G_{ab} ( \i \eta) |^{2k+4} \leq n^{\eps}
\eeq
for any $\eps >0$ with overwhelming probability. Moreover, the derivatives are deterministically bounded by $n^{3k}$. The claim now follows once we note that we have shown,
\begin{align}
\frac{ \d}{ \d t} \ee[ F_1 (X_t) ] &= \frac{1-2a_t}{4n} \sum_{a, b=1}^n \ee\left[ (\del_{\Re[X_{ab}]}^2 - \del_{\Im[X_{ab}]}^2 ) F_1 (X_t) \right] + \O ( n^{\eps-3/2} ) \notag \\
&= \frac{1-2a_t}{2n} \sum_{a,b=1}^n \ee\left[ (\del_{X_{ab}}^2 + \del_{ \bar{X}_{ab}}^2) F_1 (X_t) \right] + \O ( n^{\eps-3/2} )  \notag \\
&= \frac{1-2a_t}{2n} \sumab \ee\left[ \del_{ab}^2 F_1 (X_t) \right] + \O ( n^{\eps-3/2} )
\end{align}
where $2 \del_{X_{ab}} = \del_{\Re[X_{ab}]} - \i \del_{ \Im[X_{ab}] }$ is the usual Wirtinger derivative. \qed

\

\begin{proof}[Proof of Proposition~\ref{prop:GFTmax-2}] \, By direct calculation, the derivative $\sumab \del_{ab}^2 F_1 (X)$ can be bounded above by a constant times $n^{2 \mfa}$ times the maximum of the absolute value of the following five terms:
\begin{align}
&\mathcal{E}_1:=n \eta_1 \sum_{a, b} \del_{ab}^2 \langle \Im G^{z_i} ( \i \eta_1 ) \rangle , \qquad \mathcal{E}_2:=(n \eta_1)^2 \sum_{a, b} ( \del_{ab} \langle \Im [ G^{z_i} (\i \eta_1 ) ] \rangle )  ( \del_{ab} \langle \Im [ G^{z_j} (\i \eta_1 ) ] \rangle ) \label{eqn:dd-1} \\
&\mathcal{E}_3:=\sum_{ab} \del_{ab}^2 \int_{\eta_2}^{n^{-c_2}} n \Im \langle G^{z_i} (\i u ) - M^{z_i} (\i u ) \rangle \d u, \\
&\mathcal{E}_4:=\sum_{ab} \left( \del_{ab} \int_{\eta_2}^{n^{-c_2}} n \Im \langle G^{z_i} (\i u ) - M^{z_i} (\i u ) \rangle \d u \right) \left( \del_{ab} \int_{\eta_2}^{n^{-c_2}} n \Im \langle G^{z_j} (\i u ) - M^{z_j} (\i u ) \rangle \d u \right) \\
& \mathcal{E}_5:= n\eta_1\sum_{ab}   ( \del_{ab} \langle \Im [ G^{z_i} (\i \eta_1 ) ] \rangle ) \left( \del_{ab} \int_{\eta_2}^{n^{-c_2}} n \Im \langle G^{z_j} (\i u ) - M^{z_j} (\i u ) \rangle \d u \right). \label{eqn:dd-4}
\end{align} 

Therefore, Proposition~\ref{prop:GFTmax-2} follows immediately from the estimates \eqref{eqn:cc-1} and \eqref{eq:GFTcomplreal} and the following, since we are assuming that $\Im[z_i] \geq n^{c_1-1/2}$ for all $z_i \in \hat{P}$ (we choose, e.g., $\mfa = \frac{c_1}{100}$). The proof of this lemma is presented below, after Corollary~\ref{cor:neededcor}.
\bel
\label{lem:estGFT}
Fix any small $\eps>0$. If $X$ is a matrix of type $M$, having real Gaussian component of size $n^{-\eps/100}$, then each of the terms in \eqref{eqn:dd-1}--\eqref{eqn:dd-4} is bounded in absolute value,with overwhelming probability, by
\beq
\label{eq:desiredbgftrc}
n^2\log n\left(\frac{n^{7\eps/3}}{n\min_i |\Im z_i|^2}+n^{-\eps/3}\right).
\eeq
\eel

\end{proof}

\qed

\subsection{Proof of the local law to estimate \eqref{eq:GFTcomplreal}}

The main result of this section is the following local law for real matrices with a large Gaussian component. The proof of this proposition is presented at the end of this section.

\begin{proposition}
\label{pro:realllawpro}
Fix any small $\xi,\omega>0$, and fix $|z_1-z_2|\le n^{-\omega}$. Let $X$ be a real i.i.d matrix having a Gaussian component of size $n^{-\xi}$. Then, with overwhelming probability, we have
\begin{equation}
\label{eq:desiredllawrealcase}
\big|\langle \big(G^{z_1}(\ii\eta_1)A_1 G^{z_2}(\ii\eta_2)-M^{z_1,z_2}(\ii\eta_1,A_1,\ii\eta_2)\big)A_2\rangle\big|\lesssim n^{10\xi}\left(\frac{1}{n\eta_*^{3/2}(|z_1-z_2|^2+|\eta_1|+|\eta_2|)^{1/2}}+\frac{1}{n^{3/2}\eta_*^{5/2}}\right),
\end{equation}
with $\eta_*:=|\eta_1|\wedge |\eta_2|$, uniformly in $\eta_*\ge n^{-1+100\xi}$, and matrices $\lVert A_i\rVert\lesssim 1$.

Furthermore, \eqref{eq:desiredllawrealcase} holds if $X$ is replaced with a matrix of type $M$, as defined in Definition~\ref{def:type-M}, such that its real component has a Gaussian component of size $n^{-\xi}$. In this case, we also have
\begin{equation}
\label{eq:newb}
\big|\langle \big(G^{z_1}(\ii\eta_1)A_1 (G^{z_2}(\ii\eta_2))^\mathfrak{t}-M_t^{z_1,\overline{z_2}}(\ii\eta_1,A_1,\ii\eta_2)\big)A_2\rangle\big|\lesssim n^{10\xi}\left(\frac{1}{n\eta_*^{3/2}(|z_1-\overline{z_2}|^2+|\eta_1|+|\eta_2|)^{1/2}}+\frac{1}{n^{3/2}\eta_*^{5/2}}\right),
\end{equation}
with $M_t^{z_1,\overline{z_2}}$ defined by $\partial_t\langle M_t^{z_1,\overline{z_2}}\rangle=\langle M_t^{z_1,\overline{z_2}}\rangle$ and $M_0^{z_1,\overline{z_2}}$
is from \eqref{eq:realdefm12} with $c_*(0)=1$.
Here $t$ denotes the size of the Gaussian component in Definition~\ref{def:type-M}.
\end{proposition}

Additionally, we the following slightly weaker local law for general real i.i.d. matrices.
\begin{corollary} \label{cor:nozreal}
Let $X$ be a real i.i.d. matrix. Then, with overwhelming probability for any $\xi>0$, we have
\begin{equation}
\label{eq:hoplastb}
\big|\langle \big(G^{z_1}(\ii\eta_1)A_1 G^{z_2}(\ii\eta_2)-M^{z_1,z_2}(\ii\eta_1,A_1,\ii\eta_2)\big)A_2\rangle\big|\lesssim \frac{n^\xi}{n\eta_*^2},
\end{equation}
with $\eta_*:=|\eta_1|\wedge |\eta_2|$, uniformly in $\eta_*\ge n^{-1+10\xi}$, and matrices $\lVert A_i\rVert\lesssim 1$.
\end{corollary}
\begin{proof}
For real matrices with a Gaussian component of size $n^{-\xi/10}$ \eqref{eq:hoplastb} follows by \eqref{eq:desiredllawrealcase}. Then, by a standard comparison argument (e.g. similar to the proof of Proposition~\ref{lem:precllaw} in Section~\ref{sec:gde-removal}) we can remove this Gaussian component at a price of a negligible error $n^{-\xi/10}/(n\eta_*^2)$. This concludes the proof.
\end{proof}
\qed

As an immediate corollary of Proposition~\ref{pro:realllawpro} we have the following bound for eigenvector overlaps.
\begin{corollary}

\label{cor:ovb}
Fix any small $\omega_d\ge 100\xi>0$, and let $X$ be a matrix of type $M$ as defined in Definition~\ref{def:type-M}, such that its real component has a Gaussian component of size $n^{-\xi}$. Pick $z_1,z_2$ such that $|z_1-z_2|\ge n^{-1/2+\omega_d}$, and let $H^{z_l}$, $l=1,2$, be the Hermitization of $X-z_l$. Let ${\bm w}_i^{z_l}=({\bm u}_i^{z_l}, \pm {\bm v}_i^{z_l})$ denote the eigenvectors of $H^{z_l}$. Then there exist $\omega_E\ge \omega_d/2$, $\omega_B>0$ such that, with overwhelming probability, we have
\begin{equation}
\label{eq:boundevectorsoverl}
\big|\langle {\bm u}_i^{z_1}, {\bm u}_j^{z_2}\rangle\big|+\big|\langle {\bm v}_i^{z_1}, {\bm v}_j^{z_2}\rangle\big|\le n^{5\xi-\omega_E},\qquad\quad 1\le i,j\le n^{\omega_B}.
\end{equation}
\end{corollary}
\begin{proof}
Let $\omega>0$ from Proposition~\ref{pro:realllawpro}. In the regime $|z_1-z_2|\le n^{-\omega}$, given \eqref{eq:desiredllawrealcase}, the proof of this corollary is completely analogous to the proof of \cite[Lemma 7.9]{cipolloni2023central}. In the complementary regime $|z_1-z_2|> n^{-\omega}$ the bound \eqref{eq:boundevectorsoverl} was already proven in \cite[Lemma 7.9]{cipolloni2023central}.
\end{proof}
\qed

\

To prove Lemma~\ref{lem:estGFT}, we need a bound on $\langle G^{z_1}(\ii\eta_1)G^{z_2}(\ii\eta_2)\rangle$ also for $\eta_i$ below $1/n$:
\begin{corollary}
\label{cor:neededcor}
Fix any small $\epsilon>0$. Let $X$ be a matrix with real Gaussian component of size $n^{-\eps/100}$ such that its Hermitization satisfies \eqref{eq:newb}, then for any large $C>0$, for $l_1, l_2 \geq 1$, we have
\begin{equation}
\label{eq:bbelowscale}
\big|\langle G^{z_1}(\ii\eta_1)^{l_1} A_1 (G^{z_2}(\ii\eta_2)^{l_2})^\mathfrak{t} A_2\rangle\big|\lesssim \frac{1}{|\eta_1|^{l_1-1}|\eta_2|^{l_2-1}}\left(\frac{n^{7\eps/3}}{|z_1-\overline{z_2}|^2}+n^{1-\epsilon/3}\right),
\end{equation}
with overwhelming probability uniformly in $|\eta_i|\ge n^{-1}(\log n)^{-C}$, and matrices $\lVert A_i\rVert\lesssim 1$. Similarly, if $X$ satisfies \eqref{eq:desiredllawrealcase} then \eqref{eq:bbelowscale} holds with $(G^{z_2}(\ii\eta_2)^{l_2})^\mathfrak{t}$ replaced with $G^{z_2}(\ii\eta_2)^{l_2}$, and $\overline{z_2}$ in the RHS replaced with $z_2$.
\end{corollary}

\

\proof[Proof of Lemma~\ref{lem:estGFT}] \, We now show that all the terms in \eqref{eqn:dd-1}--\eqref{eqn:dd-4} are of a form so that the bound \eqref{eq:bbelowscale} can be applied. To keep the presentation short we neglect the fact that all the terms in \eqref{eqn:dd-1}--\eqref{eqn:dd-4} contain the imaginary part of the resolvent, as we can write
\[
2\i\Im G(\i\eta)=G(\i\eta)-G^*(\i\eta)=G(\i\eta)-G(-\i\eta),
\]
and the estimate \eqref{eq:bbelowscale} is not sensitive to $\eta_i$ being positive or negative. We thus write
\begin{equation}
\begin{split}
\mathcal{E}_1&=2n\eta_1\tilde{\sum}_{ij}\langle (G^{z_1}(\i\eta_1))^\mathfrak{t}E_i(G^{z_2}(\i\eta_1))^2E_j\rangle \\
\mathcal{E}_2&=\frac{n\eta_1^2}{2}\tilde{\sum}_{ij}\langle [(G^{z_1}(\i\eta_1))^2]^\mathfrak{t}E_i(G^{z_2}(\i\eta_1))^2E_j\rangle \\
\mathcal{E}_3&=2n\tilde{\sum}_{ij}\int_{\eta_2}^{n^{-c_2}}\langle (G^{z_1}(\i u))^\mathfrak{t}E_i(G^{z_1}(\i u))^2E_j\rangle\,\dif u  \\
\mathcal{E}_4&=\frac{n}{2}\tilde{\sum}_{ij}\int\int_{\eta_2}^{n^{-c_2}}\langle [(G^{z_1}(\i u))^2]^\mathfrak{t}E_i(G^{z_2}(\i v))^2E_j\rangle\,\dif u\dif v \\
\mathcal{E}_5&=\frac{n\eta_1}{2}\tilde{\sum}_{ij}\int_{\eta_2}^{n^{-c_2}}\langle [(G^{z_1}(\i\eta_1))^2]^\mathfrak{t}E_i(G^{z_2}(\i u))^2E_j\rangle\, \dif u.
\end{split}
\end{equation}
Using \eqref{eq:bbelowscale}, we immediately get \eqref{eq:desiredbgftrc}.

\qed

\

\begin{proof}[Proof of Corollary~\ref{cor:neededcor}] \, To keep the presentation simple we only present the proof that the estimate \eqref{eq:desiredllawrealcase} implies \eqref{eq:bbelowscale} but with $(G^{z_2} ( \i \eta_2 )^{l_2} )^\mft$  on the LHS replaced with $G^{z_2} ( \i \eta_2 )^{l_2}$, and with with $\overline{z_2}$ replaced by $z_2$ on the RHS. The proof of the fact that \eqref{eq:newb} implies \eqref{eq:bbelowscale} is completely analogous and so omitted. We also assume that $\eta_1, \eta_2 >0$, the other cases being identical.

First, we show that if \eqref{eq:desiredllawrealcase} holds then the same local law holds if one of the $G$'s is replaced by $|G|$'s, after possibly multiplying the RHS of \eqref{eq:desiredllawrealcase} by $\log n$. For this purpose we use the integral representation \cite[Eq. (5.4)]{cipolloni2022optimal}
\begin{equation}
\label{eq:absintrep}
\big|G^z(\ii\eta)\big|=\frac{2}{\pi}\int_0^\infty \Im G^z(\ii\sqrt{\eta^2+v^2})\,\frac{\dif v}{\sqrt{\eta^2+v^2}}.
\end{equation}
Defining $\eta_v:=\sqrt{\eta_1^2+v^2}$, by \cite[Eq. (5.6)]{cipolloni2022optimal}, we write the deterministic approximation $\widetilde{M}(\i\eta_1,A,\eta_2)$ of $ |G^{z_1}(\ii\eta_1)|A G^{z_2}(\ii\eta_2)$ as
\[
\widetilde{M}(\i\eta_1,A,\eta_2):=\frac{1}{\pi \ii}\int_0^\infty \frac{\widehat{M}^{z_1,z_2}(\ii \eta_v,A,\ii\eta_2)}{\eta_v}\,\dif v,
\]
with
\[
2\ii \widehat{M}^{z_1,z_2}(\ii\eta_v,A,\ii\eta_2):=M^{z_1,z_2}(\ii\eta_v,A,\ii\eta_2)-M^{z_1,z_2}(-\ii\eta_v,A,\ii\eta_2).
\]
Note that by \eqref{eq:deb12mhop}, we have (neglecting $\log n$--factors)
\begin{equation}
\label{eq:btildem}
\left\lVert\widetilde{M}(\i\eta_1,A,\eta_2)\right\rVert\lesssim \frac{\lVert A\rVert}{|z_1-z_2|^2+\eta_1+\eta_2}
\end{equation}

We thus have
\begin{equation}
\begin{split}
\label{eq:startcomp}
&\langle \big(|G^{z_1}(\ii\eta_1)|A_1 G^{z_2}(\ii\eta_2)-\widetilde{M}^{z_1,z_2}(\ii\eta_1,A_1,\ii\eta_2)\big)A_2\rangle \\
&\qquad\qquad\quad= \int_0^\infty \langle \big(\Im G^{z_1}(\ii\eta_v)A_1 G^{z_2}(\ii\eta_2)-\widehat{M}^{z_1,z_2}(\ii\eta_v,A_1,\ii\eta_2)\big)A_2\rangle \,\frac{\dif v}{\eta_v} \\
&\qquad\qquad\quad= \int_0^{n^{100}} \langle \big(\Im G^{z_1}(\ii\eta_v)A_1 G^{z_2}(\ii\eta_2)-\widehat{M}^{z_1,z_2}(\ii\eta_v,A_1,\ii\eta_2)\big)A_2\rangle \,\frac{\dif v}{\eta_v}+\mathcal{O}(n^{-10}). \\
& \qquad\qquad\quad = \O \left[ ( \log n) n^{\eps/10} \left( \frac{1}{ n \eta_*^{3/2} ( |z_1 - z_2|^2 + \eta_1 + \eta_2 )} + \frac{1}{n^{3/2} \eta_*^{5/2}} \right) \right]
\end{split}
\end{equation}
We point out that to remove the regime $\eta_v\ge n^{100}$ in \eqref{eq:startcomp} we used the norm bound $\lVert G\rVert\le 1/\eta$ for the resolvents, and \eqref{eq:deb12mhop} for the deterministic term. In the last inequality we used \eqref{eq:desiredllawrealcase} for the integrand in the third line of \eqref{eq:startcomp} and that $\int \dif v/\eta_v\lesssim \log n$.

We now show that given  \eqref{eq:startcomp} for $\eta_1,\eta_2\gg 1/n$, we can extend it below $1/n$. We will achieve this in two steps: we first prove that a bound of the form \eqref{eq:bbelowscale} holds when one $\eta_i\gg 1/n$ and the other one is smaller than $1/n$, and then, using this new bound as an input, that \eqref{eq:bbelowscale} holds when both $\eta_1$ and $\eta_2$ are smaller than $1/n$. We first prove this when $A_1=A$, $A_2=A^*$, and then we show that this easily implies the general case.

We now may assume that \eqref{eq:desiredllawrealcase} and \eqref{eq:startcomp} hold for $\eta_*\ge n^{-1+\eps}$. Assume that $(\log n)^{-C}n^{-1}\le \eta_1\le n^{-1+\eps}=:\hat{\eta}$ and that $\eta_2 \geq \hat{\eta}$. We have the general estimate,
\begin{equation}
\label{eq:specdec12}
\begin{split}
\big|\langle G^{z_1}(\ii\tau)^{l_1+1} A G^{z_2}(\ii\eta_2)^{l_2} A^* \rangle\big|&=\left|\frac{1}{2n}\sum_{i,j}\frac{|\langle {\bm w}_i^{z_1},A {\bm w}_j^{z_2}\rangle|^2}{(\lambda_i^{z_1}-\ii\tau)^{l_1+1}(\lambda_j^{z_2}-\ii\eta_2)^{l_2}}\right|  \\
&\lesssim\frac{1}{2n\tau^{l_1-1}\eta_2^{l_2-1}}\sum_{i,j}\frac{|\langle {\bm w}_i^{z_1},A {\bm w}_j^{z_2}\rangle|^2}{|\lambda_i^{z_1}-\ii\eta_1|^2|\lambda_j^{z_2}-\ii\eta_2|} \\
&= \frac{1}{\tau^{l_1}\eta_2^{l_2-1}}\langle \Im G^{z_1}(\ii\tau)A |G^{z_2}(\ii\eta_2)| A^* \rangle.
\end{split}
\end{equation}
Applying this with $\eta_1\le\tau\le \hat{\eta}$ to the integral on the RHS of the first line below we have,
\begin{align}
\label{eq:compalmthr}
\big|\langle G^{z_1}(\ii\eta_1)^{l_1} A G^{z_2}(\ii\eta_2)^{l_2} A^*\rangle- &\langle G^{z_1}(\ii\hat{\eta})^{l_1} A G^{z_2}(\ii\eta_2)^{l_2} A^* \rangle\big|=\left|\int_{\eta_1}^{\hat{\eta}} \langle G^{z_1}(\ii\tau)^{l_1+1} A G^{z_2}(\ii\eta_2)^{l_2} A^*\rangle\,\dif \tau\right| \notag \\
&\lesssim \frac{1}{\eta_2^{l_2-1}}\int_{\eta_1}^{\hat{\eta}} \frac{1}{\tau^{l_1}}\langle \Im G^{z_1}(\ii\tau)  A |G^{z_2}(\ii\eta_2)| A^* \rangle\,\dif \tau \notag \\
&\lesssim n^\eps(\log n)^{C+1}\frac{1}{\eta_1^{l_1-1}\eta_2^{l_2-1}}\langle \Im G^{z_1}(\ii \hat{\eta}) A |G^{z_2}(\ii\eta_2)| A^* \rangle \notag \\
&\lesssim \frac{n^\eps (\log n)^{C+1} n^{\eps/10}}{\eta_1^{l_1-1}\eta_2^{l_2-1}} \left(\frac{1}{|z_1-z_2|^2}+\frac{n^{1/2-3\eps/2}}{|z_1-z_2|}+n^{1-5\eps/2}\right) \notag \\
&\lesssim \frac{n^\eps (\log n)^{C+1} n^{\eps/10}}{\eta_1^{l_1-1}\eta_2^{l_2-1}} \left(\frac{1}{|z_1-z_2|^2}+n^{1-5\eps/2}\right),
\end{align}
where in the last line we used a Schwarz inequality. Here, in the first inequality we used \eqref{eq:specdec12}, in the second inequality we used that $\tau\mapsto \tau \Im G(\ii\tau)$ is increasing as an operator, and in the third inequality we used \eqref{eq:btildem} and  \eqref{eq:startcomp}. Now, note that the second term in the LHS of \eqref{eq:compalmthr} can be incorporated into the RHS by \eqref{eq:startcomp}, \eqref{eq:specdec12}, and \eqref{eq:btildem}. We therefore conclude that
\beq \label{eqn:below-scale-1}
\left| \langle G^{z_1} (\i \eta_1)^{l_1} A G^{z_2} ( \i \eta_2 )^{l_2} A^* \rangle \right| \lesssim \frac{n^\eps (\log n)^{C+1} n^{\eps/10}}{\eta_1^{l_1-1}\eta_2^{l_2-1}} \left(\frac{1}{|z_1-z_2|^2}+n^{1-5\eps/2}\right)
\eeq
holds when only $\eta_1$ is below the scale $\hat{\eta}$ but $\eta_2 \geq \hat{\eta}$. 

We now consider the case when both $\eta_1,\eta_2$ are below the scale $\hat{\eta}$. Proceeding as in \eqref{eq:compalmthr}, we find that
\begin{equation}
\label{eq:2belowscl}
\big|\langle G^{z_1}(\ii\eta_1)^{l_1}A G^{z_2}(\ii\eta_2)^{l_2}A^*\rangle- \langle G^{z_1}(\ii\hat{\eta})^{l_1}A G^{z_2}(\ii\eta_2)^{l_2} A^* \rangle\big| \lesssim n^\eps(\log n)^{C+1}\frac{1}{\eta_1^{l_1-1}\eta_2^{l_2-1}}\langle \Im G^{z_1}(\ii \hat{\eta}) A |G^{z_2}(\ii\eta_2)| A^* \rangle,
\end{equation}
where we used again the monotonicity of $\tau\mapsto \tau \Im G(\ii\tau)$ as an operator. Note that in the RHS of \eqref{eq:2belowscl} only $\eta_2$ is below the scale. Therefore, by applying \eqref{eq:absintrep} we can use the estimate \eqref{eqn:below-scale-1} to estimate the RHS of \eqref{eq:2belowscl}, finding,
\begin{equation} \label{eqn:below-scale-2}
\begin{split}
\big|\langle G^{z_1}(\ii\eta_1)^{l_1} A G^{z_2}(\ii\eta_2)^{l_2} A^* \rangle\big| & \lesssim \frac{ [n^\eps (\log n)^{C+1}]^2 n^{\eps/10} }{ \eta_1^{l_1-1} \eta_2^{l_2-1} }\left(\frac{1}{|z_1-z_2|^2}+n^{1-5\eps/2}\right) \\
& \le  \frac{1}{ \eta_1^{l_1-1} \eta_2^{l_2-1} }\left(\frac{n^{2\eps+\eps/9}}{|z_1-z_2|^2}+n^{1-\eps/2+\eps/9}\right).
\end{split}
\end{equation}
where we also used \eqref{eqn:below-scale-1} to estimate the second term on the LHS of \eqref{eq:2belowscl}. We re--iterate that the conclusion of all of the above argument is that \eqref{eqn:below-scale-2} holds for $| \eta_i| \geq n^{-1} ( \log n )^{-C}$. 

We now turn to the final part of the proof, concluding that \eqref{eq:bbelowscale} holds for general $A_1, A_2$. First, note that by applying \eqref{eq:absintrep} twice, we see that the estimate \eqref{eqn:below-scale-2} holds also in the case that $l_1=l_2 =1$ and $G^{z_i } ( \i \eta_i)$ are both replaced by $| G^{z_i} ( \i \eta_i ) |$ on the LHS. Applying this in the second inequality below, we find,
\[
\begin{split}
\big|\langle G^{z_1}(\ii\eta_1)^{l_1} A_1 G^{z_2}(\ii\eta_2)^{l_2} A_2\rangle\big|&\le \frac{\prod_{i=1}^2\langle |G^{z_1}(\ii\eta_1)| A_i |G^{z_2}(\ii\eta_2)| A_i^*\rangle^{1/2}}{\eta_1^{l_1-1}\eta_2^{l_2-1}} \\
&\lesssim \frac{(\log n)^2}{\eta_1^{l_1-1}\eta_2^{l_2-1}}\left(\frac{n^{2\eps+\eps/9}}{|z_1-z_2|^2}+n^{1-\eps/2+\eps/9}\right),
\end{split} .
\]
The first inequality followed from an argument similar to \eqref{eq:specdec12}. This concludes the proof.

\qed

\end{proof}

\

\begin{proof}[Proof of Proposition~\ref{pro:realllawpro}] \, We first prove \eqref{eq:desiredllawrealcase} for real i.i.d. matrices and then at the end of the proof we describe the very minor differences to obtain the same result for matrices of type $M$. The proof of this proposition follows very closely \cite[Section 5]{cipolloni2023mesoscopic}; in fact the only difference is that in \cite[Section 5]{cipolloni2023mesoscopic} it was considered the evolution of a certain initial matrix along complex Brownian dynamics, instead in the current case we will consider real Brownian dynamics. First of all we notice that if $1\gtrsim |\eta_i|\ge n^{-\xi}$, then \eqref{eq:desiredllawrealcase} holds by \cite[Theorem 5.2]{cipolloni2023central}. If instead $|\eta_i|\gtrsim 1$ this follows from computations analogous to \cite[Appendix B]{cipolloni2022optimal}. For this reason, to prove \eqref{eq:desiredllawrealcase}, in the reminder of the proof we use a dynamical argument to show that this bound can in fact be propagated down to $\eta_*\ge n^{-1+100\xi}$ (see \cite[Proposition 5.3]{cipolloni2023mesoscopic} for the complex case).

Consider the Ornstein--Uhlenbeck flow
\[
\dif X_t=-\frac{1}{2}X_t \dif t+\frac{\dif B_t}{\sqrt{n}}, \qquad X_0=X,
\]
with characteristics
\begin{equation}
\label{eq:charapp}
\partial_t \eta_t=-\Im m^{z_t}(\ii\eta_t)-\frac{\eta_t}{2}, \qquad\quad \partial_t z_t=-\frac{z_t}{2}.
\end{equation}
Here $(B_t)_{ij}$ are i.i.d. real Brownian motions and $X$ is an i.i.d. matrix (see Definition~\ref{def:model}). Define the resolvents $G_{i,t}:=(H^{z_{i,t}}-\ii\eta_{i,t})^{-1}$, and let $\mathfrak{B}_t$ be the Hermitization of $B_t$ defined in \eqref{eq:defbigBM}. Then, by It\^{o}'s formula, we have
\begin{equation}
\begin{split}
\label{eq:evg1g2real}
    \dif \langle G_{1,t}AG_{2,t}B\rangle &=\sum_{a,b=1}^{2n}\partial_{ab}\langle G_{1,t}AG_{2,t}B\rangle \frac{\dif (\mathfrak{B}_t)_{ab}}{\sqrt{n}}+\langle G_{1,t}AG_{2,t}B\rangle\dif t \\
&\quad+2\langle G_{1,t}AG_{2,t}E_1\rangle\langle G_{2,t}BG_{1,t}E_2\rangle\dif t+2\langle G_{1,t}AG_{2,t}E_2\rangle\langle G_{2,t}BG_{1,t}E_1\rangle\dif t \\
&\quad+\langle G_{1,t}-M_{1,t}\rangle\langle G_{1,t}AG_{2,t}BG_{1,t}\rangle\dif t +\langle G_{2,t}-M_{2,t}\rangle\langle G_{2,t}BG_{1,t}AG_{2,t}\rangle \dif t  \\
&\quad+\frac{\bm1_{\{\beta=1\}}}{n}\tilde{\sum}_{ij}\langle G_{1,t}^\mathfrak{t}E_i G_{1,t}AG_{2,t} BG_{1,t}E_j\rangle \dif t +\frac{2\bm1_{\{\beta=1\}}}{n}\tilde{\sum}_{ij}\langle [G_{1,t}AG_{2,t}]^\mathfrak{t}E_i G_{2,t}BG_{1,t}E_j\rangle\dif t \\
&\quad+\frac{\bm1_{\{\beta=1\}}}{n}\tilde{\sum}_{ij}\langle G_{2,t}^\mathfrak{t}E_i G_{2,t}BG_{1,t} AG_{2,t}E_j\rangle \dif t.
\end{split}
\end{equation}
Here $\tilde{\sum}_{ij}$ is defined below \eqref{eq:defE1E2}, and we recall that it denotes the sum over $(i,j)\in \{(1,2),(2,1)\}$.
We also point out that $M_{i,t}=M^{z_{i,t}}(\ii\eta_{i,t})$ depends on time only through the characteristics \eqref{eq:charapp}. In the following we use the short--hand notation $\sum_{ab}:=\sum_{a,b=1}^{2n}$.

Note that the only difference compared to \cite[Eq. (5.7)]{cipolloni2023mesoscopic} are the three new terms in the last two lines of \eqref{eq:evg1g2real}. For this reason we only explain how to estimate these terms and show that they do not change the proof of \cite[Proposition 5.3]{cipolloni2023mesoscopic} as they can be incorporated in the error terms that are already present in the proof of \cite[Proposition 5.3]{cipolloni2023mesoscopic}. In fact, even if the quadratic variation of the stochastic term in the RHS of \eqref{eq:evg1g2real} is slightly different compared to \cite[Eq. (5.14)]{cipolloni2023mesoscopic}, this does not imply any change in its estimate. The quadratic variation of the stochastic term is now given by
\[
\frac{1}{n}\sum_{ab}\big|\partial_{ab}\langle G_{1,t}AG_{2,t}B\rangle\big|^2+\frac{\bm1_{\{\beta=1\}}}{n}\sum_{ab}\big(\partial_{ab}\langle G_{1,t}AG_{2,t}B\rangle\big)^2.
\]
However, it is easy to see that the second term above is estimated in terms of the first one, which is equal to \cite[Eq. (5.14)]{cipolloni2023mesoscopic}; hence, no new estimate is needed for the stochastic term in \eqref{eq:evg1g2real}. 

The proof of \cite[Proposition 5.3]{cipolloni2023mesoscopic} is divided into two parts: in Part 1 it is proven a weaker local law with error $1/(n\eta_*\sqrt{\eta_1\eta_2})$, then in Part 2 this bound is improved to \eqref{eq:desiredllawrealcase}. For the purpose of Part 1 we estimate the three new terms in \eqref{eq:evg1g2real} by (we write only two for brevity)
\begin{equation}
\label{eq:newusbreal}
\frac{1}{n}\big|\langle G_{1,t}^\mathfrak{t}E_i G_{1,t}AG_{2,t} BG_{1,t}E_j\rangle\big|+\frac{1}{n}\big|\langle [G_{1,t}AG_{2,t}]^\mathfrak{t}E_i G_{2,t}BG_{1,t}E_j\rangle\big|\lesssim \frac{1}{n\eta_*\sqrt{\eta_1\eta_2}},
\end{equation}
with overwhelming probability. This bound follows by a simple Schwarz to separate the resolvents with their transposes followed by Ward identity. In Part 1 of \cite[Proposition 5.3]{cipolloni2023mesoscopic} it was considered only the special case $A=L_-', B=L_-$, with $L_-,L_-'$ being the left eigenvectors corresponding to the smallest eigenvalue of the operator $1-M_{1,t}\mathcal{S}[\cdot]M_{2,t}$ and the one with $1\to 2$ exchaged, respectively. This is a consequence of the fact that for any matrix orthogonal to this eigenvectors the desired result immediately follows from \cite[Lemma 5.4]{cipolloni2023mesoscopic}.  In particular, defining
\[
Y_t:=\big|\langle \big(G_{1,t}L_-G_{2,t}-M^{z_{1,t}z_{2,t}}(\ii\eta_{1,t},L_-',\ii\eta_{2,t} )\big)\rangle L_-\rangle\big|,
\]
and combining \eqref{eq:newusbreal} with \cite[Eq. (5.29)]{cipolloni2023mesoscopic} we obtain
\begin{equation}
Y_t =Y_0+2\int_0^t\langle M^{z_{1,s}z_{2,s}}(\ii\eta_{1,s},I,\ii\eta_{2,s})\rangle Y_s\,\dif s+\mathcal{O}\left(\frac{n^\xi}{n\eta_{*,t}\sqrt{|\eta_{1,t}\eta_{2,t}|}}\right).
\end{equation}
Finally, by the integral Gronwall inequality, using
\[
\exp\left(\int_s^t 2|\langle M^{z_{1,s}z_{2,s}}(\ii\eta_{1,s},I,\ii\eta_{2,s})\rangle|\,\dif s\right)\lesssim \frac{\eta_{*,s}\sqrt{\eta_{1,s}\eta_{2,s}}}{\eta_{*,t}\sqrt{\eta_{1,t}\eta_{2,t}}},
\]
which is \cite[Eq. (5.34)]{cipolloni2023mesoscopic}, we obtain the desired bound for $Y_t$.

For Part 2 we want to gain from $|z_1-z_2|$ being large, for this reason we do not want to separate $G_{1,t}G_{2,t}$ when they come close to each other. For this reason, defining
\[
\begin{split}
Y_t:&=\sup_{C_1,C_2\in \{A,B,E_1,E_2\}}|\langle \big(G_{1,t}(\ii\eta_{1,t})C_1G_{2,t}(\ii\eta_{2,t})-M^{z_{1,t},z_{2,t}}(\ii\eta_{1,t},C_1,\ii\eta_{2,t})\big)C_2\rangle|\\
&\quad+\sup_{C_1,C_2\in \{A,B,E_1,E_2\}}|\langle \big(G_{1,t}(\ii\eta_{1,t})C_1G_{2,t}(-\ii\eta_{2,t})-M^{z_{1,t},z_{2,t}}(\ii\eta_{1,t},C_1,-\ii\eta_{2,t})\big)C_2\rangle|,
\end{split}
\]
we estimate
\begin{equation}
\begin{split}
\frac{1}{n}\big|\langle G_{1,t}^\mathfrak{t}E_i G_{1,t}AG_{2,t} BG_{1,t}E_j\rangle\big| &\lesssim \langle G_{1,t}^\mathfrak{t}E_i G_{1,t}A(G_{1,t}^\mathfrak{t}E_i G_{1,t}A)^*\rangle^{1/2}\langle(G_{2,t} BG_{1,t}E_j)^* G_{2,t} BG_{1,t}E_j\rangle^{1/2} \\
&\lesssim \frac{\langle \Im G_{1,t}\Im G_{2,t}\rangle^{1/2}}{\eta_{*,t}\sqrt{\eta_{1,t}\eta_{2,t}}}\lesssim \frac{Y_t^{1/2}}{\eta_{*,t}\sqrt{\eta_{1,t}\eta_{2,t}}},
\end{split}
\end{equation}
with overwhelming probability. Similarly, we estimate
\begin{equation}
\big|\langle [G_{1,t}AG_{2,t}]^\mathfrak{t}E_i G_{2,t}BG_{1,t}E_j\rangle\big|\lesssim \frac{Y_t}{\eta_{1,t}\eta_{2,t}},
\end{equation}
with overwhelming probability. Combining this with \cite[Eq. (5.39)]{cipolloni2023mesoscopic}, and the display below it, we obtain
\begin{equation}
\begin{split}
\label{eq:endreal}
Y_t&\le  Y_0+C\int_0^t \left(\frac{1}{|z_{1,s}-z_{2,s}|^2}+ \frac{n^\xi}{\sqrt{n}\eta_{*,s}^{3/2}}\right)Y_s\,\dif s+\frac{n^\xi}{N\sqrt{\eta_{*,t}|\eta_{1,t}\eta_{2,t}|(|z_{1,t}-z_{2,t}|^2+\eta_t^*)}}\\
&\quad+\frac{1}{\sqrt{n\eta_{*,t}}}\cdot\frac{n^{2\xi}}{n\eta_{*,t}\sqrt{|\eta_{1,t}\eta_{2,t}|}},
\end{split}
\end{equation}
which, by the integral Gronwall inequality, gives \eqref{eq:desiredllawrealcase}.

To conclude the proof, we only need to show that \eqref{eq:desiredllawrealcase} holds for matrices of type $M$, and that the same bound holds if one of the resolvents is replaced by its transpose. As an input we rely on the following single resolvent local laws for matrices of type $M$, whose proof is postponed to Appendix~\ref{sec:m-laws-proof}.
\bel
\label{lem:Mllaw}
For matrices of type $M$ as in Definition~\ref{def:type-M}, we have that
\beq
\big|\langle G^z( \i \eta ) - M^z ( \i \eta) \rangle\big| \leq \frac{n^\xi}{n \eta} , \qquad\quad \left| \langle {\bm x}, G^z( \i \eta) - M^z ( \i \eta ){\bm y}\rangle \right| \leq \frac{ n^\xi}{\sqrt{ n \eta}}
\eeq
with overwhelming probability, for any $\xi>0$ and any unit vectors ${\bm x}, {\bm y}$.
\eel

We now consider the Ornstein--Uhlenbeck flow
\[
\dif \widehat{X}_t=-\frac{1}{2}\widehat{X}_t \dif t+\frac{\dif \widehat{B}_t}{\sqrt{n}}, \qquad \widehat{X}_0=\widehat{X},
\]
with $\widehat{X}$ being a real i.i.d matrix such that its Hermitization $\widehat{H}^{z_i}$ satisfies \eqref{eq:desiredllawrealcase}. Then, it is easy to see that the resolvents $\widehat{G}_{i,t}:=(\widehat{H}^{z_{i,t}}-\ii\eta_{i,t})^{-1}$ satisfy \eqref{eq:evg1g2real} with $\beta=2$. The proof of \eqref{eq:desiredllawrealcase} in this case is then immediate by \cite[Proposition 5.3]{cipolloni2023mesoscopic} together with \eqref{eq:desiredllawrealcase} for real i.i.d. matrices to estimate the initial condition, as \cite[Theorem 3.3 and Proposition 5.3]{cipolloni2023mesoscopic} used the matrix to be complex only to bound the initial condition. We are thus left only with the case when one of the resolvents is replaced by its transposes. By It\^{o}'s formula we obtain (cf. \eqref{eq:evg1g2real})
\begin{equation}
\begin{split}
\label{eq:evg1g2realcompl}
    \dif \langle G_{1,t}AG_{2,t}^\mathfrak{t}B\rangle &=\sum_{ab}\partial_{ab}\langle G_{1,t}AG_{2,t}^\mathfrak{t}B\rangle \frac{\dif \widehat{\mathcal{B}}_{ab,t}}{\sqrt{n}}+\langle G_{1,t}AG_{2,t}^\mathfrak{t}B\rangle\dif t \\
    &\quad+\langle G_{1,t}-M_{1,t}\rangle\langle G_{1,t}AG_{2,t}^\mathfrak{t}BG_{1,t}\rangle\dif t +\langle G_{2,t}-M_{2,t}\rangle\langle G_{2,t}^\mathfrak{t}BG_{1,t}AG_{2,t}^\mathfrak{t}\rangle \dif t  \\
&\quad
%+\frac{1}{n}\tilde{\sum}_{ij}\langle G_{1,t}^\mathfrak{t}E_i G_{1,t}AG_{2,t} BG_{1,t}E_j\rangle \dif t
+\frac{2}{n}\tilde{\sum}_{ij}\langle [G_{1,t}AG_{2,t}]^\mathfrak{t}E_i G_{2,t}BG_{1,t}E_j\rangle\dif t \\
%&\quad+\frac{1}{n}\tilde{\sum}_{ij}\langle G_{2,t}^\mathfrak{t}E_i G_{2,t}BG_{1,t} AG_{2,t}E_j\rangle \dif t.
\end{split}
\end{equation}
Here $\widehat{\mathcal{B}}_t$ is the Hermitization of $\widehat{B}_t$. Additionally, for the deterministic approximation $M_t^{z_1,\overline{z_2}}$ of $ G_{1,t}AG_{2,t}^\mathfrak{t}$ we have the equation $\partial_t\langle M_t^{z_1,\overline{z_2}}\rangle=\langle M_t^{z_1,\overline{z_2}}\rangle$. Note that this evolution is analogous to \eqref{eq:evg1g2real} with the second and the first term of the fourth and fifth lines removed. For this reason, the estimate of the RHS of \eqref{eq:evg1g2realcompl} is completely analogous to \eqref{eq:evg1g2real}--\eqref{eq:endreal} and \cite[Eqs. (5.17) and (5.39)]{cipolloni2023mesoscopic}\normalcolor, and so omitted. This concludes the proof.

\qed

\end{proof}

\section{Local laws for matrices of mixed symmetry}
 \label{sec:m-laws-proof}

In this section we present the proof of the necessary averaged and isotropic single resolvent local laws for matrices of type $M$.

\

\begin{proof}[Proof of Lemma~\ref{prop:m-laws}]\, We first prove that \eqref{eqn:ll-1} holds, then we use this as an input to prove \eqref{eq:goodll} hold. We can consider $X$ to be the solution of \eqref{eq:OU} with initial data being a real i.i.d. matrix and $B_t$ being a matrix of complex Brownian motions. 

First, the proof of Lemma \ref{lem:ll-2} is easily modified to show that for all $\eps,\kappa>0$ we have that,
\beq \label{eqn:type-m-1} 
\left| \langle G_t^z (w) - M_t^z (w) \rangle \right| \leq \frac{ n^{\eps}}{n  \Im[w]}
\eeq
for all $n^{\eps-1} \leq \Im[w] \leq 10$ and $\Re[w]$ such that $\rho^z ( \Re [w] ) > \kappa$.  With this as input, all of the arguments in Section \ref{sec:dynamical-local-law} apply line-by-line, yielding the estimate \eqref{eq:goodll}. Corollary~\ref{lem:goodrig} and \eqref{eqn:usual-rigidity} are a direct consequence of these local laws. 

Next, \eqref{eqn:type-m-rig} is a consequence of \eqref{eqn:type-m-1} and the estimate for $\lambda_n^z$ follows from comparing the eigenvalues of the Hermitization of $(1-t)^{1/2} Y + t^{1/2} G$ to those of $Y$ using the Weyl bound $| \lambda_i (A) - \lambda_i (B) | \leq \| A - B \|$.

\end{proof}

\qed

\

\begin{proof}[Proof of Lemma~\ref{lem:Mllaw}] \, The averaged law follows by Lemma~\ref{prop:m-laws}. We now use the averaged law to prove the isotropic law. We only present a sketch of the proof for brevity, as it is very similar (actually simpler) to the proof of Lemma \ref{lem:ll-2}.

By It\^{o} formula, it is easy to see that along the characteristics \eqref{eq:charnew} we have (we use the notations $G_t:=(H_t^{z_t}-\i \eta_t)^{-1}$, $M_t:=M^{z_t}(\i\eta_t)$)
\begin{equation}
\label{eq:isoflow}
\dif (G_t-M_t)_{{\bm x}{\bm y}}=\frac{1}{\sqrt{n}}\sum_{ab=1}^{2n}\partial_{ab}(G_t)_{{\bm x}{\bm y}}\dif (\mathfrak{B}_t)_{ab}+\frac{1}{2}(G_t-M_t)_{{\bm x}{\bm y}}\dif t+\langle G_t-M_t\rangle (G_t^2)_{{\bm x}{\bm y}}\dif t,
\end{equation}
where we used the short--hand notation $(G_t)_{{\bm x}{\bm y}}:=\langle {\bm x}, G_t{\bm y}\rangle$. Here $\mathfrak{B}_t$ is the Hermitization of $B_t$ defined in \eqref{eq:defbigBM}; in particular, $(\mathfrak{B}_t)_{ab}=0$ for $a,b\le n$ and $a,b\ge n+1$. Notice  that the second term in the RHS of \eqref{eq:isoflow} can be removed by looking at the evolution of $e^{-t/2}(G_t-M_t)_{{\bm x}{\bm y}}$; we thus neglect this term. Define the stopping time
\[
\tau:=\inf\left\{t: \sup_{{\bm u},{\bm y}\in \{{\bm x},{\bm y}\}}\big|(G_t-M_t)_{{\bm u}{\bm v}}\big|=\frac{n^{2\xi}}{\sqrt{n\eta_t}}\right\}.
\]
Then, we estimate
\[
\left|\int_0^{t\wedge \tau} \langle G_s-M_s\rangle (G_s^2)_{{\bm x}{\bm y}}\,\dif s\right|\le \frac{n^{\xi}}{\sqrt{n\eta_{t\wedge\tau}}}.
\]
Finally, the quadratic variation process of the stochastic term in the RHS of \eqref{eq:isoflow} can be estimated by
\[
\frac{1}{n\eta_t^2}(\Im G_t)_{{\bm x}{\bm x}}(\Im G_t)_{{\bm y}{\bm y}}\lesssim \frac{1}{n\eta_{t\wedge\tau}^2}.
\]
Then, using the BDG inequality we see that the stochastic term is also bounded by $n^\xi/\sqrt{n\eta_{t\wedge\tau}}$. This concludes the proof. 
    
\end{proof}

\qed

\section{Miscellaneous results}

\subsection{Proof of Lemma~\ref{lem:proprho}}

This lemma follows by the analysis of the density of states from \cite[Proposition 3.1]{bourgade2014local}. However, these results concerns the limiting density of states of the $N\times N$ matrix $(X-z)(X-z)^*$ rather than the one of the $2N\times 2N$ matrix $H^z$. For this reason we first relate these two densities and then conclude the proof by \cite[Proposition 3.1]{bourgade2014local}.

For $x>0$, define
\[
\widetilde{\rho}^z(x):=\lim_{\eta\to 0^+}\frac{1}{\pi}\Im \widetilde{m}^z(x+\ii \eta),
\]
with $\widetilde{m}^z(w)$ being the unique solution of (see \cite[Eq. (11)]{cipolloni2020optimal}):
\[
-\frac{1}{\widetilde{m}^z(w)}=w(1+\widetilde{m}^z(w))-\frac{|z|^2}{1+\widetilde{m}^z(w)}, \qquad\quad \Im[w]\Im[\widetilde{m}^z]>0.
\]
Then, it is easy to check that $\rho^z(x)=x\widetilde{\rho}^z(x^2)$ for $x\in \R$ (see e.g. the sentence below \cite[Eq. (11)]{cipolloni2020optimal}), and so that $\rho^z$ is symmetric around the origin. Using this relation, the points (i)--(iv) immediately follow from \cite[Proposition 3.1]{bourgade2014local} (see also \cite[Eqs. (18a)--(18b)]{cipolloni2020optimal}). The last point follows by the display above Eq. (3.14) of \cite{cipolloni2024precise}, which can be derived by explicitly solving the equation \eqref{eq:mde} via Cardano's formula. \qed

\subsection{Proof of Proposition \ref{prop:der-bd}} \label{sec:der-bd}

\bel \label{lem:der-1}
Let $z_s := z + s w$, with $|z_s| < 1$ for some interval of $|s| < r$. Then,
\beq
\frac{ \d}{\d s} \Re  \int \log ( x - \i \eta ) \rho^{z_s} (x) \d x =  \Re[ \dot{z}_s \bar{z_s} ] \left( u^{z_s}(\i \eta ) \right),
\eeq
where $\dot{F}$ means $\frac{ \d }{ \d s} F$.
\eel
\proof We can write,
\beq
\Re \int \log ( x - \i \eta ) \rho^{z_s} (x) \d x = -\i \int_0^\eta  m^{z_s} ( \i u)  \d u +\frac{ -1+ |z_s|^2 }{2}
\eeq
using $m^z ( \i u ) = \i \Im [ m^z (\i u)]$. Here we used the fact that
\beq
\Re \int \log (x) \rho^{z} (x) \d x = \frac{1}{ \pi} \int_{ |u| <1 } \log |u-z| \d u \d \bar{u} = \frac{1}{2}  = \frac{ |z|^2-1}{2}
\eeq
for $|z| <1$. The first equality may be deduced from the fact that both sides are the almost sure limit of $ \log \det |X-z|$. 
Differentiating \eqref{eq:mde} wrt $s$ and re-arranging yields,
\beq
 \dot{m}^{z_s} (w) = \frac{ 2 \Re[ \dot{z}_s \bar{z}_s] (m^{z_s} (w) )^2}{2 ( m^{z_s} (w) )^2 (w+m^{z_s} (w) ) - w }.
\eeq
Similarly, differentiating \eqref{eq:mde} wrt $w$ and re-arranging yields,
\beq
\partial_w m^{z_s} (w) = - m^{z_s} (w) \frac{ 2 (w + m^{z_s} (w) ) m^{z_s} (w) +1}{2 ( m^{z_s} (w) )^2 (w+m^{z_s} (w) ) - w }.
\eeq
Therefore,
\begin{align}
\partial_w u^{z_s} (w) &= \partial_w \left( \frac{ m^{z_s} (w) }{ m^{z_s} (w) + w } \right) = \frac{ (\partial_w m^{z_s} (w) ) w - m^{z_s} (w) }{(m^{z_s} (w)  + w)^2} \\
&=  - \frac{ 2( m^{z_s} (w))^2}{ 2 ( m^{z_s} (w) )^2 (w+ m^{z_s} (w) ) - w} 
\end{align}
with the last line following by direct substitution. Therefore, $ \dot{m}^{z_s} (w) = -\Re[ \dot{z_s} \bar{z}_s ] \partial_w u^{z_s} (w)$, and so
\beq
-\i \int_0^\eta  \dot{m}^{z_s} ( \i u)= \Re[ \dot{z}_s \bar{z}_s ] \int_{0}^\eta \del_w u^{z_s} (\i u) \i \d u = \Re[ \dot{z}_s \bar{z}_s ] ( u^{z_s} ( \i \eta ) - 1 )
\eeq
The claim follows. \qed

\bel \label{lem:der-2}
Let $z_s := z + s w$, with $|z_s| < 1$ for some interval of $|s| < r$. Let $\eps >0$. Then with overwhelming probability we have,
\beq
\frac{ \d}{ \d s } \Re \log \det ( H^{z_t} - \i \eta )  = n 2 \Re[ \dot{z}_s \bar{z}_s] u^{z_s} ( \i \eta ) + \O ( n^{\eps} \eta^{-1/2} )
\eeq
for $n^{\eps} / n \leq \eta \leq \delta$ for some $\delta >0$.  
\eel
\proof We have that,
\beq
\frac{ \d}{ \d s }  \log \det ( H^{z_s} - \i \eta ) = \tr \left( \frac{1}{ H^{z_s} - \i \eta } ( -\dot{z}_s F - \dot{\bar{z}}_s F^* ) \right)
\eeq
where $ F = \left( \begin{matrix} 0 & 1 \\ 0 & 0 \end{matrix} \right)$. But by \cite[Theorem 4.5]{cipolloni2023optimal} we have
\beq \label{eqn:der-bd-1}
\tr (G ( \i \eta ) F) = - n\bar{z}_s u^{z_s} ( \i \eta) + \O ( \eta^{-1/2} n^{\eps} ), 
\eeq
and a similar estimate for $F^*$. Here we used the fact $F$ is regular in the sense of \cite[Definition 4.2]{cipolloni2023optimal}; see also the discussion in the proof of Theorem 2.4 near Eq. (3.22) of \cite{cipolloni2023optimal}. \qed

\bel \label{lem:der-3} 
Let $0 < r < 1$ and fix any small $\xi>0$. Uniformly in $\eta_1, \eta_2$ and $z$ satisfying $n^{-\eps/10} /n \leq \eta_1 \leq \eta_2 \leq 1$ and $|z|\le r$ we have with overwhelming probability,
\beq
\left| \langle G^z ( \i \eta_1) F \rangle - \langle G^z ( \i \eta_2 ) F \rangle \right| \leq ( \eta_2 - \eta_1 ) n^{1/2+\xi}
\eeq
\eel
\proof By the spectral theorem we have,
\beq
\tr G^z ( \i \eta) F = \sum_{i} \frac{ 1}{ \lambda_i^z - \i \eta } u_i^* v_i
\eeq
where $u_i, v_i$ are the normalized eigenvectors of $X^*X$ and $X X^*$, respectively. By \cite[Eq. (2.8c)]{cipolloni2023optimal} we see that,
\beq \label{eqn:der-bd-2}
\left| u_i^* v_i \right| \leq n^\xi \left( \frac{1}{ \sqrt{n}} + \frac{ |i|}{n} \right)
\eeq
for $|i| \leq c n$ for some $c>0$, with overwhelming probability. Therefore,
\beq
\left| \sum_{ |i| < c n } u_i^* v_i \left( \frac{1}{ \lambda_i^z - \i \eta_1} - \frac{1}{ \lambda_i^z - \i \eta_2 } \right) \right| \leq n^{3 \xi} (\eta_2 - \eta_1) \sum_{ |i| < c n } \frac{ n^2}{ i^2} \left( \frac{1}{ n^{1/2}} + \frac{ |i|}{n} \right) \leq n^{4\xi} (\eta_2-\eta_1)  n^{3/2} 
\eeq
Using $|u_i^* v_i | \leq 1 $ for all $i>0$ we easily see that
\beq
\left| \sum_{ |i| < c n } u_i^* v_i \left( \frac{1}{ \lambda_i^z - \i \eta_1} - \frac{1}{ \lambda_i^z - \i \eta_2 } \right) \right| \leq C n (\eta_2 - \eta_1),
\eeq
and the claim follows.  \qed

\

\noindent{\bf Proof of Proposition \ref{prop:der-bd}}.  We treat first the complex case $\beta=2$. The claim for $\eta \geq n^{\eps} / n$ follows immediately from differentiation and Lemmas \ref{lem:der-1} and \ref{lem:der-2}. For $1/n \leq \eta \leq n^{\eps} /n$ we apply Lemma \ref{lem:der-3} with $\eta_1 = \eta$ and $\eta_2 = n^{\eps}/n$ to approximate $\langle G( \i \eta_1) F \rangle = \langle G ( \i \eta_2 ) F \rangle + \O ( n^{\eps} n^{-1/2} )$. The claim then follows because $u^{z} ( \i \eta_1) = u^{z} ( \i \eta_2) + \O ( n^{-1+\eps})$. In the real case we need only to check that the additional deterministic term present in \eqref{eqn:psi-def} also obeys the same estimate. But this is clear because for $\eta \geq n^{-1}$ the difference between these terms is,
\beq
\left| \log ( 4 \Re[z_1]^2 + 2 \eta \sqrt{1-|z_1|^2} ) - \log ( 4 \Re[z_2]^2 + 2 \eta \sqrt{ 1 - |z_2|^2} ) \right| \leq C \frac{ |z_1-z_2|}{ \sqrt{\eta}}
\eeq
under the assumption $|z_1 -z_2 | \leq \sqrt{\eta}$. If this does not hold, then the deterministic term can be absorbed into the error on the RHS of \eqref{eqn:der-bd}. 
\qed

\subsection{Proof of Corollary \ref{lem:goodrig}} \label{a:goodrig-proof}

The proof will follow by the estimate \eqref{eq:goodll} together with an application of the well known Helffer-Sj\"ostrand formula (see \eqref{eq:HS} below). Let $f:\R\to\R$ be a smooth function, and define its almost analytic extension by
\begin{equation}
\label{eq:almostanext}
\tilf (x+\ii y):=\big[f(x)+\ii \partial_x f(x)\big]\chi_{\mathfrak{a}}(y).
\end{equation}
Here $\chi_\mfa(y)$ is a smooth cut--off function such that $\chi_\mfa(y)=1$ for $|y|\le \mathfrak{a}$ and $\chi_\mfa(y)=0$ for $|y|\ge 2\mathfrak{a}$, for some $n$--dependent $\mathfrak{a}$ satisfying $n^{-\delta} \geq \mfa \ge (\log n)^{1/2+10\delta}$, which we will choose later in the proof. We may also assume that $| \chi'_\mfa (y) | \leq C / \mfa$ and that $\chi_\mfa (y)$ is even. Then, by the Helffer-Sj\"ostrand formula, we have
\begin{equation}
\label{eq:HS}
f(\lambda)=\frac{1}{\pi}\int_\C \frac{\partial_{\overline{w}} \tilf (w)}{\lambda-w}\,\dif^2 w=\frac{1}{\pi}\int_{\R}\int_\R \frac{\partial_{\overline{w}}\tilf (w)}{\lambda-w}\,\dif x\dif y.
\end{equation}

We now choose $f(x)$ to be a smooth function such that $f(x)=1$ for $|x|\le B$ and $f(x)=0$ for $|x|\ge A+B$, for some $n$--dependent $A,B$, which we will choose shortly. We point out that this $f$ is a smooth version of the eigenvalue counting function as a consequence of the symmetry of the spectrum of $H^z$ around zero. Using \eqref{eq:HS}, we have
\begin{equation}
\langle f(H^z)\rangle-\int_\R f(x) \rho^z(x)\,\dif x =\frac{1}{\pi}\int_\C\partial_{\overline{w}} \tilf (w) \langle G^z(w)-M^z(w)\rangle\, \dif^2 w.
\end{equation}

Notice that by \eqref{eq:almostanext} we have
\begin{equation}
\label{eq:esttwoterm}
\partial_{\overline{w}}\tilf (x+\ii y)=\ii y f''(x)\chi_\mathfrak{a}(y)+\big[f(x)+\ii \partial_x f(x)\big]\chi_\mathfrak{a}'(y).
\end{equation}
We now estimate the two terms in the RHS of this equality one by one. Using Lemma~\ref{lem:precllaw}, we have, using that $\chi_\mfa' (y) \neq 0$ only for $\mfa \leq y \leq 2 \mfa$,
\begin{equation}
\left|\frac{1}{\pi}\int_{\R^2}\chi_\mfa'(y)[f(x)+\ii yf'(x)\big] \langle G^z(x+\ii y)-M^z(x+\ii y)\rangle\, \dif x\dif y\right|\lesssim \frac{(\log n)^{1/2+\delta}}{n} \left(1+\frac{B+A}{\mathfrak{a}}\right),
\end{equation}

We now estimate the first term in the RHS of \eqref{eq:esttwoterm}. Let $y_0:=(\log n)^{1/2+10\delta}$. Since $y \to \Im \langle G(\i y ) \rangle$ is monotonically increasing, and since $\langle M( \i y) \rangle$ is a bounded function, for $y < y_0$,
\begin{equation}
\label{eq:monoton}
\langle \Im[ G^z(\ii y)] \rangle \leq C \frac{y_0}{y}, \qquad \langle \Im[ M^z ( \ii y ) ] \rangle \leq C \leq  C \frac{y_0}{y}
\end{equation}
with overwhelming probability. 
By \eqref{eq:monoton}, we thus estimate
\begin{align}
\label{eq:smally}
& \left|\int_{|y|< y_0} \chi_\mfa(y)yf''(x) \langle G^z(x+\ii y)-M^z(x+\ii y)\rangle\, \dif x\dif y\right| \\
= & 2 \left|\int_{0< y< y_0} \chi_\mfa(y)yf''(x) \langle \Im[ G^z(x+\ii y)-M^z(x+\ii y)]\rangle\, \dif x\dif y\right| 
 \le \frac{y_0^2}{A}= \frac{(\log n)^{1/2+10\delta}}{n},
\end{align} 
with overwhelming probability, where in the last equality we chose $A=y_0$.  The first inequality follows because we assumed that $\chi_\mfa (y)$ is even. 

Before turning to the portion of the integral with $y \geq y_0$, we first remark that from \eqref{eq:goodll} and the Cauchy Integral formula we obtain that
\begin{equation}
\label{eq:goodllG2}
|\partial_w\langle G(w)-M^z(w)\rangle|\lesssim \frac{(\log n)^{1/2+\delta}}{n|\Im w|^2},
\end{equation}
with overwhelming probability, uniformly in $n^{-\delta}\ge|\Im w|\ge (\log n)^{1/2+\delta}/n$ and $\Re (w) \in I_z ( \kappa)$. 

Now for $y\ge y_0$, we estimate
\begin{equation}
\begin{split}
&\left|\frac{1}{\pi}\int_{|y|\ge y_0} \chi_\mfa(y)yf''(x) \langle G^z(x+\ii y)-M^z(x+\ii y)\rangle\, \dif x\dif y\right| \\
&\qquad\qquad\qquad\quad=\left|\frac{1}{\pi}\int_{|y|\ge y_0} \chi_\mfa(y)yf'(x) \partial_w\langle G^z(x+\ii y)-M^z(x+\ii y)\rangle\, \dif x\dif y\right|\\
&\qquad\qquad\qquad\quad\lesssim \frac{(\log n)^{1/2+\delta}}{n}\int_{\mathfrak{a}\ge |y|\ge y_0} \frac{1}{|y|}\,\dif y \\
&\qquad\qquad\qquad\quad\lesssim  \frac{(\log n)^{1/2+\delta}}{n} \big(1+\log(\mathfrak{a}/y_0)\big),
\end{split}
\end{equation}
where in the first equality we used integration by parts and used that $\partial_x F(w) = \partial_w F(w)$ for holomorphic $F$ by the Cauchy-Riemann equations. In the first inequality we used \eqref{eq:goodllG2}.   Now as long as $\mfa$ satisfies $(\log n)^{10\delta} / n \leq \mfa \leq (\log n)^{C_1} / n$ for any fixed $C_1 >0$ we see that,
\begin{equation}
\left|\frac{1}{\pi}\int_{\R^2}\chi_\mfa(y)yf''(x) \langle G^z(x+\ii y)-M^z(x+\ii y)\rangle\, \dif x\dif y\right|\lesssim  \frac{(\log n)^{1/2+10\delta}}{n}.
\end{equation}

For any $B \leq \mathfrak{a}$, we therefore conclude that with overwhelming probability,
\begin{equation}
\left|\langle f(H^z)\rangle-\int_\R f(x) \rho^z(x)\,\dif x\right|\lesssim \frac{(\log n)^{1/2+10\delta}}{n}.
\end{equation}
By choosing $B$ to be the location of the quantiles $\gamma_i^z$ for $0 < i < (\log n)^{C_1}$, for some fixed $C_1 >0$, and $\mfa = B \vee ((\log n)^{10\delta} / n)$, and using the symmetry $\lambda_{i}^z = \lambda_{-i}^z$, one therefore finds the estimate,
\beq
\left| | \{ j :  0 \leq \lambda_j^z \leq \gamma_i^z \} - i \right| \leq C (\log n)^{1/2+10\delta} 
\eeq
with overwhelming probability. 
From this, the estimate  \eqref{eq:vergoodrig}  follows immediately. For \eqref{eq:goodrig}, one can repeat the above argument, instead choosing $\mathfrak{a} = n^{-\eps}$ for some $\eps >0$. Then one finds that with overwhelming probability, for all $i < n^{1-\eps}$, we have
\beq
\left| | \{ j :  0 \leq \lambda_j^z \leq \gamma_i^z \} - i \right| \leq (\log n)^{3/2+20\delta},
\eeq
which implies \eqref{eq:goodrig}. \qed

\section{Proof of Proposition~\ref{prop:CLT-1}} \label{app:addtechres}

The proof of this proposition is divided into two part: using Stein's method (see \eqref{eq:diffeqpsi} and Lemma~\ref{pro:stein} below), we first show \eqref{eq:steinexpl} for general smooth test functions $f$, and then we specialize to $f(x)=\Re \log (x-\ii\eta)$ for which we explicitly compute the expectation and the variance $V(f)$ in \eqref{eq:explvarlog}. 

\

\begin{proof}[Proof of \eqref{eq:steinexpl}] \, The proof of \eqref{eq:steinexpl} is based on \cite[Section 10]{landon2022almost}. The main difference is that \cite{landon2022almost} considers Wigner (Hermitian) matrices, while now we consider the technically more complicated Hermitization $H^z$ of an i.i.d. matrix $X-z$ defined as in \eqref{eq:herm}. On the other hand \cite{landon2022almost} needed a higher precision in the estimate of the error term.

Let $f$ be a smooth function supported on $[-5,5]$, and for any $k\in \N$ we denote its almost analytic extension by
\[
\widetilde{f}(x+\ii y)=\widetilde{f}_k(x+\ii y):=\chi(y) \sum_{l=0}^k (\ii y)^kf^{(l)}(x),
\] 
with $\chi$ a smooth cut-off function which is equal to one for $|y|\le 1$ and equal to zero for $|y|\ge 2$. Here $f^{(l)}$ denotes the $l$--th derivative of $f$. In particular, it is easy to see that
\begin{equation}
\label{eq:boundftilde}
\big|\partial_{\overline{w}}\widetilde{f}(w)\big|\lesssim |\Im w|^{k-1}\lVert f\rVert_{C^k}\lesssim |\Im w|^{k-1} n^{\gamma k}.
\end{equation}
The precise value of $k$ will be chosen shortly. We recall that by Helffer-Sj\"ostrand formula we can write
\[
\langle f(H^z)\rangle =\frac{1}{\pi}\int_\C \partial_{\overline{w}}\widetilde{f}(w) \langle G^z(w)\rangle\,\dif^2 w.
\]
We thus define the characteristic function:
\begin{equation}
e(\lambda):=\exp\left[\ii \lambda \left(\frac{2n}{\pi}\int_\C\partial_{\overline{w}}\widetilde{f}(w) \big(\langle G^z(w)\rangle-\E\langle G^z(w)\rangle \big) \, \dif^2 w\right)\right],
\end{equation}
and its approximation
\begin{equation}
e_\mathfrak{a}(\lambda):=\exp\left[\ii \lambda \left(\frac{2n}{\pi}\int_{\Omega_a}\partial_{\overline{w}}\widetilde{f}(w) \big(\langle G^z(w)\rangle-\E\langle G^z(w)\rangle \big) \, \dif^2 w\right)\right].
\end{equation}
where
\begin{equation}
\label{eq:defomegafraka}
\Omega_{\mathfrak{a}}:=\left\{(x,y)\in\R^2: |y|\ge n^{-\mathfrak{a}}\right\},
\end{equation}
for some $\mathfrak{a}>0$ which we will choose shortly.

The goal of this section is to use Stein's method to compute $\psi_\mathfrak{a}(\lambda):=\E e_\mathfrak{a}(\lambda)$, and thus $\psi(\lambda):=\E e(\lambda)$. In particular, we use that $\psi_\mathfrak{a}(\lambda)$ is a very good approximation of $\psi(\lambda)$ as a consequence of
\begin{equation}
\label{eq:vergoodappr}
\left|\frac{2n}{\pi}\int_{\C\setminus\Omega_a}\partial_{\overline{w}}\widetilde{f}(w) \big(\langle G^z(w)\rangle-\E\langle G^z(w)\rangle \big) \, \dif^2 w \right|\lesssim n^{-(\gamma-\mathfrak{a})k} \lesssim n^{-1/2},
\end{equation}
choosing $\mathfrak{a}=1/100$ and $k$ large enough (in terms of $\mathfrak{a}$) in the last inequality, say $k=60$.  This follows from the local law \eqref{eq:goodll} and the bound \eqref{eq:boundftilde}. Using \eqref{eq:vergoodappr}, we then immediately conclude
\begin{equation}
\label{eq:closepsi}
\big|\psi_\mathfrak{a}(\lambda)-\psi(\lambda)\big|\lesssim  n^{-1/2+1/100},
\end{equation}
for $|\lambda|\le n^{1/100}$.

Given \eqref{eq:vergoodappr}, the main ingredient to prove \eqref{eq:steinexpl} is to compute $\E[e_\mathfrak{a}(\lambda) \langle G^z(w)-\E G^z(w)\rangle]$ as in the following lemma. We present its proof after the conclusion of the proof of \eqref{eq:steinexpl}.
\begin{lemma}
\label{pro:stein}
Fix any sufficiently small $\gamma>0$ as in Proposition~\ref{prop:CLT-1}, then we have
\begin{equation}
\begin{split}
\label{eq:explsteinvar}
&2n\E[e_\mathfrak{a}(\lambda) \langle G^z(w)-\E G^z(w)\rangle] \\
&=\frac{\ii\lambda\psi_\mathfrak{a}(\lambda)}{2\pi} \tilde{\sum}_{ij}\int_{\Omega_\mathfrak{a}}\partial_{\overline{w_1}}\widetilde{f}(w_1) \langle M^{z,z,z}(w_1,I,w_1,E_i,w) A(w)E_j\rangle   \,\dif^2 w_1  \\
&\quad+\frac{\ii\bm1_{\{\beta=1\}}\lambda\psi_\mathfrak{a}(\lambda)}{2\pi} \tilde{\sum}_{ij}\int_{\Omega_\mathfrak{a}}\partial_{\overline{w_1}}\widetilde{f}(w_1) \langle M^{z,z,\overline{z}}(w_1,I,w_1,E_iA(w)^{\mathfrak{t}},w) E_j\rangle   \,\dif^2 w_1 \\
&\quad+ \frac{\ii\lambda\kappa_4\psi_\mathfrak{a}(\lambda)}{2\pi}\int_{\Omega_\mathfrak{a}}\partial_{\overline{w_1}} \widetilde{f}(w_1) (m^z(w_1)^2)'  (m^z(w)^2)'\, \dif ^2 w_1+\mathcal{O}\left(\frac{n^{200\gamma+5\mathfrak{a}}}{n^{1/2}}\right) \\
&\qquad\qquad\qquad\qquad\qquad\qquad=: -\ii \int_{\Omega_\mathfrak{a}} \partial_{\overline{w_1}}\widetilde{f}(w_1) R(w_1,w)\, \dif w_1^2+\mathcal{O}\left(\frac{n^{200\gamma+5\mathfrak{a}}}{n^{1/2}}\right).
\end{split}
\end{equation}
Here we used the definitions $\kappa_4:=n^2\E|X_{ab}|^4-(2+\bm1_{\{\beta=1\}})$ for the fourth cumulant, 
\begin{equation}
\label{eq:defaw}
A(w):=\frac{M^z(w)}{1-\langle (M^z(w))^2\rangle},
\end{equation}
and for $z_i\in\C$, $B_i\in \C^{2n \times 2n}$,
\begin{equation}
\begin{split}
\label{eq:defm123}
M^{z_1,z_2}(w_1,B_1,w_2):&=\big[1-M^z(w_1)\mathcal{S}[\cdot]M^{z_2}(w_2)\big]^{-1}\big[M^{z_1}(w_1)B_1 M^{z_2}(w_2)\big] \\
M^{z_1,z_2,z_3}(w_1,B_1,w_2,B_2,w_3):&=\big[1-M^{z_1}(w_1)\mathcal{S}[\cdot]M^{z_3}(w_3)\big]^{-1}\bigg[M^{z_1}(w_1)B_1M^{z_2,z_3}(w_2,B_2,w_3) \\
&\qquad\qquad\qquad\quad+M^{z_1}(w_1)\mathcal{S}[M^{z_1,z_2}(w_1,B_1,w_2)]M^{z_2,z_3}(w_2,B_2,w_3)\bigg].
\end{split}
\end{equation}
\end{lemma}

Using Lemma~\ref{pro:stein}, we then compute
\begin{equation}
\label{eq:diffeqpsi}
\frac{\dif}{\dif \lambda}\psi_\mathfrak{a}(\lambda)= \frac{2\ii n}{\pi}\int_{\Omega_a}\partial_{\overline{w}}\widetilde{f}(w) \E\big[e_\mathfrak{a}(\lambda)\langle G^z(w)\rangle-\E\langle G^z(w)\rangle \big] \, \dif^2 w=-\lambda V_\mathfrak{a}(f)\psi_\mathfrak{a}(\lambda)+\mathcal{O}\left(\frac{n^{200\gamma+5\mathfrak{a}}}{n^{1/2}}\right),
\end{equation}
where we defined
\begin{equation}
\label{eq:defvfusef}
V(f)=V_\mathfrak{a}(f):=\frac{1}{\pi^2}\int_{\Omega_\mathfrak{a}} \int_{\Omega_\mathfrak{a}} \partial_{\overline{w_1}}\widetilde{f}(w_1) \partial_{\overline{w}}\widetilde{f}(w) R(w_1,w)\, \dif w_1^2\dif^2 w.
\end{equation}
We now claim the lower bound (the proof is presented after the end of the proof of \eqref{eq:steinexpl})
\begin{equation}
\label{eq:deslobvaf}
V(f)\ge - n^{-1/5}.
\end{equation}
This ensures that $\exp(-\lambda^2V(f))\lesssim 1$ for $|\lambda|\le n^{1/100}$. From \eqref{eq:diffeqpsi} we thus conclude
\begin{equation}
\label{eq:comppsial}
\psi_\mathfrak{a}(\lambda)=\exp(-\lambda^2V(f)/2)+\mathcal{O}(n^{-1/2+200\gamma+5\mathfrak{a}}),
\end{equation}
for all $| \lambda | \leq n^{1/100}$. 
Combining \eqref{eq:comppsial} with \eqref{eq:closepsi} we conclude \eqref{eq:steinexpl}.

\qed

\end{proof}

We now present the proof of a few technical results that we used within the proof of \eqref{eq:steinexpl}.

\

\begin{proof}[Proof of \eqref{eq:deslobvaf}] \, Define
\begin{equation}
\label{eq:defrvz}
Z:= \frac{2n}{\pi}\int_{\Omega_a}\partial_{\overline{w}}\widetilde{f}(w) \big(\langle G^z(w)\rangle-\E\langle G^z(w)\rangle \big) \, \dif^2 w,
\end{equation}
then by the local law \eqref{eq:goodll} and the bound \eqref{eq:boundftilde} we have $|Z|\le n^{200\gamma}$. Choose $\lambda=n^{-1/4}$, then, proceeding similarly to the proof of \cite[Lemma 5.10]{landon2024single}, using $|Z|\le n^{200\gamma}$, we obtain
\begin{equation}
\label{eq:trick1}
\E\big[\ii Z e^{\ii\lambda Z}\big]=-\lambda \mathrm{Var}(Z)+\mathcal{O}(|\lambda|^2 n^{200\gamma})=-\lambda \mathrm{Var}(Z)\E\big[e^{\ii\lambda Z}\big]+\mathcal{O}(|\lambda|^2 n^{200\gamma}),
\end{equation}
where in the second equality we used $\E\big[e^{\ii\lambda Z}\big]=1+\mathcal{O}(|\lambda|n^{200\gamma})$. On the other hand, by \eqref{eq:diffeqpsi}, we have
\begin{equation}
\label{eq:trick2}
\E\big[\ii Z e^{\ii\lambda Z}\big]=\frac{\dif}{\dif \lambda}\E\big[e^{\ii\lambda Z}\big]=-\lambda V_\mathfrak{a}(f)\E\big[e^{\ii\lambda Z}\big]+\mathcal{O}(n^{-1/2+{200\gamma}}).
\end{equation}
Subtracting \eqref{eq:trick2} to \eqref{eq:trick1} and dividing by $\lambda$, we conclude
\begin{equation}
\label{eq:usvarrel}
V(f)=\mathrm{Var}(Z)+\mathcal{O}(n^{-1/4+200\gamma}),
\end{equation}
which implies the desired lower bound on $V(f)$, for $\gamma$ sufficiently small.

\qed
    
\end{proof}

\begin{proof}[Proof of Lemma~\ref{pro:stein}] \, The proof of this lemma is similar to the proof of \cite[Proposition 10.1]{landon2022almost}; for this reason we only present the main differences and omit some of the technical details which can be readily seen to adapt to the current case in an immediate way.

Within this proof we may often omit the $z,w$--dependence of the resolvent $G=G^z(w)$ and of its deterministic approximation $M=M^z(w)$, to keep the notation simpler. Let $W$ be the Hermitization of $X$ defined as in \eqref{eq:herm} with $X-z$ replaced with $X$. Then we have
\begin{equation}
\label{eq:equation}
\mathcal{B}[G-M]=-M\underline{WG}+M\langle G-M\rangle (G-M), \qquad\quad \mathcal{B}[\cdot]:=1-M\mathcal{S}[\cdot]M,
\end{equation}
with
\begin{equation}
\label{eq:defunderline}
\underline{WG}:=WG+\langle G\rangle G.
\end{equation}
We point out that the definition of the underline term in \eqref{eq:defunderline} is so that $\E \langle \underline{WG}A\rangle\approx 0$. Using \eqref{eq:equation}, for the (normalized) trace of $G-\E G$ we get
\begin{equation} \label{eqn:a-clt-1}
\langle G-\E G\rangle=-\langle \underline{WG}A\rangle+\E \langle \underline{WG}A\rangle+\langle G-M\rangle\langle (G-M)A\rangle-\E\langle G-M\rangle\langle (G-M)A\rangle,
\end{equation}
with
\[
A=A(w):=((\mathcal{B}^{-1})^*[1])^*M(w).
\]
We point out that using
\begin{equation}
\label{eq:inversestabop}
\big((\mathcal{B}^{-1})^*[B^*]\big)^*=B+\frac{\langle (M^z(w))^2 B\rangle}{1-\langle (M^z(w))^2\rangle},
\end{equation}
for any $B\in\C^{2n \times 2n}$, we obtain that $A(w)$ can be written as in \eqref{eq:defaw}.

To reflect the block structure of $W$, throughout this section we use the short--hand notation:
\begin{equation} \label{eqn:sumab-def}
\sum_{ab}:=\sum_{1\le a\le n, \atop n+1\le b\le 2n }+\sum_{n+1\le a\le 2n, \atop 1\le b\le n}.
\end{equation}

Next, using that (recall $|\Im w|\ge n^{-\mathfrak{a}}$)
\[
n\langle G-M\rangle\langle (G-M)A\rangle\lesssim \frac{n^\xi}{n|\Im w|^2}\le \frac{n^{\xi+2\mathfrak{a}}}{n}
\]
with overwhelming probability by \eqref{eqn:ll-1}, and performing cumulant expansion (which was first used in the random matrix context in \cite{khorunzhy1996asymptotic} and then revived in \cite{he2017mesoscopic, lee2015edge}; see these references for more details), we have
\begin{equation}
\label{eq:writesum}
\begin{split}
 2n\E \big[e_\mathfrak{a} \langle G-\E G\rangle\big] &= 2\sum_{ab}\E[\partial_{ba}e_\mathfrak{a} \langle \Delta^{ab}G A\rangle]+2\bm1_{\{\beta=1\}}\sum_{ab}\E[\partial_{ab}e_\mathfrak{a} \langle \Delta^{ab}G A\rangle] \\
 &\quad+2n\sum_{k= 2}^R \sum_{ab}\sum_{{\bm \alpha}\in\{ab,ba\}^k}\frac{\kappa(ba,{\bm \alpha})}{k!}\big(\E\partial_{\bm \alpha} [e_\mathfrak{a}\langle \Delta^{ba}GA\rangle]-\psi_\mathfrak{a}(\xi)\E\partial_{\bm \alpha} [\langle \Delta^{ba}GA\rangle]\big) \\
 &\quad+\Omega_R+\O\left(\frac{n^{\xi+2\mathfrak{a}}}{n}\right).
\end{split}
\end{equation}
Notice that in the second line of \eqref{eq:writesum} we truncated the cumulant expansion at $k=R$. In fact, it is easy to see that $\Omega_R=\mathcal{O}(N^{-2})$ for $R=12$ (see e.g. \cite[Proposition 3.2]{erdHos2019random}). Here $\partial_{ab}:=\partial_{W_{ab}}$ denotes the directional derivative in the direction $W_{ab}$, $\partial_{\bm \alpha}:=\partial_{\alpha_1}\dots\partial_{\alpha_k}$, with $\alpha_i\in \{ab,ba\}$, and $\kappa(ba,{\bm \alpha})$ denotes the $k+1$--th cumulant of the random variables $W_{ab}, W_{\alpha_1},\dots, W_{\alpha_k}$, with ${\bm \alpha}:=(\alpha_1,\dots,\alpha_k)$. We point out that in \eqref{eq:writesum} we also used that the term $\langle G\rangle\langle GA\rangle$ from \eqref{eq:defunderline} cancels after cumulant expansion as a consequence of $\partial_{ba}\langle \Delta^{ab}GA\rangle=-\langle G\rangle\langle GA\rangle$. By analogous computations to \cite[Section 10.1]{landon2022almost} one can see that the only order one terms are the ones in the first line of \eqref{eq:writesum} and a certain term which comes from $k=3$ in the second line of \eqref{eq:writesum}. We thus compute precisely these three terms and neglect the estimates of the other terms as their estimate it completely analogous to \cite{landon2022almost}.

We thus compute
\begin{equation}
\begin{split}
\label{eq:maintermvar}
&2\sum_{ab}\E[\partial_{ba} e_\mathfrak{a}\langle \Delta^{ab}G A\rangle] \\
&\qquad =\frac{2\ii\lambda}{\pi}\tilde{\sum}_{ij}\int_{\Omega_\mathfrak{a}}\partial_{\overline{w_1}}\widetilde{f}(w_1) \E\big[ e_\mathfrak{a}\langle G^z(w_1)^2E_iG^z(w)A(w)E_j\rangle\big]\, \dif^2 w_1 \\
&\qquad=\frac{2\ii\lambda\psi_\mathfrak{a}(\lambda)}{\pi}\tilde{\sum}_{ij}\int_{\Omega_\mathfrak{a}}\partial_{\overline{w_1}}\widetilde{f}(w_1) \left\langle M^{z,z,z}(w_1,I,w_1,E_i,w) A(w)E_j\rangle\right\rangle\, \dif^2 w_1 \\
&\qquad \quad+\frac{2\ii\lambda}{\pi}\tilde{\sum}_{ij}\int_{\Omega_\mathfrak{a}}\partial_{\overline{w_1}}\widetilde{f}(w_1) \E\big[ e_\mathfrak{a}\left\langle \left(G^z(w_1)^2E_iG^z(w)-M^{z,z,z}(w_1,I,w_1,E_i,w) A(w)E_j\rangle\right)A(w)E_j\right\rangle\big]\, \dif^2 w_1.
\end{split}
\end{equation}
Here $\tilde{\sum}_{ij}$ is defined below \eqref{eq:defE1E2}. The last line in \eqref{eq:maintermvar} can be easily seen to be lower order using the (almost) global law (recall the definition of $M^{z_1,z_2,z_3}$ from \eqref{eq:defm123}, cf. \cite[Proposition 4.1]{cipolloni2023optimal}) 
\begin{equation}
\label{eq:llawtd}
\E\big|\left\langle \big(G^{z_1}(w_1)B_1G^{z_2}(w_2)B_2G^{z_3}(w_3)-M^{z_1,z_2,z_3}(w_1,B_1,w_2,B_2,w_3)\big)B_3\right\rangle\big|^2\lesssim \frac{n^{5\mathfrak{a}}}{n}
\end{equation}
for deterministic $\lVert B_i\rVert\lesssim 1$. Note that we need \eqref{eq:llawtd} only in second moment sense. The proof of \eqref{eq:llawtd} is presented after the conclusion of the proof of this lemma. This concludes the computation of the first term in the RHS of \eqref{eq:explsteinvar}. Next, we compute
\[
2\sum_{ab}\E[\partial_{ab} e_\mathfrak{a}\langle \Delta^{ab}G A\rangle]=\frac{2\ii\lambda}{\pi}\tilde{\sum}_{ij}\int_{\Omega_\mathfrak{a}}\partial_{\overline{w_1}}\widetilde{f}(w_1) \E\big[e_\mathfrak{a}\langle G^z(w_1)^2E_iA(w)^{\mathfrak{t}}G^{\overline{z}}(w)E_j\rangle\big]\, \dif^2 w_1.
\]
Proceeding as in \eqref{eq:maintermvar} and using again \eqref{eq:llawtd}, we obtain the second term in the RHS of \eqref{eq:explsteinvar}.

Finally, we conclude the proof of this lemma by computing the order one term coming from $k=3$. In the second line of \eqref{eq:writesum}, we notice that for $k=3$ the only order one term is given by (all the other terms a lower order)
\begin{equation}
\begin{split}
&\frac{\kappa_4}{n^2}\sum_{ab}\E [(\partial_{ab}\partial_{ba} e_\mathfrak{a}) G_{aa}(GA)_{bb}] \\
&\qquad=\frac{\ii \kappa_4\lambda}{n^3\pi}\sum_{ab}\int_{\Omega_\mathfrak{a}}\partial_{\overline{w_1}} \widetilde{f}(w_1) \E\big[e_\mathfrak{a}(G^z(w_1))_{aa}(G^z(w_1)^2)_{bb} (G^z(w))_{aa}(G^z(w)A(w))_{bb}\big]\, \dif ^2 w_1 \\
&\qquad=\frac{2\ii\lambda\kappa_4 \psi_\mathfrak{a}(\lambda)}{\pi}\int_{\Omega_\mathfrak{a}}\partial_{\overline{w_1}} \widetilde{f}(w_1) m^z(w_1)(m^z(w_1))' m^z(w) (m^z(w))'\, \dif ^2 w_1+\mathcal{O}(n^{-1/2+\xi+k\gamma}).
\end{split}
\end{equation}
We also point out that in the last equality we used
\[
(M^z(w)A(w))_{bb}=\langle M^z(w)A(w)\rangle=\langle \mathcal{B}^{-1}[M^z(w)M^z(w)]\rangle=\langle \partial_w M^z(w)\rangle=(m^z(w))'.
\]

\qed

\end{proof}

\

\begin{proof}[Proof of \eqref{eq:llawtd}] \, By \cite[Theorem 5.2]{cipolloni2023central} we have
\begin{equation}
\label{eq:2gllaw}
\big|\left\langle \big(G^{z_1}(w_1)B_1G^{z_2}(w_2))-M^{z_1,z_2}(w_1,B_1,w_2)\big)B_2\right\rangle\big|\lesssim \frac{n^{3\mathfrak{a}}}{n}\lVert B_1\rVert \lVert B_2\rVert,
\end{equation}
with overwhelming probability, with $M^{z_1,z_2}$ being defined as in \eqref{eq:defm123}. We now show that, using \eqref{eq:goodll} and \eqref{eq:2gllaw} as an input, we can easily conclude \eqref{eq:llawtd}. Note that there is no assumption on the sign of $\Im w_1\Im w_2$ in \eqref{eq:2gllaw}.

In the reminder of the proof we use the short--hand notations $M_i:=M^{z_i}(w_i)$, $G_i:=G^{z_i}(w_i)$, $M_{ij}^B:=M^{z_i,z_j}(w_i,B,w_j)$, and $\mathcal{B}_{ij}[\cdot]:=1-M_i\mathcal{S}[\cdot]M_j$. Using \eqref{eq:equation} for $G_1$, we have
\begin{equation}
\begin{split} \label{eqn:a-clt-2}
\mathcal{B}_{13}[G_1B_1G_2B_2G_3]&=M_1B_1M_{23}^{B_2}+M_1\mathcal{S}[M_{12}^{B_1}]M_{23}^{B_2}-M_1\underline{WG_1B_1G_2B_2G_3}+M_1B_1(G_2B_2G_3-M_{23}^{B_2}) \\
&\quad+M_1\mathcal{S}[G_1B_1G_2-M_{12}^{B_1}]\big(M_{23}^{B_2}+G_2B_2G_3-M_{23}^{B_2}\big)+M_1\mathcal{S}[M_{12}^{B_1}]\big(G_2B_2G_3-M_{23}^{B_2}\big) \\
&\quad+M_1\mathcal{S}[G_1B_1G_2B_2G_3](G_3-M_3)+\langle G_1-M_1\rangle M_1G_1B_1G_2B_2G_3,
\end{split}
\end{equation}
where we defined
\begin{equation}
\label{eq:underline3}
\underline{WG_1B_1G_2B_2G_3}:=\underline{WG_1}B_1G_2B_2G_3+\mathcal{S}[G_1B_1G_2]G_2B_2G_3+\mathcal{S}[G_1B_1G_2B_2G_3]G_3.
\end{equation}

From now on we assume that $\lVert B_i\rVert\lesssim 1$. Note that all the random terms in the RHS of \eqref{eqn:a-clt-2}, except for the underline term, can be estimated relying on the single resolvent local law \eqref{eqn:ll-1} or on the two--resolvent local law \eqref{eq:2gllaw}. We now present the estimate of two representative terms, all other terms can be estimated analogously. We notice that for any matrices $R_1,R_2$ we have
\[
\langle \mathcal{B}_{13}^{-1}[R_1]R_2\rangle=\langle R_1 \big((\mathcal{B}_{13}^{-1})^*[R_2^*]\big)^*\rangle.
\]
Denote
\[
C:=\big[(\mathcal{B}_{13}^{-1})^*[B_3^*]\big]^*;
\]
and notice that by $\lVert \mathcal{B}_{13}^{-1}\rVert+\lVert (\mathcal{B}_{13}^{-1})^*\rVert\lesssim \min_i 1/|\Im w_i|\le n^\mathfrak{a}$ it follows $\lVert C\rVert\lesssim n^{\mathfrak{a}}$. Then, using \eqref{eqn:ll-1} together with a Schwarz inequality, we estimate
\begin{equation}
\big|\langle G_1-M_1\rangle\langle M_1G_1B_1G_2B_2G_3C\rangle\big|\lesssim \frac{n^\xi\lVert C\rVert}{n|\Im w_1\Im w_2\Im w_3|}\le \frac{n^{5\mathfrak{a}}}{n}.
\end{equation}
Additionally, using \eqref{eq:2gllaw} and \eqref{eq:deb12mhop}, we estimate
\begin{equation}
\big|\langle M_1\mathcal{S}[M_{12}^{B_1}](G_2B_2G_3-M_{23}^{B_2})C\rangle\big|=2\left|\tilde{\sum}_{ij}\langle M_{12}^{B_1}E_i\rangle\langle M_1 E_j(G_2B_2G_3-M_{23}^{B_2})C\rangle\right|\lesssim \frac{n^{5\mathfrak{a}}}{n}.
\end{equation}
Putting these estimates together, we thus conclude
\begin{equation}
\begin{split}
\label{eq:goodp3gs}
\langle G_1B_1G_2B_2G_3B_3\rangle&=\left\langle \mathcal{B}_{13}^{-1}\big[M_1B_1M_{23}^{B_2}+M_1\mathcal{S}[M_{12}^{B_1}]M_{23}^{B_2}\big]B_3\right\rangle \\
&\quad-\langle M_1\underline{WG_1B_1G_2B_2G_3}C\rangle +\mathcal{O}\left(\frac{n^{6\mathfrak{a}}}{n}\right),
\end{split}
\end{equation}
with overwhelming probability.

Finally, using cumulant expansion, proceeding similarly to \cite[Eq. (6.32)]{cipolloni2023central} (recall that $|\Im w_i|\ge n^{-\mathfrak{a}}$), we estimate
\[
\E\big|\langle M_1\underline{WG_1B_1G_2B_2G_3}C\rangle \big|^2\lesssim \frac{n^{4\mathfrak{a}}}{n},
\]
where we used that $\lVert C\rVert\lesssim n^\mathfrak{a}$. Recalling the definition \eqref{eq:defm123}, this concludes the proof. 

\qed

\end{proof}

\

\begin{proof}[Proof of \eqref{eq:explvarlog}] \, We start with the computation for the variance $V(f)$. Recall the definition of $Z$ from \eqref{eq:defrvz}. By \eqref{eq:usvarrel}, to prove  \eqref{eq:explvarlog}, it is enough to compute $\mathrm{Var}(Z)$ when $f(x)=\Re\log (x-\ii\eta)$, with $\eta=n^{-\gamma}$.

Using \eqref{eq:boundftilde} and the local law \eqref{eq:goodll}, it is easy to see that
\begin{equation}
Z=\mathrm{Tr} f(H^z)-\E \mathrm{Tr} f(H^z)+\mathcal{O}(n^{-1/2})=\frac{1}{2}\sum_i \log\big[(\lambda_i^z)^2+\eta^2\big]-\E(\dots)+\mathcal{O}(n^{-1/2}).
\end{equation}
We now write $Z=Z_1+Z_2$ (neglecting the negligible error term $n^{-1/2}$), with
\begin{equation}
\begin{split}
Z_1:&=-\int_{\eta}^1 \mathrm{Tr}[\Im G^z(\ii \tau)-\E \Im G^z(\ii \tau)]\, \dif \tau \\
Z_2:&=\frac{1}{2}\sum_i\log\big[(\lambda_i^z)^2+1\big]-\E \frac{1}{2}\sum_i\log\big[(\lambda_i^z)^2+1\big].
\end{split}
\end{equation}
We use the following lemma (whose proof is presented at the end of this section) to estimate $\mathrm{Var}(Z_2)$:
\begin{lemma}
\label{lem:estvarO1}
There exists $C>0$ such that 
\begin{equation}
\label{eq:bvarO1}
\mathrm{Var}(Z_2)\le C.
\end{equation}
\end{lemma}
Then, using \eqref{eq:bvarO1}, we have
\begin{equation}
\label{eq:concleq}
V(f)=\mathrm{Var}(Z_1)+\mathcal{O}\left(\mathrm{Var}(Z_1)^{1/2}+n^{-1/5}\right).
\end{equation}
We are thus left only with the computation of $\mathrm{Var}(Z_1)$. Using \cite[Proposition 3.3]{cipolloni2023central}, \cite[Proposition 3.3]{cipolloni2021fluctuation} for the complex and real case, respectively, we get
\begin{equation}
\mathrm{Var}(Z_1)=4\int_\eta^1\int_\eta^1 \big[\widehat{V}(z,\tau_1,\tau_2)+\kappa_4 U(z,\tau_1)U(z,\tau_2)\big] \,\dif \tau_1\tau_2+\mathcal{O}\left(\frac{n^{6\mathfrak{a}}}{\sqrt{n}}\right).
\end{equation}
Here, using the notations $m_i:=m^{z_i}(\ii\tau_1)$, $u_i:=u^{z_i}(\ii\tau_i)$, we defined  $U(z_i,\tau_i):=\ii \partial_{\tau_i}m_i/\sqrt{2}$, and $\widehat{V}(z,\tau_1,\tau_2)=V(z,z,\tau_1,\tau_2)+\bm1_{\{\beta=1\}}V(z,\overline{z},\tau_1,\tau_2)$, with
\begin{equation}
V(z_1,z_2,\tau_1,\tau_2):=-\frac{1}{2}\partial_{\tau_1}\partial_{\tau_2}\log\big[1+|z_1z_2|^2u_1^2u_2^2-m_1^2m_2^2-2u_1u_2\Re[z_1\overline{z_2}]\big].
\end{equation}
We point out that in \cite[Proposition 3.3]{cipolloni2023central} and \cite[Proposition 3.3]{cipolloni2021fluctuation} it was assumed that $z_1\ne z_2$, $z_1\ne \overline{z_2}$, however, inspecting the proof of these propositions it is clear that this assumption is not needed for $\eta_i\ge n^{-\mathfrak{a}}$ (see also \cite[Remark 3.4]{cipolloni2023central}).

Then, using that both $U$ and $\widehat V$ are complete derivatives, we then compute (for simplicity we only consider the case $\beta=2$)
\begin{equation}
\begin{split}
\mathrm{Var}(Z_1)&=-\log\big[1+|z|^4(u^z(\ii\eta))^4-(m^z(\ii\eta))^4-2(u^z(\ii\eta))^2|z|^2\big] \\
&\quad+2\log\big[1+|z|^4(u^z(\ii\eta))^2(u^z(\ii)^2-(m^z(\ii\eta))^2(m^z(\ii))^2-2u^z(\ii\eta)u^z(\ii)|z|^2\big]+\mathcal{O}(1).
\end{split}
\end{equation}
Finally, using the analog of \eqref{eq:exapnsmu} for $m^z(\ii\eta), u^z(\ii\eta)$, we conclude
\begin{equation}
\mathrm{Var}(Z_1)=-\log \eta+\mathcal{O}(1).
\end{equation}
Similarly, in the real case we get
\[
\mathrm{Var}(Z_1)=-\log \eta-\log[|z-\overline{z}|^2+\eta]+\mathcal{O}(1)
\]
Combining this with \eqref{eq:concleq} we conclude the computation of $V(f)$.

We now turn to the computation of $\E \mathrm{Tr} f(H^z)$. We write (recall that $\eta=n^{-\gamma}$)
\begin{equation}
\begin{split}
\label{eq:precisecompexp}
\E \mathrm{Tr} f(H^z)&=\frac{1}{2}\E\sum_i\log\big[(\lambda_i^z)^2+\eta^2] \\
&=\frac{1}{2}\E\sum_i\log\big[(\lambda_i^z)^2+1]-\int_\eta^1 \E \mathrm{Tr}[\Im G^z(\ii\tau)]\, \dif \tau.
\end{split}
\end{equation}
Denote $m:=m^z(\ii\tau)$ and $u:=u^z(\ii\tau)$. In order to compute the expectation of the second term in the RHS of \eqref{eq:precisecompexp} we rely on \cite[Eqs. (3.11)--(3.12)]{cipolloni2023central}:
\begin{equation}
\label{eq:explexp}
\E\mathrm{Tr}[\Im G(\ii\eta)]=-2\ii n m-\frac{\kappa_4}{2}\partial_\tau m^4+\frac{\bm1_{\{\beta=1\}}}{2}\partial_\tau\log\big(1-u^2+2u^3|z|^2-u^2(z^2+\overline{z}^2)\big)+\mathcal{O}\left(\frac{n^{2\gamma}}{\sqrt{n}}\right).
\end{equation}
We point out that in \cite[Eqs. (3.11)--(3.12)]{cipolloni2023central} the error term deteriorates with $|\Im z|^2$; however, inspecting its proof, it is clear that every instance of $1/|\Im z|^2$ can be replaced by $1/\eta$, giving \eqref{eq:explexp}.

We now use
\begin{equation}
\label{eq:globlogllaw}
\frac{1}{2}\E\sum_i\log\big[(\lambda_i^z)^2+1]=n\int \log(x^2+1)\rho^z(x)\,\dif x+\O(1).
\end{equation}
This easily follows by an application of Helffer--Sj\"ostrand formula together with \eqref{eq:explexp} for $\gamma=0$ (see the proof of Lemma~\ref{lem:estvarO1} for similar computations). Plugging \eqref{eq:globlogllaw} into \eqref{eq:precisecompexp} and using \eqref{eq:explexp}, we obtain
\begin{equation}
\label{eq:fincomp}
\begin{split}
\E \mathrm{Tr} f(H^z)&= n\int \log(x^2+1)\rho^z(x)\,\dif x-2\ii n\int_\eta^1 m^z(\ii\tau)\, \dif \tau\\
&\quad+\frac{\bm1_{\{\beta=1\}}}{2}\log\big(1-u^z(\ii\eta)^2+2u^z(\ii\eta)^3|z|^2-u^z(\ii\eta)^2(z^2+\overline{z}^2)\big)+\O(1).
\end{split}
\end{equation}
Finally, using that
\[
\partial_\tau \int\log(x^2+\tau^2)\rho^z= -2\ii m^z(\ii\tau),
\]
and the analog of \eqref{eq:exapnsmu} for $m^z(\ii\eta), u^z(\ii\eta)$ in \eqref{eq:explexp}, \eqref{eq:fincomp} concludes the computation of the expectation of $f(H^z)$.

\qed

\end{proof}

\

\begin{proof}[Proof of Lemma~\ref{lem:estvarO1}] \, Consider $f(x):=\log(x^2+1)$, then, using \eqref{eq:defvfusef} and \eqref{eq:usvarrel} (applied to this function), we write 
\begin{equation}
\label{eq:explvarz2d}
\mathrm{Var}(Z_2)=V(f)=\frac{1}{\pi^2}\int_{\Omega_\mathfrak{a}} \int_{\Omega_\mathfrak{a}} \partial_{\overline{z}}\widetilde{f}(z) \partial_{\overline{w}}\widetilde{f}(w) R(z,w)\, \dif z^2\dif^2 w+\mathcal{O}(n^{-1/4+200\gamma}),
\end{equation}
with $\mathfrak{a}>0$ arbitrary small, $\Omega_\mathfrak{a}$ from \eqref{eq:defomegafraka}, and $R(z,w)$ being defined in \eqref{eq:explsteinvar}. Next, by \eqref{eq:boundftilde} for $\gamma=0$, we have
\begin{equation}
\label{eq:newbtildef}
\big|\partial_{\overline{z}}\widetilde{f}(z)\big|\lesssim |\Im z|^{k-1},
\end{equation}
for any $k\in \N$. Additionally, by the definition of $R(z,w)$ in \eqref{eq:explsteinvar} (see also \eqref{eq:defaw}--\eqref{eq:defm123}) together with \eqref{eq:deb12mhop} for $z_1=z_2=z$, we also have
\begin{equation}
\label{eq:boundR}
\big|R(z,w)\big|\lesssim \frac{1}{(|\Im z|+|\Im w|)^3}.
\end{equation}
Plugging \eqref{eq:newbtildef}--\eqref{eq:boundR} into \eqref{eq:explvarz2d} we conclude the desired bound.

\qed

\end{proof}

\section{Green's function comparison}

\subsection{Proof of Proposition \ref{prop:ll-gfct}} \label{a:ll-gfct}

The proof given here is similar in spirit to \cite[Section 2.3]{erdHos2023small}. 
Suppose first that $X$ and $Y$ differ only in one matrix entry, say the $(1, 1)$--th. Let $W(\theta)$ be the matrix with this entry set to $\theta$ so that $X = W(X_{11})$ and $Y = W(Y_{11})$. Let $F : \cc \to \rr$ be a smooth function so that 
\beq
F(z) =1, \quad |z| > (k+3/4) (\log n)^{1/2+\delta}, \qquad F(z) = 0, \quad |z| < (k+1/2) (\log n)^{1/2+\delta}
\eeq
Using the local law of Theorem \ref{thm:ll} and a resolvent expansion it is not hard to check that for any $\frac{1}{10} > \eps >0$ we have, with overwhelming probability
\beq
\left| \sup_{ |\theta| \leq n^{\eps/2-1/2} } \partial_{\theta}^a \partial_{\bar{\theta}}^b Z(W(\theta) )\right| \leq n^{ \eps}.
\eeq
for all $0 \leq a + b \leq 5$, $a, b \geq 0$. Specifically, the derivatives appearing above will involve products of resolvent entries of $W$. If $\theta$ were a random variable, the bounds for these entries would follow from \eqref{eqn:entrywise-ll}. To deal with the $\sup$ over $\theta$, one can apply a resolvent expansion similar to (16.8) of \cite{erdHos2017dynamical}.

By Taylor expansion to fifth order we then see that,
\beq
\left| \ee[ F ( X)] - \ee[ F (Y) ] \right| \leq (T n^{-2} + n^{-5/2+\eps} ) \ee[ \sup_{ |\theta| \leq n^{\eps/2-1/2}, 1 \leq a + b \leq 5 } \partial_{\theta}^a \partial_{\theta}^b F( Z ( W (\theta ) ) ) ] + n^{-D}.
\eeq
The derivatives of $F(z)$ are non-zero only if $|z| > (k+1/2) (\log n)^{1/2+\delta}$. Moreover, one can check by resolvent expansion that
\beq
\sup_{ | \theta| \leq n^{\eps-1/2} } \left| Z( W ( \theta) ) - Z(W(X_{11} )) \right| \leq n^{-1/4}
\eeq
with overwhelming probability. Therefore,
\beq
\left| \ee[ F ( X)] - \ee[ F (Y) ] \right| \leq (T n^{-2} + n^{-5/2+\eps} ) n^{\eps} p(k) + n^{-D}.
\eeq
In the general case where $X$ and $Y$ differ in all entries but the moments match, we follow the Lindeberg replacement strategy by replacing the matrix elements of $X$ by those of $Y$ one at a time; the above illustrates one step of the procedure (see e.g. the fourn moment method of Tao--Vu \cite{tao2011random} or \cite[Chapter 16]{erdHos2017dynamical} for a pedagogical introduction).  At each of the $n^2$ steps we apply the above inequality to conclude that
\beq
\left| \ee[ F(X)] - \ee[ F(Y) ] \right| \leq T^{1/2} p(k) + n^{-D}.
\eeq
We may also apply the above estimate with $X = W^{(ab)}$, i.e., at one of the intermediate steps. The claim now follows. \qed

\subsection{ Proof of Lemma \ref{lem:upper-compare-1} and Proposition~\ref{pro:GFTmax}}
\label{app:GFTmax}

We prove only Proposition \ref{pro:GFTmax}, as the proof of Lemma \ref{lem:upper-compare-1} is easier. 
The proof of this proposition will go through the standard Lindeberg replacement strategy of replacing the matrix elements of one matrix by another one by one and estimating the difference at each step (see above). First, we will need to regularize the max, using a similar strategy to \cite{landon2020comparison}. We notice that for any $B >0$ we have
\beq
\max_{ z \in \mathcal{P}_2} Q(z) \XX_n(z) \leq \frac{1}{B} \log\left( \sum_{z \in \mathcal{P}_2} e^{ B Q(z) \XX_n(z) } \right) \leq \max_{ z \in \mathcal{P}_2} Q(z) \XX_n(z) + \frac{\log | \mathcal{P}_2 |}{B},
\eeq
where we abbreviated $Q(z) \XX_n(z) :=Q ( n \eta_4 \Im \langle G^z( \ii \eta_4 )\rangle ) ( \Psi_n (z, \eta_3 ) - \Psi_n (z, n^{-\mfb} ) ) $. 
We take $B= n^\mfa$ for some fixed $\mfa >0$. Let
\beq
\hat{\Xi} :=  \frac{1}{B} \log\left( \sum_{z \in \mathcal{P}_2} e^{ B \beta Q(z) \XX_n(z) } \right),
\eeq
so that
\beq
| \E[ F ( \Xi ) ] - \E[ F ( \hat{\Xi} )  ] | \leq C \| F \|_{C^1} n^{-\mfa} \log n
\eeq
Let $\del_{ij}$ be directional derivatives with respect to the matrix elements. Let $W (\theta)$ be a real or complex i.i.d. matrix with the $(a, b)$--th entry set to $\theta$. It is not hard to check that for any sufficiently small $\eps, \xi>0$
\beq \label{eqn:a-resolvent}
\sup_{i, j, |\theta| \leq n^{1/2+\eps}, (\log n)^{-C} n^{-1} \leq \eta } \left| (H^z_\theta - \i \eta )^{-1}_{ij} \right| \leq n^{\xi}
\eeq
with overwhelming probability. Here, $H^z_\theta$ is the Hermitization of $W(\theta)$. Define,
\beq
M_\theta := \max_{ z \in \mathcal{P}_2, 0 \leq a+b \leq 5 } | \del_{ij}^a \del_{ji}^{b} (n^{\mfa} Q(z) \XX_n(z)) | \bigg\vert_{W(\theta)}
\eeq
The derivatives of the above quantity can be written in terms of the resolvent entries of $W(\theta)$. Therefore, using \eqref{eqn:a-resolvent} one finds that
\beq
\sup_{ | \theta | \leq n^{-1/2+\eps}} |M_\theta | \leq n^{\xi +  \mfa}
\eeq
with overwhelming probability and that $|M| \leq n^{10}$ almost surely. By direct calculation (see Lemma \ref{lem:general-der-bd} below), we find that
\beq
| \del_{ij}^{a} \del_{ij}^b\hat{\Xi} | \leq C_k  M_\theta^k
\eeq
for $a+ b \leq k$. With this as input, the proof of the proposition follows from applying the standard Lindeberg replacement strategy as indicated above, after taking $\mfa >0$ sufficiently small. Again, we refer the reader to \cite{erdHos2017dynamical} for a pedagogical introduction. \qed

\bel \label{lem:general-der-bd}
Let $Z = B^{-1} \log \left( \sum_z e^{ B \XX_z } \right)$ where $\XX_z$ are some real valued functions on the space of matrices. Then for any $k$ there is a constant $C_k>0$ so that,
\beq
\left| \del_{ij}^k Z (X) \right| \leq C_k B^{k-1} \max_{ z, 0 \leq a \leq k } | \del_{ij}^a \XX_z |^k
\eeq
\eel
\proof This follows by a straightforward and direct calculation, which is outlined in, e.g.,  the proof of Lemma 3.4 of \cite{landon2020comparison}.  For reader convenience, we present the proof for the first derivative, as the general case is not much harder:
\begin{align}
| \del_{ij} Z |& =  B \left| \frac{ \sum_{z} e^{ B \XX_z } \del_{ij} \XX_z }{ \sum_z e^{ B \XX_z } } \right| \leq  \sup_z | \del_{ij} \XX_z |
\end{align}
where we used that $\XX_z$ is real-valued so that $e^{ B \XX_z}$ is positive. \qed

\bibliography{char_bib}
\bibliographystyle{abbrv}

\end{document}